\newcommand{\pointsize}{11pt}

\documentclass[oneside, \pointsize]{amsbook}

\usepackage{amsmath,amssymb,subfigure,caption}

\usepackage[
   includehead,
   includefoot,
     left = 1in, 
      top = 1in, 
    right = 1in,
   bottom = 1in
]{geometry}
\usepackage{fancyhdr}
\usepackage{setspace}
\usepackage{calc}
\usepackage[nocompress]{cite} 
\usepackage[pdfborder={0 0 0}, pdfpagemode=UseNone, pdfstartview=FitH]{hyperref} 

\fancyheadoffset[R]{0.5in} 

\fancypagestyle{prelim}{%
   \renewcommand{\headrulewidth}{0pt} 
   \fancyhf{}           
   \pagenumbering{roman}    
   \cfoot{-\thepage-}       
}

\fancypagestyle{maintext}{%
   \renewcommand{\headrulewidth}{0.4pt}
   \pagenumbering{arabic}
   \fancyhf{}
   \fancyhead[L]{\rightmark}
   \cfoot{\thepage}
}

\fancypagestyle{abstract1}{%
   \renewcommand{\headrulewidth}{0.4pt}
	 \fancyheadoffset[R]{0in}
   \pagenumbering{arabic}
   \fancyhf{}
}

\fancypagestyle{abstract2}{%
   \renewcommand{\headrulewidth}{0.4pt}
   \fancyheadoffset[R]{0in}
	 \pagenumbering{arabic}
   \fancyhf{}
   \rhead{-\thepage-}
}

\numberwithin{figure}{chapter} 
\numberwithin{table}{chapter}
\numberwithin{equation}{chapter}
\numberwithin{section}{chapter}

\newtheorem{thm}{Theorem}[section]
\newtheorem{cor}[thm]{Corollary}
\newtheorem{lemma}[thm]{Lemma}

\newtheorem{prop}[thm]{Proposition}

\newtheorem{mthm}{Theorem}

\usepackage{tikz}
\usetikzlibrary{calc,decorations.markings,decorations.pathreplacing,arrows}
\colorlet{darkred}{red!75!black}
\colorlet{darkblue}{blue!60!black}
\colorlet{medgreen}{green!70!black!70!white}

\tikzset{
mybrace/.style={decorate,decoration={brace,amplitude=10pt,aspect=#1}}
}

\newcommand{\ie}{\emph{i.e.}}
\newcommand{\Cf}{\emph{Cf.}}

\newcommand{\BPP}{\mathsf{BPP}}
\newcommand{\FP}{\mathsf{FP}}
\newcommand{\NPC}{\mathsf{NPC}}
\newcommand{\NP}{\mathsf{NP}}
\newcommand{\ccC}{\mathsf{C}}
\newcommand{\ccP}{\mathsf{P}}
\newcommand{\shPC}{\mathsf{\#PC}}
\newcommand{\shP}{\mathsf{\#P}}

\newcommand{\CSAT}{\mathsf{CSAT}}
\newcommand{\RSAT}{\mathsf{RSAT}}

\newcommand{\ZSAT}{\mathsf{ZSAT}}
\newcommand{\shCSAT}{\mathsf{\#CSAT}}
\newcommand{\shRSAT}{\mathsf{\#RSAT}}
\newcommand{\shZSAT}{\mathsf{\#ZSAT}}

\newcommand{\AD}{\operatorname{AD}}
\newcommand{\Alt}{\operatorname{Alt}}
\newcommand{\Aut}{\operatorname{Aut}}

\newcommand{\Inn}{\operatorname{Inn}}
\newcommand{\MCG}{\operatorname{MCG}}

\newcommand{\Out}{\operatorname{Out}}
\newcommand{\PSp}{\operatorname{PSp}}
\newcommand{\Rub}{\operatorname{Rub}}
\newcommand{\SU}{\operatorname{SU}}
\newcommand{\Sp}{\operatorname{Sp}}

\newcommand{\Sym}{\operatorname{Sym}}
\newcommand{\Tor}{\operatorname{Tor}}

\newcommand{\ab}{{\operatorname{ab}}}
\newcommand{\bd}{\operatorname{bd}}

\newcommand{\lk}{\operatorname{lk}}
\newcommand{\per}{{\operatorname{per}}}
\newcommand{\rep}{\operatorname{rep}}
\newcommand{\sch}{\operatorname{sch}}
\newcommand{\inv}{\operatorname{inv}}
\renewcommand{\wr}{\operatorname{wr}}

\newcommand{\onto}{\twoheadrightarrow}
\newcommand{\into}{\hookrightarrow}

\renewcommand{\setminus}{\smallsetminus}
\newcommand{\normaleq}{\unlhd}
\newcommand{\normal}{\lhd}

\newcommand{\cC}{\mathcal{C}}

\newcommand{\yes}{\mathrm{yes}}
\newcommand{\no}{\mathrm{no}}
\newcommand{\ceil}[1]{\lceil #1 \rceil}

\newcommand{\Z}{\mathbb{Z}}
\newcommand{\R}{\mathbb{R}}
\newcommand{\N}{\mathbb{N}}
\newcommand{\C}{\mathbb{C}}

\newcommand{\AND}{\mathrm{AND}}
\newcommand{\SWAP}{\mathrm{SWAP}}
\newcommand{\NOT}{\mathrm{NOT}}
\newcommand{\OR}{\mathrm{OR}}
\newcommand{\COPY}{\mathrm{COPY}}
\newcommand{\CNOT}{\mathrm{CNOT}}
\newcommand{\CCNOT}{\mathrm{CCNOT}}

\newcommand{\hR}{\hat{R}}

\newcommand{\tM}{\tilde{M}}

\newcommand{\cA}{\mathcal{A}}
\newcommand{\cB}{\mathcal{B}}

\newcommand{\vsubseteq}{\rotatebox{90}{$\subseteq$}}
\newcommand{\veq}{\rotatebox{90}{$=$}}
\newcommand{\vonto}{\rotatebox{270}{$\onto$}}
\newcommand{\defeq}{\stackrel{\mathrm{def}}=}

\newcommand{\Cor}[1]{Corollary~\ref{#1}}
\newcommand{\Fig}[1]{Figure~\ref{#1}}
\newcommand{\Lem}[1]{Lemma~\ref{#1}}
\newcommand{\Prop}[1]{Proposition~\ref{#1}}
\newcommand{\Sec}[1]{Section~\ref{#1}}
\newcommand{\Thm}[1]{Theorem~\ref{#1}}

\newtheorem{theorem}{Theorem}[section]

\begin{document}
\frontmatter

\pagestyle{prelim}
   

\fancypagestyle{plain}{\fancyhf{}\cfoot{-\thepage-}}
\begin{center}
   \null\vfill
   \textbf{%
      Computational Complexity of Enumerative 3-Manifold Invariants
   }%
   \\
   \bigskip
   By \\
   \bigskip
   ERIC GRIFFIN SAMPERTON \\
   \bigskip
   B.S. (California Institute of Technology) 2012 \\
   \bigskip
   DISSERTATION \\
   \bigskip
   Submitted in partial satisfaction of the requirements for the
   degree of \\
   \bigskip
   DOCTOR OF PHILOSOPHY \\
   \bigskip
   in \\
   \bigskip
   MATHEMATICS \\
   \bigskip
   to the \\
   \bigskip
   OFFICE OF GRADUATE STUDIES \\
   \bigskip        
   of the \\
   \bigskip
   UNIVERSITY OF CALIFORNIA \\
   \bigskip
   DAVIS \\
   \bigskip
   Approved: \\
   \bigskip
   \bigskip
   \makebox[3in]{\hrulefill} \\
   Greg Kuperberg (Chair) \\
   \bigskip
   \bigskip
   \makebox[3in]{\hrulefill} \\
   Joel Hass \\
   \bigskip
   \bigskip
   \makebox[3in]{\hrulefill} \\
   Michael Kapovich \\
   \bigskip
   Committee in Charge \\
   \bigskip
   2018 \\
   \vfill
\end{center}

\newpage


\thispagestyle{plain}
\vspace*{20em}
\begin{center}
To my parents.
\end{center}
\newpage

\doublespacing

\tableofcontents
\newpage
{\singlespacing
   \begin{flushright}
      Eric Griffin Samperton \\
      June 2018 \\
      Mathematics \\
   \end{flushright}
}

\bigskip

\begin{center}
Computational Complexity of Enumerative 3-Manifold Invariants
\end{center}

\section*{Abstract}
Fix a finite group $G$.  We analyze the computational complexity of the problem of counting homomorphisms $\pi_1(X) \to G$, where $X$ is a topological space treated as computational input.  We are especially interested in requiring $G$ to be a fixed, finite, nonabelian, simple group.  We then consider two cases: when the input $X=M$ is a closed, triangulated 3-manifold, and when $X=S^3 \setminus K$ is the complement of a knot (presented as a diagram) in $S^3$.  We prove complexity theoretic hardness results in both settings.  When $M$ is closed, we show that counting homomorphisms $\pi_1(M) \to G$ (up to automorphisms of $G$) is $\shP$-complete via parsimonious Levin reduction---the strictest type of polynomial-time reduction.  This remains true even if we require $M$ to be an integer homology 3-sphere.  We prove an analogous result in the case that $X=S^3 \setminus K$ is the complement of a knot.

Both proofs proceed by studying the action of the pointed mapping class group $\MCG_*(\Sigma)$ on the set of homomorphisms $\{\pi_1(\Sigma) \to G\}$ for an appropriate surface $\Sigma$.  In the case where $X=M$ is closed, we take $\Sigma$ to be a closed surface with large genus.  When $X=S^3 \setminus K$ is a knot complement, we take $\Sigma$ to be a disk with many punctures.  Our constructions exhibit classical computational universality for a combinatorial topological quantum field theory associated to $G$.  Our ``topological classical computing" theorems are analogs of the famous results of Freedman, Larsen and Wang establishing the quantum universality of topological quantum computing with the Jones polynomial at a root of unity.  Instead of using quantum circuits, we develop a circuit model for classical reversible computing that is equivariant with respect to a symmetry of the computational alphabet.
\newpage
\section*{Acknowledgments}
Foremost, I would like to thank Greg Kuperberg for being a fanstastic Ph.D.\ advisor.  During my first year of grad school, I approached him with an interest in quantum topology and computational complexity.  He responded by introducing me to various fruitful questions and ideas, some of which have culminated in this dissertation six years later.  The present work is an amalgamation of two papers he and I co-wrote \cite{K:zombies,K:coloring}.  I've learned much about mathematical writing by studying Greg's changes to my first drafts.  Of course, I've also learned a lot of other things from Greg while I've been here at Davis.  I'm especially going to miss the sprawling political/historical/social/scientific/mathematical conversations at our weekly group lunch.

The Davis math department has been a great place to call home for the last six years, in large part because the faculty and staff here operate a well-run department.  Kudos to y'all.  Most importantly, I want to thank all of my family and friends, especially my parents Amy and Mike, my brother Kyle, my sister-in-law Ashley, their little goofball Leo, Colleen and the entire Delaney family, Danielle, George Mossessian, and Yoni Ackerman.  It is a privilege to have your love and support.

\mainmatter

\pagestyle{maintext}
   
\fancypagestyle{plain}{\renewcommand{\headrulewidth}{0pt} \fancyhf{} \cfoot{\thepage}}
   
\chapter{Introduction}
\label{ch:introduction}
This chapter introduces our main results and their proofs.  Section \ref{s:background} begins with a brief background review of our subject, with the goal of acquainting our reader with the big picture.  Subsection \ref{ss:closedresults} presents our results for closed 3-manifolds, and Subsection \ref{ss:knotresults} presents analogous results for knot complements in $S^3$.  We give a sketch of our proofs in Section \ref{s:sketch}.  A brief outline of the dissertation is provided in Subsection \ref{ss:outline}.  The chapter concludes with Section \ref{s:related}, in which we relate our work to the existing literature.

\section{Background}
\label{s:background}
Knots have frustrated people for a long time, well before the proliferation of portable headphones led to countless man-hours of tedious untying.  One of the more ancient and dramatic stories is that of the Gordian knot.  This famous knot was so complicated, a prophecy was born: whoever could unravel it was destined to rule all of Asia (modern Turkey).  One version of the legend has it that after trying but failing to untie the Gordian knot, Alexander the Great drew his sword, and cut it clear in half, thus conquering his fate just like he would conquer the entire known world.  Unfortunately, that strategy is not an option for anyone trying to keep their headphones both untangled and functional.

Another version of the legend gives us something more to aspire to: Alexander unraveled the Gordian knot simply by being more clever than everyone else (in this case, by removing the linchpin of the ox-cart yolk to which the knot was tied).  
The subject of this dissertation is computational complexity in 3-dimensional geometric topology, the goal of which is to answer the following question: to what extent is it possible to understand all 3-dimensional spaces as well as Alexander the Great understood the Gordian knot?

The mathematical origins of this question---albeit with a more precise formulation---are found in the work of Poincar\'{e} and Dehn.  Both were interested in trying to characterize the 3-sphere $S^3$ topologically, especially using the fundamental group.  While Dehn did not prove the Poincar\'{e} conjecture, he understood that combinatorial group theory is useful for studying 3-dimensional manifolds.  For instance, Dehn used group theory to show that the left- and right-handed trefoils are distinct \cite{Dehn:trefoils}.  His applications of group theory to topology provided motivation for his interest in algorithmic questions about finitely presented groups \cite{Dehn:problems}.  While the precise mathematical definition of ``algorithm" had yet to be developed (by Church, Turing, and others), Dehn was surely well aware that if the Poincar\'{e} conjecture were true, it would not be possible to actually use it for detecting $S^3$ if one did not also have a way of solving the triviality question for finitely-presented 3-manifold groups.

The first big step in understanding 3-manifolds algorithmically came from the work of Haken.  In \cite{Haken:normal}, he used normal surface theory to show that there exists an algorithm to detect the unknot.  Normal surface techniques were extended to other situations, for example, to show there exists an algorithm for determining if two Haken 3-manifolds are homeomorphic \cite{Hemion:haken}.

An important point about the aforementioned work is that none of it explicitly considered the resources required for algorithms.  History is partially to blame for this, since the modern understanding of algorithms had not yet been fully developed.  (Of course, it is clear even without modern definitions that Alexander the Great's first algorithm runs in constant time as long as one has a sword as a resource.)  In retrospect, the results of Dehn and Haken are interpreted as \emph{computability} results, as opposed to \emph{complexity} results.

One of the first complexity theoretic results in low-dimensional topology is due to Anick, who showed that computing the (rational) homotopy groups of a simply-connected 4-dimensional CW complex is $\shP$-hard \cite{Anick:homotopy}.  One caveat is that while the space is 4-dimensional (arguably a low dimension), the dimension of the homotopy groups must be large in order to have hardness.

While Welsh was perhaps the first to pose qualitative complexity theoretic questions in a strictly low-dimensional, \emph{geometric} topology setting \cite{Welsh}, the credit for proving the first theorem along these lines goes to Hass, Lagarias and Pippenger \cite{HLP:complexity}.  They showed that the problem of deciding if a diagram of a knot represents the unknot is in $\NP$.

Let us explain what we mean by qualitative complexity.  Briefly, this is the study of complexity classes defined by qualitative resources.  Whereas a ``quantitative" complexity class---such as polynomial time $\mathsf{P}$ or exponential space $\mathsf{EXPSPACE}$---is defined via an explicit quantitative bound on the resources required to solve a problem, a qualitative complexity class allows some kind of ``qualitative" assistance in the computation.  For instance, $\NP$ consists of those decision problems whose YES instances can be verified in polynomial time with some polynomial amount of advice.  The polynomial amount of advice serves as the qualitative assistance.

An important point is that qualitative complexity classes are not always realistic, in the sense that they do not necessarily model resources that one expects to have in the real world.  Nevertheless, qualitative upper and lower bounds on the complexity of a problem lead to both a more refined understanding of the structure of the problem, as well as more insight into what improvements to algorithms may or may not be possible.

Consider the case of unknot recognition.  In order to show that this problem is in $\NP$, the authors of \cite{HLP:complexity} had to show that not only does every unknot bound a disk, but that one can find a disk that is not ``too complicated."  This is an interesting topological fact in its own right.  Moreover, once one knows the problem is upper-bounded by $\NP$, a wealth of follow-up questions can be asked.  One might start by asking if the problem is $\NP$-complete.  It is widely believed that $\mathsf{P}$ does \emph{not} equal $\NP$, and so $\NP$-completeness results are interpreted as a form conditional hardness.  Thus, it would be useful to know if unknot recognition were $\NP$-complete, since then one would know not to waste their energy looking for improved algorithms, unless they believe they can amaze the world with a proof that $\mathsf{P} = \NP$.

In fact, unknot recognition was later shown to be in $\mathsf{coNP}$ \cite{K:knottedness,Lackenby:norm}.  Brassard showed that if a problem in $\mathsf{coNP} \cap \NP$ is $\NP$-complete, then the polynomial hierarchy collapses.  It is widely believed that this is not the case.  Thus, modulo a difficult complexity theoretic conjecture that has nothing to do with knots, we conclude that unknot recognition is likely not $\NP$-complete.  In particular, it is reasonable to hope to find improvements over the current state of the art.  Perhaps we can always be as clever as Alexander the Great when it comes to untangling unknots---that is, perhaps there is a polynomial-time algorithm for unknot recognition.

Since \cite{HLP:complexity}, there have been many more results concerning the qualitative complexity of problems in 3-manifold topology.  We will not attempt to review them.  Instead, we will content ourselves with another example showing why qualitative complexity is important.

Building on ideas of Kitaev \cite{Kitaev:anyons} and Freedman \cite{Freedman:field}, the results of \cite{FLW:universal,FLW:two} and \cite{FKW:simulation} can be understood as a mathematical proof of concept for topological quantum computing, and its equivalence with the standard circuit modeling of quantum computing.  In particular, the results of \cite{FKW:simulation} imply the existence of a polynomial time quantum algorithm for approximating the Jones polynomial of a knot evaluated at a root of unity (see also \cite{AJL:approx}).  However, this approximation is so bad that it does not reveal anything useful about the topology of the knot.  Furthermore, Kuperberg used the results of \cite{FLW:two} and theorems from quantum complexity theory to show that any topologically useful approximation of the Jones polynomial is $\shP$-hard \cite{K:jones}.  In other words, the very reason that the Jones polynomial is useful for building a quantum computer undermines the usefulness of quantum computers for approximating the Jones polynomial!  The metatheorem is that if quantum computers turn out to be helpful for solving problems in topology, it will not be because of topological quantum computing.

\section{Statement of results}
\label{s:results}
\subsection{Preliminaries}
\label{ss:preliminaries}
Given a finite group $G$ and a path-connected topological space $X$, let
\[ H(X,G) = \{f:\pi_1(X) \to G\} \]
be the set of homomorphisms from the fundamental group of $X$ to $G$.
Then the number $\#H(X,G) = |H(X,G)|$ is an important topological invariant
of $X$. For example, in the case that $X$ is a knot complement and $G =
\Sym(n)$ is a symmetric group, $\#H(X,G)$ was useful for compiling a table
of knots with up to 15 crossings \cite{Lickorish:intro}.  (We use both
notations $\#S$ and $|S|$ to denote the cardinality of a finite set $S$,
the former to emphasize algorithmic counting problems.)

Although these invariants can be powerful, our main results are that
they are often computationally intractable, assuming that $\ccP \neq \NP$.
We review certain considerations:

\begin{itemize}
\item We suppose that $X$ is given by either a finite triangulation or a diagram of a link in $S^3$, as reasonable standards for computational input.

\item We are interested in the case that $\#H(X,G)$ is intractable when
$G$ is fixed and $X$ is the only computational input.  We are also more
interested in $G$ per se, not its subgroups.  Thus, we seek intractability even
when $\#H(X,J)$ is as small as possible for every proper subgroup $J < G$.

\item If $G$ is abelian, then $\#H(X,G)$ is determined by the integral
homology group $H_1(X) = H_1(X;\Z)$; both of these invariants can be computed
in polynomial time (\Thm{th:homology}).  We are thus more interested in
the case that $H_1(X)$ is as trivial as possible and $G$ is perfect, in particular when $G$
is non-abelian simple.

\item If $X$ is a simplicial complex, or even an $n$-manifold with $n \ge
4$, then $\pi_1(X)$ can be any finitely presented group.   By contrast,
3-manifold groups are highly restricted.  We are more interested in the
cases where $X = M$ is either a closed 3-manifold or $X=S^3 \setminus K$ is a knot complement in $S^3$.  If in addition $H_1(M) = 0$, then $M$ is a homology 3-sphere.
\end{itemize}

\subsection{Closed 3-manifolds}
\label{ss:closedresults}
To state our first main result, we pass to the related invariant $\#Q(X,G) =
|Q(X,G)|$, where $Q(X,G)$ is the set of normal subgroups $\Gamma \normaleq
\pi_1(X)$ such that the quotient $\pi_1(X)/\Gamma$ is isomorphic to $G$.

\begin{mthm}
Let $G$ be a fixed, finite, non-abelian simple group.
If $M$ is a triangulated homology 3-sphere regarded as computational
input, then the invariant $\#Q(M,G)$ is $\shP$-complete via parsimonious
reduction.  The reduction also guarantees that $\#Q(M,J) = 0$ for any
non-trivial, proper subgroup $J < G$.
\label{th:main1} \end{mthm}

\Sec{s:classes} gives more precise definitions of the complexity theory
concepts in \Thm{th:main1}.  Briefly, a counting problem is in $\shP$
if there is a polynomial-time algorithm to verify the objects being
counted; it is $\shP$-hard if it is as hard as any counting problem in
$\shP$; and it is $\shP$-complete if it is both in $\shP$ and $\shP$-hard.
A \emph{parsimonious reduction} from a counting problem $g$ to a counting
problem $f$ (to show that $f$ is as hard as $g$) is a mapping $h$, computable
in polynomial time, such that $g(x) = f(h(x))$.  This standard of hardness
tells us not only that $\#Q(M,G)$ is computationally intractable, but
also that any partial information from it is intractable, for instance,
its parity.  (See \Thm{th:vv}.)  An even stricter standard is a \emph{Levin
reduction}, which asks for a bijection between the objects being counted
that is computable in polynomial time (in both directions).  In fact, our
proof of \Thm{th:main1} yields a Levin reduction from any problem in $\shP$
to the problem $\#Q(M,G)$.

Another point of precision is that \Thm{th:main1} casts $\#Q(M,G)$ as
a promise problem, requiring the promise that the simplicial complex
input describes a 3-manifold and more specifically a homology 3-sphere.
Since this promise can be checked in polynomial time (\Prop{p:mpromise}),
this is equivalent to a non-promise problem (since an algorithm to calculate
$\#Q(M,G)$ can reject input that does not satisfy the promise).

The invariants $\#H(X,G)$ and $\#Q(X,G)$ are related by the following
equation:
\begin{equation} |H(X,G)| = \sum_{J \leq G} |\Aut(J)|\cdot|Q(X,J)|.
\label{e:homsum}
\end{equation}
If $\pi_1(X)$ has no non-trivial surjections to any proper subgroup of $G$, as \Thm{th:main1} can provide, then
\begin{equation} |H(X,G)| = 1 + |\Aut(G)|\cdot|Q(X,G)|. \label{e:hq}\end{equation}
Thus we can say that $\#H(M,G)$ is \emph{almost parsimoniously}
$\shP$-complete for homology 3-spheres.  It is parsimonious except for the
trivial homomorphism and up to automorphisms of $G$, which are both minor,
unavoidable corrections.  This concept appears elsewhere in complexity
theory; for instance, the number of 3-colorings of a planar graph is almost
parsimoniously $\shP$-complete \cite{Barbanchon:unique}.

In particular, the fact that $\#Q(M,G)$ is parsimoniously $\shP$-hard
implies that existence is Karp $\NP$-hard (again see \Sec{s:classes}).
Thus \Thm{th:main1} has the following corollary.

\begin{cor}
Let $G$ be a fixed, finite, non-abelian simple group, and
let $M$ be a triangulated homology 3-sphere regarded as computational input.
Then it is Karp $\NP$-complete to decide whether there is a non-trivial
homomorphism $f:\pi_1(M) \to G$, even with the promise that every such
homomorphism is surjective.
\label{c:np} \end{cor}

\Cor{c:np} in turn has a corollary concerning connected covering spaces.
In the proof of the corollary and later in the dissertation, we let $\Sym(m)$
be the symmetric group and $\Alt(m)$ be the alternating group, both acting
on $m$ letters.

\begin{cor}
For each fixed $m \ge 5$, it is $\NP$-complete to decide
whether a homology 3-sphere $M$ has a connected $m$-sheeted cover, even
with the promise that it has no connected $k$-sheeted cover with $1 < k < m$.
\label{c:covers} \end{cor}

\begin{proof} Recall that $\Alt(m)$ is simple when $n \ge 5$.
The $m$-sheeted covers $\tM$ of $M$ are bijective with homomorphisms
$f:\pi_1(M) \to \Sym(m)$, considered up to conjugation in $\Sym(m)$.
If $M$ is a homology 3-sphere, then $\pi_1(M)$ is a perfect group and we can
replace $\Sym(m)$ by $\Alt(m)$. If $\tM$ is disconnected, then $f$ does not
surject onto $\Alt(m)$.  Thus, we can apply \Cor{c:np} with $G = \Alt(m)$.
\end{proof}

\subsection{Knot complements}
\label{ss:knotresults}
Our second main result is an analog of Theorem \ref{th:main1} for complements of knots in $S^3$.  In this case, we analyze coloring invariants, which are more refined invariants than $\#H(X,G)$.  Our conclusions are slightly different from Theorem \ref{th:main1} because knot complements have nontrivial homology.

Let $K$ be a diagram of an oriented knot in the 3-sphere $S^3$.  Around 1960, Ralph Fox defined the idea of a $3$-coloring of the diagram $K$: an assignment of a color $1,2,3$ to each arc in the diagram so that at every crossing, the color of the over-arc is the average (mod $3$) of the colors of the two other arcs (see \cite[Ch.~6, Exercises 6-7]{CF:knot}).  It is easy to check that the number of $3$-colorings of a diagram is invariant with respect to the Reidemeister moves, and, hence, is an isotopy invariant of the knot represented by the diagram $K$.

With some basic algebraic topology available, generalizations of this definition abound (as Fox was well aware).  Fix a finite group $G$, and a conjugacy class $C \subset G$.  For convenience, we fix a \emph{meridian}, \ie, an element $\gamma$ of the knot group $\pi_1(S^3 \setminus K)$ that is freely homotopic to a simple closed curve that winds once around $K$ in the direction determined by the right hand rule.  (For example, we can let $\gamma$ be any of the Wirtinger generators of the knot group.)  Let
\[ H(K,G,C) \defeq \{ f: \pi_1(S^3 \setminus K) \to G \mid f(\gamma) \in C \}\]
denote the set of homomorphisms from the knot group to $G$ taking $\gamma$ to $C$.  Since all meridians are conjugate in $\pi_1(S^3\setminus K)$, $H(K,G,C)$ does not depend on the choice of $\gamma$, and $\#H(K,G,C)=|H(K,G,C)|$ is an integer valued invariant of knots.

If $G=D_{6}$ is the dihedral group of order $6$, and $C$ is the conjugacy class of reflections, then $\#H(K,G,C)$ is precisely the number of Fox $3$-colorings of $K$ as defined above.  For general $G$ and $C$, we call elements of the set $H(K,G,C)$ \emph{$C$-colorings} of $K$.  The goal of our second main theorem is to show that for many choices of $G$ and $C$, counting $C$-colorings of knots is computationally intractable.  Fox observed ``\dots $A_5$ is a simple group, so that I know of no method of finding representations on $A_5$ other than just trying" \cite{F:quick}.  Theorem \ref{th:main2} makes this precise by showing the existence of an efficient method for this problem implies $\mathsf{P} = \mathsf{NP}$.  It is expected that this is not the case, and that for the most difficult problems in $\NP$ one can do no better than ``just trying."

Before stating our precise results, we make some observations about $H(K,G,C)$.  First, we can refine the coloring invariant $\#H(K,G,C)$ by keeping track of which element of $C$ the meridian $\gamma$ maps to: for $c \in C$, we let
\[ H(K,\gamma,G,c) \defeq \{ f \in H(K,G,C) \mid f(\gamma) = c \}. \]
If $\alpha: G \to G$ is any automorphism, then
\[ |H(K,\gamma,G,c)| = |H(K,\gamma,G,\alpha(c))|.\]
In particular, $|H(K,\gamma,G,c)| = |H(K,\gamma,G,c')|$ for any $c,c' \in C$, so
\[ |H(K,G,C)| = |C| \cdot |H(K,\gamma,G,c)| \]
and $\#H(K,\gamma,G,c)=|H(K,\gamma,G,c)|$ does not depend on the choice of $\gamma$.

Let $\Aut(G,c)$ denote the automorphisms of $G$ that fix $c$.  Our second set of observations comes from considering the action of $\Aut(G,c)$ on $H(K,\gamma,G,c)$.  Define
\[ Q(K,\gamma,G,c) \defeq  \{ \Gamma \lhd \pi_1(S^3 \setminus K) \mid \exists \alpha: \pi_1(S^3 \setminus K) / \Gamma \xrightarrow{\cong} G, \alpha(\gamma) = c \}. \]
Alternatively, we can equate $Q(K,\gamma,G,c)$ with the $\Aut(G,c)$-equivalence classes of surjective homomorphisms in $H(K,\gamma,G,c)$.  Thus, $\#Q(K,\gamma,G,c)=|Q(K,\gamma,G,c)|$ does not depend on the choice of $\gamma$ and is an invariant of $K$.  Since $\Aut(G,c)$ acts freely on the subset of surjections in $H(K,\gamma,G,c)$, we have
\[ |H(K,\gamma,G,c)| = \sum_{c \in J \leq G} |\Aut(J,c)| \cdot |Q(K,\gamma,J,c)|.\]
Combining our observations, we conclude
\[ |H(K,G,C)| = \sum_{c \in J \leq G} |C| \cdot |\Aut(J,c)| \cdot |Q(K,\gamma,J,c)|.\]
This formula is important for understanding the complexity of $\#H(K,G,C)$, because it shows the invariant is constrained: $\#H(K,G,C)$ is always a weighted sum of the $\#Q(K,\gamma,H,c)$ invariants, with multiplicities independent of $K$ and $\gamma$.  Our main theorem implies that if $G$ is nonabelian simple, then $\#H(K,G,C)$ is as difficult to compute as could be expected, given these constraints.

\begin{mthm}
Let $G$ be a fixed, finite, non-abelian simple group, and fix a nontrivial element $c \in G$. If $K$ is an oriented knot diagram with meridian $\gamma \in \pi_1(S^3 \setminus K)$, together regarded as computational input, then the invariant $\#Q(K,\gamma,G,c)$ is $\shP$-complete via parsimonious reduction.  The reduction also guarantees that $\#Q(K,\gamma,J,c) = 0$ whenever $J$ is a proper subgroup of $G$ distinct from the cyclic subgroup $\langle c \rangle$.
\label{th:main2}
\end{mthm}

We refer the reader to the paragraph after the statement of Theorem \ref{th:main1} for a brief review of the relevant complexity theory.  See Chapter \ref{ch:complexity} for more details.

The theorem provides knots $K$ with meridians $\gamma$ such that
\[\begin{aligned}
|H(K,G,C)| &= |C| \cdot |\Aut(\langle c\rangle,c)| \cdot 1 + |C| \cdot |\Aut(G,c)| \cdot |Q(K,\gamma,G,c)|\\
&= |C| + |C| \cdot |\Aut(G,c)| \cdot |Q(K,\gamma,G,c)|,
\end{aligned}\]
where  $C$ is the conjugacy class of the group element $c$.  The first term corresponds to $|C|$ unavoidable homomorphisms that factor through the abelianization $\pi_1(K)_{ab} \cong \mathbb{Z}$ and have a cyclic image generated by an element of $C$.  In other words, the first term counts the $|C|$ unavoidable trivial $C$-colorings.  The remaining homomorphisms are all surjective.  Thus, our reduction to $\#H(K,G,C)$ is almost parsimonious in the same way that our reduction to $\#H(M,G)$ is.

Just as we had Corollary \ref{c:np} from Theorem \ref{th:main1}, we conclude that finding nontrivial $C$-colorings is $\NP$-complete in the strictest sense:

\begin{cor}
Let $G$ be a fixed, finite, non-abelian simple group, and fix a nontrivial conjugacy class $C \subset G$.  If $K$ is an oriented knot diagram, thought of as computational input, then deciding whether $\#H(K,G,C) > \#C$ is $\NP$-complete via Karp reduction, even with the promise that every such homomorphism is surjective.
\label{c:main}
\end{cor}

We remark that there is no analog of Corollary \ref{c:covers} for knots because knot complements always have cyclic covers of every degree.

\section{Proof sketches}
\label{s:sketch}
We now sketch the proofs of Theorems \ref{th:main1} and \ref{th:main2}.  We first discuss Theorem \ref{th:main1} because it is easier.

\subsection{Theorem \ref{th:main1} proof sketch}
\label{ss:sketch1}
Let $\Sigma_g$ be a standard oriented surface of genus $g$ with a marked basepoint, and let $G$
be a (not necessarily simple) finite group.  Then we can interpret the
set of homomorphisms, or \emph{representation set}, \[ \hR_g(G) \defeq
H(\Sigma_g,G) = \{f:\pi_1(\Sigma_g) \to G\} \] as roughly the set of
states of a computer memory.  We can interpret a word in a fixed generating
set of the pointed, oriented mapping class group $\MCG_*(\Sigma_g)$ as a
reversible digital circuit acting on $\hR_g(G)$, the set of memory states.
(See Section \ref{s:circuits} and Chapter \ref{ch:CSP}  for discussion of complexity of circuits and
reversible circuits.)  Every closed, oriented 3-manifold $M$ can be
constructed as two handlebodies $(H_g)_I$ and $(H_g)_F$ that are glued
together by an element $\phi \in \MCG_*(\Sigma_g)$.  We can interpret
$\phi$ as a reversible digital circuit in which the handlebodies partially
constrain the input and output.

To understand the possible effect of $\phi$, we want to decompose $\hR_g(G)$
into $\MCG_*(\Sigma_g)$-invariant subsets.  The obvious invariant of $f
\in \hR_g(G)$ is its image $f(\pi_1(\Sigma_g)) \leq G$; to account for it,
we first restrict attention to the subset
\[ R_g(G) \defeq \{f: \pi_1(\Sigma_g) \onto G\} \subseteq \hR_g(G) \]
consisting of surjective homomorphisms.

We must also consider a less obvious invariant.  Let $BG$ be the classifying
space of $G$, and recall that the group homology $H_*(G) = H_*(G;\Z)$
can be defined as the topological homology $H_*(BG)$.  Recall that a
homomorphism $f:\pi_1(\Sigma_g) \to G$ corresponds to a map $f:\Sigma_g
\to BG$ which is unique up to pointed homotopy.  Every $f \in \hR_g(G)$
then yields a homology class
\[ \sch(f) \defeq f_*([\Sigma_g]) \in H_2(G), \]
which we call the \emph{Schur invariant} of $f$; it is
$\MCG_*(\Sigma_g)$-invariant.  Given $s \in H_2(G)$, the subset
\[ R_g^s(G) \defeq \{f \in R_g \mid \sch(f) = s\} \]
is then also $\MCG_*(\Sigma_g)$-invariant.  Note that $\sch(f)$ is not always
$\Aut(G)$-invariant because $\Aut(G)$ may act non-trivially on $H_2(G)$.
Fortunately, $R_g^0(G)$ is always $\Aut(G)$-invariant.  We summarize the
relevant results of Dunfield-Thurston in the following theorem.

\begin{theorem}[Dunfield-Thurston {\cite[Thms. 6.23 \& 7.4]{DT:random}}] Let
$G$ be a finite group.
\begin{enumerate}
\item For every sufficiently large $g$ (depending on $G$), the Schur invariant $\sch$ is a complete invariant for the orbits of the action of $\MCG_*(\Sigma_g)$ on $R_g(G)$.
\item If $G$ is non-abelian and simple, then for every sufficiently large
$g$, the image of the action of $\MCG_*(\Sigma_g)$ on $R^0_g(G)/\Aut(G)$
is $\Alt(R^0_g(G)/\Aut(G))$.
\end{enumerate}
\label{th:dt} \end{theorem}

To make effective use of \Thm{th:dt}, we strengthen its second part in
three ways to obtain \Thm{th:dtrefine}.   First, \Thm{th:dt} holds for the
pointed Torelli group $\Tor_*(\Sigma_g)$.  Second, we define an analogue
of alternating groups for $G$-sets that we call \emph{Rubik groups}, and
we establish \Thm{th:rubik}, a non-trivial structure theorem to generate a
Rubik group.  \Thm{th:dtrefine} gives a lift of the image of $\MCG_*(\Sigma_g)$
from $\Alt(R^0_g(G)/\Aut(G))$ to the Rubik group $\Rub_{\Aut(G)}(R^0_g(G))$.
Third, we still obtain the image $\Rub_{\Aut(G)}(R^0_g(G))$ even if
we restrict to the subgroup of $\Tor_*(\Sigma_g)$ that pointwise fixes
$\hR_g(G) \setminus R_g(G)$, the set of non-surjective homomorphisms.

As a warm-up for our proof of \Thm{th:main1}, we can fix $g$, and try
to interpret
\[ A = R_g^0(G)/\Aut(G) \]
as a computational alphabet.  If $g$ is large enough, then we can apply
\Thm{th:dt} to $R^0_{2g}(G)$ to obtain a universal set of reversible binary
gates that act on $A^2 \subset R^0_{2g}(G)/\Aut(G)^2$, implemented as mapping
class elements or \emph{gadgets}.  (A gadget in computational complexity is
an informal concept that refers to a combinatorial component of a complexity
reduction.)  The result can be related to a certain constraint satisfaction problem for reversible circuits $\RSAT_{A,I,F}$.  (See \Sec{s:csp}.  The $\shP$-hardness of
$\RSAT$, established in \Thm{th:rsat}, is a standard result but still takes
significant work.)  We can convert a reversible circuit of width $n$ to an
element $\phi \in \MCG_*(\Sigma_{ng})$ that acts on $A^n$, and then make
$M$ from $\phi$.  In this way, we can reduce $\shRSAT_{A,I,F}$ to $\#Q(M,G)$.

For our actual reduction, we will need to take steps to address three
issues, which correspond to the three ways that \Thm{th:dtrefine} is sharper
than \Thm{th:dt}.
\begin{itemize}
\item We want the larger calculation in $\hR_{ng}(G)$ to avoid symbols in
$\hR_g(G) \setminus R^0_g(G)$ that could contribute to $\#Q(M,G)$.

\item We want a parsimonious reduction to $\#Q(M,G)$, which means that we
must work with $R^0_g(G)$ rather than its quotient $A$.

\item Mapping class gadgets should be elements of the Torelli group,
to guarantee that $M$ is a homology 3-sphere.
\end{itemize}

To address the first issue: We can avoid states in $R^s_g(G)$ with
$s \ne 0$ because, if a surface group homomorphism $f:\pi_1(\Sigma_g)
\onto G$ has $\sch(f) \ne 0$, then it cannot extend over a handlebody.
If $f(G)$ has a non-trivial abelianization, then the fact that we will
produce a homology 3-sphere will kill its participation.  If $f$ is not
surjective but $f(G)$ is perfect, then we will handle this case by acting
trivially on $R_g(Q)$ for a simple quotient $Q$ of $f(G)$.  The trivial
homomorphism $z \in \hR_g(G)$ is particularly problematic because it cannot
be eliminated using the same techniques; we call it the \emph{zombie symbol}.
We define an ad hoc reversible circuit model, $\ZSAT$, that has zombie symbols.
We reduce $\RSAT$ to $\ZSAT$ by converting the zombie symbols to warning
symbols that do not finalize, unless all of the symbols are zombies.
The full construction, given in Lemmas \ref{l:zsat} and \ref{l:glue},
is more complicated because these steps must be implemented with binary
gates in $\MCG_*(\Sigma_{2g})$ rather than unary gates in $\MCG_*(\Sigma_g)$.

To address the second issue: A direct application of \Thm{th:dt} would
yield a factor of $|\Aut(G)|^n$ in the reduction from $\shRSAT_{A,I,F}$ to
$\#H(M,G)$, when the input is a reversible circuit of width $n$.  We want
to reduce this to a single factor of $|\Aut(G)|$ in order to construct a
parsimonious reduction to $\#Q(M,G)$.  The $\ZSAT$ model also has an action
of $J = \Aut(G)$ on its alphabet to model this.  \Lem{l:zsat} addresses
the problem by relying on the Rubik group refinement in \Thm{th:dtrefine},
and by creating more warning symbols when symbols are misaligned relative
to the group action.

To ensure that the resulting manifold is a homology 3-sphere, we
implement gates in the pointed Torelli subgroup $\Tor_*(\Sigma_g)$ of
$\MCG_*(\Sigma_g)$.  This is addressed in \Thm{th:dtrefine}.  Recall that
$\Tor_*(\Sigma_g)$ is the kernel of the surjective homomorphism
\[ f:\MCG_*(\Sigma_g) \to \Aut(H_1(\Sigma_g)) \cong \Sp(2g,\Z) \]
where $H_1(\Sigma_g)$ is equipped with its integral symplectic intersection
form.  The proof of \Thm{th:dtrefine} uses rigidity properties of $\Sp(2g,\Z)$
combined with Goursat's lemma (\Lem{l:goursat}).

\subsection{Theorem \ref{th:main2} proof sketch}
\label{ss:sketch2}
The proof of Theorem \ref{th:main2} is conceptually very similar to the Proof of Theorem \ref{th:main1} but requires some less familiar ideas in place of the Schur invariant.  In this case we consider the action of the $k$-strand braid group $B_k$ on the finite set of $G$-representations
\[ \hat{R}_k(G) \defeq \{ \pi_1(D_k) \to G \}. \]
Here $D_k$ is the $k$-punctured disk
\[ D_k \defeq D^2 \setminus \{p_1,\dots,p_k\}, \]  
where $p_1,\dots,p_k$ are $k$ distinct points in the interior of $D^2$.

By choosing simple closed loops that wind counterclockwise around the $k$ punctures in a standard way, we identify $\pi_1(D^2_k)$ with $F_k$, the free group on $k$ generators.  We further make the identification
\[ \hat{R}_k(G) = G^k.\]
In this notation, the elementary braid generator $\sigma_i$, which swaps strand $i$ over strand $i+1$, takes $(g_1,\dots,g_k)$ to
\[ (g_1,\dots, g_{i-1},g_ig_{i+1}g_i^{-1},g_i,g_{i+2},\dots,g_k). \]

In fact, for the purposes of Theorem \ref{th:main2}, we do not need to study the action of all of $B_k$ on all of $\hat{R}_k(G)$, but only of a certain subgroup $B_v$ on a certain $B_v$-invariant subset $\hat{R}_v(G)$.  See Section \ref{s:disks} for the definitions.  Associated to a homomorphism $f \in \hat{R}_v(G)$ is a $B_v$-invariant denoted $\inv_v(f)$ that plays a role analogous to the Schur invariant in the reduction to closed 3-manifolds.  In analogy with $R_g(G)$ and $R_g^0(G)$ in the previous subsection, we then define subsets $R_v(G)$ (the surjections in $\hat{R}_v(G)$) and $R_v^0(G)$ (the surjections with $\inv_v(f)=0$).  We refer the reader to Section \ref{s:disks} for details.  The importance of $\inv_v(f)$ is exhibited by the following theorem, which is an analog of Theorem \ref{th:dt} for the action of $B_v$ on $R_v(G)$.

\begin{theorem}
Let $G$ be a finite group and let $C$ be a union of conjugacy classes that generates $G$.
\begin{enumerate}
\item (Conway-Parker \cite{CP:hurwitz}, see {\cite[Prop.~4.1]{RV:hurwitz}}) For every sufficiently large $k$ (depending on $G$ and $C$), $\inv_v$ is a complete invariant for the orbits of the action of $B_v$ on $R_v(G)$.
\item (Roberts-Venkatesh {\cite[Thm.~5.1]{RV:hurwitz}}) If $G$ is non-abelian and simple, then for every sufficiently large $k$, the image of the action of $B_v$ on $R_v^0(G)/\Aut(G)$ is $\Alt(R_v^0(G)/\Aut(G))$.
\end{enumerate}
\label{th:rv}
\end{theorem}

We reprove the second part of Theorem \ref{th:rv} in Subsection \ref{ss:rvrefine}.  We then prove Theorem \ref{th:rvrefine}, which refines Theorem \ref{th:rv} in all of the same ways that Theorem \ref{th:dtrefine} refines Theorem \ref{th:dt}.  In Section \ref{s:knotreduction}, we reduce $\ZSAT$ to $\#Q(K,\gamma,G,c)$.

\subsection{Outline}
\label{ss:outline}
In Chapter \ref{ch:complexity}, we review the complexity theory necessary for understanding the results and proofs contained in this dissertation.  Section \ref{s:standard} includes some standard algorithms that supplement our main theorems.

Chapter \ref{ch:group} contains a number of results in group theory that are useful later.

Chapter \ref{ch:mcg} studies two different families of mapping class group actions in detail.  Section \ref{s:closed} analyzes the action of $\MCG_*(\Sigma_g)$ on $\hat{R}_g(G)$.  The main result is Theorem \ref{th:dtrefine}, which is a refinement of Theorem \ref{th:dt}.
Section \ref{s:disks} analyzes the action of (a certain subgroup of) $B_k$ on (a certain subset of) $\hat{R}_k(G)$.  The main result is Theorem \ref{th:rvrefine}, which refines Theorem \ref{th:rv} in the same way that Theorem \ref{th:dtrefine} refines Theorem \ref{th:dt}.

The main goal of Chapter \ref{ch:CSP} is to introduce $\#\ZSAT$.  In fact, $\#\ZSAT = \#\ZSAT_{J,A,I,F}$ depends on several parameters.  Lemma \ref{l:zsat} says that $\#\ZSAT$ is $\shP$-complete for any choice of these parameters satisfying some technical conditions.

Chapter \ref{ch:reductions} contains the proofs of the main theorems.  In Section \ref{s:homologyreduction}, we combine Lemma \ref{l:zsat} with Theorem \ref{th:dtrefine} to prove Theorem \ref{th:main1}.  In Section \ref{s:knotreduction}, we combine Lemma \ref{l:zsat} with Theorem \ref{th:rvrefine} to prove Theorem \ref{th:main2}.   We tune the parameters $J,A,I,F$ for $\ZSAT$ differently for each of our two reductions.

The final Chapter \ref{ch:discussion} discusses some ways one might try to improve or extend our results.

\begin{figure}[t]
\begin{tikzpicture}
\node (a) at (0,0) {$\#\CSAT$};
\node[above] at (1.5,0) {Sec.~\ref{s:revcircuits}};
\node (b) at (3,0) {$\#\RSAT$};
\node[above] at (4.5,0) {Sec.~\ref{s:zombies}};
\node (c) at (6,0) {$\#\ZSAT$};
\node[left] at (7,1.2) {Sec.~\ref{s:homologyreduction}};
\node (d) at (8,2) {$\#Q(M,G)$};
\node[left] at (7,-1.2) {Sec.~\ref{s:knotreduction}};
\node (e) at (8,-2) {$\#Q(K,\gamma,G,c)$};
\draw[->] (a) to (b);
\draw[->] (b) to (c);
\draw[->] (c) to (d);
\draw[->] (c) to (e);
\end{tikzpicture}
\caption{The reductions in the proofs of Theorems \ref{th:main1} and \ref{th:main2}.}
\label{f:reductions} \end{figure}
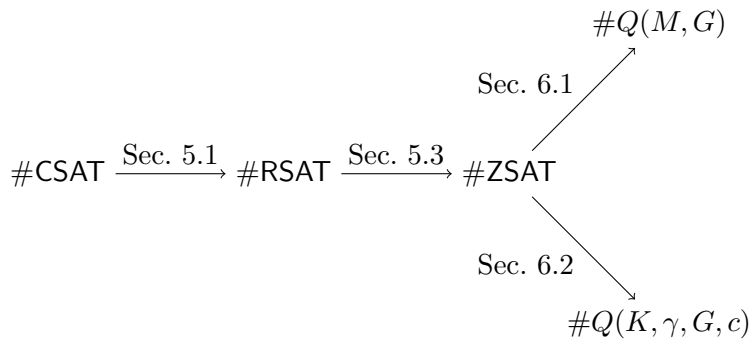


\Fig{f:reductions} summarizes the main reductions in the proofs of
Theorems \ref{th:main1} and \ref{th:main2}, and the sections in which they are constructed.

We note that, with the exception of the main results Theorems \ref{th:main1} and \ref{th:main2}, all lemmas, propositions, corollaries and theorems are numbered by section using the same counter.  Equations and figures are numbered by chapter on their own counters.

\section{Related results}
\label{s:related}
As far as we know, the closest prior result to our Theorems \ref{th:main1} and \ref{th:main2} is due to
Krovi and Russell \cite{KR:finite}.  They prove that $\#H(L,\Alt(m),C)$, considered as an invariant of \emph{links} $L$ in $S^3$ and not just knots, is $\shP$-complete when $m \geq 5$ and elements of $C$ have at least 4 fixed points \cite{KR:finite}.  Note their reduction is not almost parsimonious (or even weakly parsimonious), as it contains an exponentially small error term.  In particular, they are unable to deduce $\NP$-completeness of the corresponding decision problem as in our Corollary \ref{c:main}.

\begin{figure}[htb]
\begin{tabular}{l|c|c|c|c|c|c}
& \multicolumn{2}{|c}{one equation} &
    \multicolumn{2}{|c}{equations} &
    \multicolumn{2}{|c}{homomorphisms} \\ \hline
finite target $G$ & $\exists$ \cite{GR:solving} & \#  \cite{NJ:counting}
    & $\exists$ \cite{GR:solving} & \#  \cite{NJ:counting}
    & $\exists$ & \# \\ \hline
abelian & $\ccP$ & $\FP$ & $\ccP$ & $\FP$ & $\ccP$ & $\FP$ \\
nilpotent & $\ccP$ & ? & $\NPC$ & $\shPC$ & ? & ? \\
solvable & ?& ? & $\NPC$ & $\shPC$ & ? & ? \\
non-solvable & $\NPC$ & $\shPC$ & $\NPC$ & $\shPC$ & ? & ? \\
non-ab. simple & $\NPC$ & $\shPC$ & $\NPC$ & $\shPC$ & $\NPC$! & $\shPC$!
\end{tabular}
\caption{The complexity of solving equations over or finding homomorphisms
to a fixed finite target group $G$.  Here $\ccP$ denotes polynomial time for
a decision problem, $\FP$ denotes polynomial time for a function problem,
$\NPC$ is the class of $\NP$-complete problems, and $\shPC$ is the class of
$\shP$-complete problems.  Exclamation marks indicate results implied by Theorem \ref{th:main1}.}
\label{f:knownunknowns}
\end{figure}

We can place \Thm{th:main1} in the context of other counting
problems involving finite groups.   We summarize what is known in
\Fig{f:knownunknowns}.  Given a finite group $G$, the most general analogous
counting problem is the number of solutions to a system of equations that
may allow constant elements of $G$ as well as variables.  Nordh and Jonsson
\cite{NJ:counting} showed that this problem is $\shP$-complete if and only
if $G$ is non-abelian, while Goldman and Russell \cite{GR:solving} showed
that the existence problem is $\NP$-complete.  If $G$ is abelian, then any
finite system of equations can be solved by the Smith normal form algorithm.
These authors also considered the complexity of a single equation.  In this
case, the existence problem has unknown complexity if $G$ is solvable but
not nilpotent, while the counting problem has unknown complexity if $G$
is solvable but not abelian.

If all of the constants in a system of equations over $G$ are set to $1
\in G$, then solving the equations amounts to finding group homomorphisms
$f:\Gamma \to G$ from the finitely presented group $\Gamma$ given by the
equations.  By slight abuse of notation, we can call this counting problem
$\#H(\Gamma,G)$.  This is equivalent to the topological invariant $\#H(X,G)$
when $X$ is a simplicial complex, or even a triangulated $n$-manifold for
any fixed $n \ge 4$; in this case, given any finitely presented $\Gamma$,
we can construct $X$ with $\Gamma = \pi_1(X)$ in polynomial time.  To our
knowledge, \Thm{th:main1} is a new result for the invariant $\#H(\Gamma,G)$,
even though we specifically construct $\Gamma$ to be a 3-manifold group
rather than a general finitely presented group.  For comparison, both
the non-triviality problem and the word problem are as difficult as the
halting problem for general $\Gamma$ \cite{Poonen:sampler}.  By contrast,
the word problem and the isomorphism problem are both recursive for
3-manifold groups, in fact elementary recursive \cite{AFW:decision,K:homeo3}.

Going a dimension lower, if $M$ is a closed 2-manifold, then
there are well known formulas of Frobenius-Schur and Mednykh for
$\#H(M,G)$ \cite{FS:gruppen,Mednykh:compact,FQ:finite} for any
finite group $G$ as a function of the genus and orientability of $M$
\cite{FS:gruppen,Mednykh:compact,FQ:finite}.  Mednykh's formula was
generalized by Chen \cite{Chen:seifert} to the case of Seifert-fibered
3-manifolds.  In \Sec{s:standard}, we give a generalization of these
formulas to the class of bounded-width simplicial complexes.

Our approach to Theorems \ref{th:main1} and \ref{th:main2} (like that of Krovi and Russell)
is inspired by quantum computation and topological quantum
field theory.  Every unitary modular tensor category (UMTC) $\cC$
yields a unitary 3-dimensional topological quantum field theory
\cite{RT:ribbon,RT:manifolds,Turaev:quantum}.  The topological quantum
field theory assigns a vector space $V(\Sigma_g)$, or \emph{state
space}, to every oriented, closed surface.  It also assigns a state
space $V(\Sigma_{g,n},C)$ to every oriented, closed surface with $n$
boundary circles, where $C$ is an object in $\cC$ interpreted as the
``color" of each boundary circle.  Each state space $V(\Sigma_{g,n},C)$
has a projective action of the mapping class group $\MCG_*(\Sigma_{g,n})$.
(In fact the unpointed mapping class group $\MCG(\Sigma_{g,n})$ acts, but
we will keep the basepoint for convenience.)  These mapping class group
actions then extend to invariants of 3-manifolds and links in 3-manifolds.

Finally, the UMTC $\cC$ is universal for quantum computation if the image of
the mapping class group action on suitable choices of $V(\Sigma_{g,n},C)$
is large enough to simulate quantum circuits on $m$ qubits, with $g,n =
O(m)$.  If the action is only large enough to simulate classical circuits
on $m$ bits, then it is still classically universal.  These universality
results are important for the fault-tolerance problem in quantum computation
\cite{FLW:universal,K:tvcodes}.

One early, important UMTC is the (truncated) category $\rep_q(\SU(2))$
of quantum representations of $\SU(2)$ at a principal root of unity.  This
category yields the Jones polynomial for a link $L \subseteq S^3$ (taking $C
= V_1$, the first irreducible object) and the Jones-Witten-Reshetikhin-Turaev
invariant of a closed 3-manifold.  In separate papers, Freedman, Larsen,
and Wang showed that $V(\Sigma_{g,0})$ and $V(\Sigma_{0,n},V_1)$ are
both quantumly universal representations of $\MCG_*(\Sigma_{g,0})$ and $\MCG_*(\Sigma_{0,n})$ \cite{FLW:universal,FLW:two}.

Universality also implies that any approximation of these invariants that
could be useful for computational topology is $\shP$-hard.  Kuperberg
\cite{K:jones} obtained such results for the Jones polynomial of knots (see also
Aharonov-Arad \cite{AA:hardness}), while Alagic and Lo \cite{AL:manifolds}
obtained the analogous result for the corresponding 3-manifold invariant.
Note that exact evaluation of the Jones polynomial was earlier shown to
be $\shP$-hard without quantum computation methods \cite{JVW:complexity}.

If $G$ is a finite group, then the invariant $\#H(M,G)$ for a 3-manifold
$M$ also comes from a UMTC, namely the categorical double $D(\rep(G))$
of $\rep(G)$, that was treated (and generalized) by Dijkgraaf and Witten
and others \cite{K:hopf,DW:group,FQ:finite}.  In this case, the state
space $V(\Sigma_{g,0})$ is the vector space $\C[\hR_g(G)/\Inn(G)]$, and
the action of $\MCG_*(\Sigma_{g,0})$ on $V(\Sigma_{g,0})$ is induced by
its action on $\hR_g(G)$.  Some of the objects in $D(\rep(G))$ are given by
conjugacy classes $C \subseteq G$, and the representation of the braid group
$\MCG_*(\Sigma_{0,n})$ with braid strands colored by a conjugacy class $C$
yields the invariant $\#H(S^3 \setminus L,G,C)$.  
In this framework, Theorem \ref{th:main1} can be understood as a classical, combinatorial analog of \cite{FLW:universal}, whereas Theorem \ref{th:main2} is an analogue of \cite{FLW:two}.

Motivated by the fault tolerance problem, Ogburn and Preskill
\cite{OP:topological} found that the braid group action for $G = \Alt(5)$
is classically universal (with $C$ the conjugacy class of 3-cycles) and they
reported that Kitaev showed the same thing for $\Sym(5)$.   They also showed
if these actions are enhanced by quantum measurements in a natural sense,
then they become quantumly universal.  Later Mochon \cite{Mochon:finite}
extended this result to any non-solvable finite group $G$.  In particular,
he proved that the action of $\MCG_*(\Sigma_{0,n})$ is classically universal
for a suitably chosen conjugacy class $C$.

Mochon's result is evidence, but not proof, that $\#H(S^3 \setminus L,G,C)$
is $\shP$-complete for every fixed, non-solvable $G$ and every suitable
conjugacy class $C \subseteq G$ that satisfies his theorem.  His result
implies that if we constrain the associated braid group action with arbitrary
initialization and finalization conditions, then counting the number of
solutions to the constraints is parsimoniously $\shP$-complete.  However,
when we use a braid to describe a link via a plat presentation, as we do in our proof of Theorem \ref{th:main2},
then the description yields specific initialization and
finalizations conditions that must be handled algorithmically to obtain
hardness results.  Similarly, in our proof of \Thm{th:main1}, states in
$\hR_g(G)$ are initialized and finalized using the handlebodies $(H_g)_I$
and $(H_g)_F$.
If we could choose any initialization and finalization
conditions whatsoever, then it would be much easier to establish (weakly
parsimonious) $\shP$-hardness; it would take little more work than to
cite Theorems \ref{th:dt} and \ref{th:rv}.

For further discussion of our results, please see Chapter \ref{ch:discussion}.

\chapter{Complexity theory and algorithms review}
\label{ch:complexity}
This chapter reviews the necessary background from computational complexity.  In Section \ref{s:classes}, we define the complexity classes and notions of reduction relevant to our main theorems.  In Section \ref{s:circuits}, we define Boolean circuits, and state the Cook-Levin theorem, which provides the starting point of our reductions.  The final Section \ref{s:standard} contains a review of some standard algorithms in topology.  For more background on complexity theory, see Arora and Barak \cite{AB:modern} and the Complexity Zoo \cite{W:zoo}.

\section{Complexity classes}
\label{s:classes}

Let $A$ be a finite \emph{alphabet} (a finite set with at least 2 elements)
whose elements are called \emph{symbols}, and let $A^*$ be the set of finite
words in $A$.   We can consider three kinds of computational problems with
input in $A^*$: decision problems $d$, counting problems $c$, and function
problems $f$, which have the respective forms
\begin{equation} d:A^* \to \{\yes,\no\} \qquad \qquad c:A^* \to \N \qquad \qquad
    f:A^* \to A^*. \label{e:dcf} \end{equation}
The output set of a decision problem can also be identified with the
Boolean alphabet
\[ A = \Z/2 = \{1,0\} \cong \{\text{true},\text{false}\}
    \cong \{\yes,\no\}. \]

A \emph{complexity class} $\ccC$ is any set of function, counting, or
decision problems, which may either be defined on all of $A^*$ or require
a promise.  A specific, interesting complexity class is typically defined as
the set of all problems that can be computed with particular computational
resources.  For instance, $\ccP$ is the complexity class of all decision
problems $d$ such that $d(x)$ can be computed in polynomial time (in the
length $|x|$ of the input $x$) by a deterministic Turing machine.  $\FP$
is the analogous class of function problems that are also computable in
polynomial time.

A \emph{promise problem} is a function $d$, $c$, or $f$ of the same
form as \eqref{e:dcf}, except whose domain can be an arbitrary subset
$S \subseteq A^*$.  The interpretation is that an algorithm to compute
a promise problem can accept any $x \in A^*$ as input, but its output is
only taken to be meaningful when it is promised that $x \in S$.

The input to a computational problem is typically a data type such as an
integer, a finite graph, a simplicial complex, etc.  If such a data type can
be encoded in $A^*$ in some standard way, and if different standard encodings
are interconvertible in $\FP$, then the encoding can be left unspecified.
For instance, the decision problem of whether a finite graph is connected is
easily seen to be in $\ccP$; the specific graph encoding is not important.
Similarly, there are various standard encodings of the non-negative integers
$\N$ in $A^*$.  Using any such encoding, we can also interpret $\FP$
as the class of counting problems that can be computed in polynomial time.

The complexity class $\NP$ is the set of all decision problems $d$ that can
be answered in polynomial time with the aid of a prover who wants to convince
the algorithm (or verifier) that the answer is ``yes".  In other words, every
$d \in \NP$ is given by a two-variable predicate $v \in \ccP$.  Given an input
$x$, the prover provides a witness $y$ whose length $|y|$ is some polynomial
in $|x|$.  Then the verifier computes $v(x,y)$, with the conclusion that
$d(x) = \yes$ if and only if there exists $y$ such that $v(x,y) = \yes$.
The witness $y$ is also called a \emph{proof} or \emph{certificate},
and the verification $v$ is also called a \emph{predicate}.  Likewise,
a function $c(x)$ is in $\shP$ when it is given by a predicate $v(x,y)$;
in this case $c(x)$ is the number of witnesses $y$ that satisfy $v(x,y)$.
For instance, whether a finite graph $G$ (encoded as $x$) has a 3-coloring
is in $\NP$, while the number of 3-colorings of $G$ is in $\shP$.  In both
cases, a 3-coloring of $G$ serves as a witness $y$.

A computational problem $f$ may be $\NP$-hard or $\shP$-hard with the
intuitive meaning that it is provably at least as difficult as any problem
in $\NP$ or $\shP$.   A more rigorous treatment leads to several different
standards of hardness.  One quite strict standard is that any problem $g$
in $\NP$ or $\shP$ can be reduced to the problem $f$ by converting the
input; \ie, there exists $h \in \FP$ such that
\[ g(x) = f(h(x)). \]
If $f, g \in \NP$, then this is called \emph{Karp reduction}; if $f, g \in
\shP$, then it is called \emph{parsimonious reduction}.  Evidently, if a
counting problem $c$ is parsimoniously $\shP$-hard, then the corresponding
existence problem $d$ is Karp $\NP$-hard.

When a problem $f$ is $\shP$-hard by some more relaxed standard than
parsimonious reduction, there could still be an algorithm to obtain
some partial information about the value $f$, such as a congruence or an
approximation, even if the exact value is intractable.   For instance,
the permanent of an integer matrix is well-known to be $\shP$-hard
\cite{Valiant:permanent}, but its parity is the same as that of the
determinant, which can be computed in polynomial time.  However, when
a counting problem $c$ is parsimoniously $\shP$-hard, then the standard
conjecture that $\NP \not\subseteq \BPP$ implies that it is intractable to
obtain any partial information about $c$.  Here $\BPP$ is the set of problems
solvable in randomized polynomial time with a probably correct answer.

\begin{thm}[Corollary of Valiant-Vazirani \cite{VV:unique}]  Let $c$
be a parsimoniously $\shP$-hard problem, and let $b > a \ge 0$ be distinct,
positive integers.  Then distinguishing $c(x) = a$ from $c(x) = b$ is
$\NP$-hard via a Cook reduction in $\BPP$, given the promise that $c(x)
\in \{a,b\}$.
\label{th:vv} \end{thm}

When we say that an algorithm $\cA$ obtains partial information about
the value of $c(x)$, we mean that it can calculate $f(c(x))$ for some
non-constant function $f$.  Thus it can distinguish some pair of cases
$c(x) = a$ and $c(x) = b$; and by \Thm{th:vv}, this is $\NP$-hard.  Here a
\emph{Cook reduction} is a polynomial-time algorithm $\cB$ (in this case
randomized polynomial time) that can call $\cA$ as a subroutine.

\begin{proof} Given a problem $d \in \NP$, Valiant and Vazirani construct
a randomized algorithm $\cB$ that calculates $d(x)$ using a collection of
predicates $v_1(x,y)$ in $\ccP$ that usually have at most one solution in $y$.
Thus, if an algorithm $\cA$ can solve each problem
\[ d_1(x) = \exists? y \text{\ such that\ } v_1(x,y) = \yes \]
under the promise that at most one $y$ exists, then $\cA$ can be used as
a subroutine to compute the original $d$.  Such a predicate $v_1(x,y)$
may occasionally have more than one solution, but this happens rarely and
still allows $\cB$ to calculate $d$ by the standard that its output only
needs to be probably correct.

Given such a predicate $v_1(x,y)$, it is easy to construct another
predicate $v_2(x,y)$ in $\ccP$ that has $b-a$ solutions in $y$ for each
solution to $v_1(x,y)$, and that has $a$ other solutions in $y$ regardless.
Thus $v_2(x,y)$ has $b$ solutions when $d_1(x) = \yes$ and $a$ solutions
when $d_1(x) = \no$.  Thus, an algorithm $\cA$ that can distinguish
$c(x) = a$ from $c(x) = b$ can be used to calculate $d_1(x)$, and by the
Valiant-Vazirani construction can be used to calculate $d(x)$.
\end{proof}

A decision problem $d$ which is both in $\NP$ and $\NP$-hard is called
\emph{$\NP$-complete}, while a counting problem which is both in $\shP$
and $\shP$-hard is called \emph{$\shP$-complete}.   For instance the
decision problem $\CSAT$, circuit satisfiability over an alphabet $A$, is
Karp $\NP$-complete, while the counting version $\shCSAT$ is parsimoniously
$\shP$-complete (\Thm{th:csat}).  Thus, we can prove that any other problem
is $\NP$-hard by reducing $\CSAT$ to it, or $\shP$-hard by reducing $\shCSAT$
to it.

We mention three variations of parsimonious reduction.   A counting function
$c$ is \emph{weakly parsimoniously} $\shP$-hard if for every $b \in \shP$,
there are $f,g \in \FP$ such that
\[ b(x) = f(c(g(x)),x). \]
The function $c$ is \emph{almost parsimoniously} $\shP$-hard if $f$ does
not depend on $x$, only on $c(g(x))$.  In either case, we can also ask for
$f(c,x)$ to be 1-to-1 on the set of values of $c$ with $f^{-1} \in \FP$,
linear or affine linear in $c$, etc.  So, for instance, \Thm{th:main1} says that $\#H(M,G)$
is almost parsimoniously $\shP$-complete.

Finally, suppose that $c(x)$ counts the number of solutions to $v(x,y)$
and $b(x)$ counts the number of solutions to $u(x,y)$.  Then a \emph{Levin
reduction} is a map $h \in \FP$ and a bijection $f$ with $f,f^{-1} \in \FP$
such that
\[ u(x,y) = v(h(x),f(y)). \]
Obviously Levin reduction implies parsimonious reduction.

\section{Circuits}
\label{s:circuits}

Given an alphabet $A$, a \emph{gate} is a function $\alpha:A^k \to A^\ell$.
A \emph{gate set} $\Gamma$ is a finite set of gates, possibly with
varying sizes of domain and target, and a \emph{circuit} over $\Gamma$ is a
composition of gates in $\Gamma$ in the pattern of a directed, acyclic graph.
A gate set $\Gamma$ is \emph{universal} if every function $f:A^n \to A^m$
has a circuit.  For example, if $A = \Z/2$, then the gate set
\[ \Gamma =\{\AND, \OR, \NOT, \COPY\} \]
is universal, where $\AND$, $\OR$, and $\NOT$ are the standard Boolean
operations and the $\COPY$ gate is the diagonal embedding $a \mapsto (a,a)$.

Let $A$ be an alphabet with a universal gate set $\Gamma$, and suppose that
$A$ has a distinguished symbol $\yes \in A$.  Choose a standard algorithm
to convert an input string $x \in A^*$ to a circuit $Z_x$ with one output.
Then the \emph{circuit satisfiability problem} $\CSAT_{A,\Gamma}(x)$ asks
whether the circuit $Z_x$ has an input $y$ such that $Z_x(y) = \yes$.
It is not hard to construct a Levin reduction of $\CSAT_{A,\Gamma}$ from
any one alphabet and gate set to any other, so we can just call any such
problem $\CSAT$.  $\CSAT$ also has an obvious counting version $\shCSAT$.

\begin{thm}[Cook-Levin-Karp] $\CSAT$ is Karp $\NP$-complete and $\shCSAT$
is parsimoniously $\shP$-complete.
\label{th:csat} \end{thm}

See Arora-Barak \cite[Sec. 6.1.2 \& Thm. 17.10]{AB:modern} for a proof of
\Thm{th:csat}.

\section{Standard algorithms}
\label{s:standard}

In this section we will review a few standard algorithms that supplement
Theorems \ref{th:main1} and \ref{th:main2}.  Instead of hardness results, they are all easiness
results. (Note that \Thm{th:inshp} produces a conditional type of easiness,
namely predicates that can be evaluated in polynomial time.)

\begin{thm} The integer homology $H_*(X)$ of a finite simplicial
complex $X$ can be computed in polynomial time.
\label{th:homology} \end{thm}

Briefly, \Thm{th:homology} reduces to computing the Smith normal form of an
integer matrix and a corresponding matrix factorization \cite{DC:homology}.
Kannan and Bachem \cite{KB:smith} showed that a Smith factorization can
be computed in polynomial using a refinement of the standard Smith normal
form algorithm based on row and column operations.

\begin{prop} If $X$ is a finite simplicial complex given as
computational input, then it can be confirmed in polynomial time whether
$X$ is a closed 3-manifold $M$, and whether $M$ is a homology 3-sphere.
\label{p:mpromise} \end{prop}

\begin{proof} To be concrete, $X$ is described by a set of vertices and
a set of subsets of those vertices representing simplices.  We can then
trivially check the first two properties:
\begin{enumerate}
\item That every maximal simplex is 3-dimensional.
\item That the link of every edge is a polygon.
\end{enumerate}
It follows that the link $\lk(v)$ of every vertex $v$ is a surface; to check
that that $M$ is a closed 3-manifold, we want to know that every $\lk(v)$ is
a 2-sphere.   We can confirm this for instance by computing $H_*(\lk(v))$
using \Thm{th:homology}.  Then to confirm that $M$ is a homology 3-sphere
(including that it is orientable), we can again use \Thm{th:homology}
to calculate $H_*(M)$.
\end{proof}

\begin{thm} If $G$ is a fixed finite group and $X$ is a finite,
connected simplicial complex regarded as the computational input, then
$\#H(X,G)$ and $\#Q(X,G)$ are both in $\shP$.
\label{th:inshp} \end{thm}

\begin{proof} By choosing a spanning tree for the 1-skeleton of $X$, we
can convert its 2-skeleton to a finite presentation $P$ of $\pi_1(X)$.
Then we can describe a homomorphism $f:\pi_1(X) \to G$ by the list of
its values on the generators in $P$.  This serves as a certificate; the
verifier should then check whether the values satisfy the relations in $P$.
This shows that $\#H(X,G)$ is in $\shP$.

The case of $\#Q(X,G)$ is similar but slightly more complicated.  The map $f$
is surjective if and only if its values on the generators in $P$ generate
$G$; the verifier can check this.  The verifier can also calculate the
$\Aut(G)$-orbit of $f$.   Given an ordering of the generators and an
ordering of the elements of $G$, the verifier can accept $f$ only when
it is alphabetically first in its orbit.  Since only surjections are
counted and each orbit is only counted once, we obtain that $\#Q(X,G)$
certificates are accepted.
\end{proof}

There is an analog of Theorem \ref{th:inshp} for coloring invariants of knot diagrams, which has a similar proof.

\begin{thm}
If $G$ is a fixed finite group with $c \in G$ fixed, and $K$ is a knot diagram with meridian $\gamma$, together regard as computational input, then $\#H(K,\gamma,G,c)$ and $\#Q(K,\gamma,G,c)$ are both in $\shP$. \qed
\end{thm}

In the input to our last algorithm, we decorate a finite simplicial
complex $X$ with a complete ordering of its simplices (of all dimensions)
that refines the partial ordering of simplices given by inclusion.  If there
are $n$ simplices total, then for each $0 \le i \le n$, we let $X_i$ be
subcomplex formed by the first $i$ simplices, so that $X_0 = \emptyset$
and $X_n = X$.  Each $X_i$ has a relative boundary $\bd(X_i)$ in $X$.
(Here we mean boundary in the set of general topology rather than manifold
theory, \ie, closure minus interior.)  We define the \emph{width} of $X$
with its ordering to be the maximum number of simplices in any $\bd(X_i)$.

\begin{thm} If $G$ is a fixed finite group and $X$ is a finite, connected
simplicial complex with a bounded-width ordering, then $\#H(X,G)$ and
$\#Q(X,G)$ can be computed in polynomial time (non-uniformly in the width).
\label{th:width} \end{thm}

It is easy to make triangulations for all closed surfaces with uniformly
bounded width.  For instance, we can make such a triangulation of an
orientable surface $\Sigma_g$ from a Morse function chosen so that each
regular level is either one or two circles.  With more effort, we can make
a bounded-width triangulation of a Seifert-fibered 3-manifold $M$ using a
bounded-width triangulation of its orbifold base.   Thus \Thm{th:width}
generalizes the formulas of Mednykh \cite{Mednykh:compact} and Chen
\cite{Chen:seifert} in principle, although in practice their formulas
are more explicit and use better decompositions than triangulations.
\Thm{th:width} also applies to 3-manifolds with bounded Heegaard genus
(more generally, $n$-manifolds with bounded Morse width), as well as knot diagrams.

\begin{proof} We can calculate $|H(X,G)|$ using the formalism of non-abelian
simplicial cohomology theory with coefficients in $G$ \cite{Olum:nonabelian}.
In this theory, we orient the edges of $X$ and we mark a vertex $x_0 \in
X$ as a basepoint.  A 1-cocycle is then a function from the edges of $X$ to
$G$ that satisfies a natural coboundary condition on each triangle, while
a 0-cochain is a function from the vertices to $G$ that takes the value 1
at $x_0$.  The 1-cocycle set $Z^1(X;G)$ has no natural group structure when
$G$ is non-commutative, while the relative 0-cochain set $C^0(X,x_0;G)$
is a group that acts freely on $Z^1(X;G)$. Then the set of orbits
\[ H^1(X,x_0;G) \defeq Z^1(X;G) / C^0(X,x_0;G). \]
can be identified with the representation set $H(X,G)$, while if $X$ has $v$
vertices, then $C^0(X,x_0;G) \cong G^{v_1}$.  Thus
\[ |H(X,G)| = |Z^1(X;G)|/|G|^{v-1}. \]
Our approach is to compute $|Z^1(X;G)|$ and divide.  We can then also
obtain $|Q(X,G)|$ from $|H(X,G)|$ by applying M\"obius inversion to
equation \eqref{e:homsum}.

The algorithm is an example of \emph{dynamical programming} in computer
science.  Working by induction for each $i$ from $0$ to $n$, it maintains a
vector $v_i$ of non-negative integers that consists of the number of ways to
extend each 1-cocycle on $\bd(X_i)$ to a 1-cocycle of $X_i$.  The dimension
of $v_i$ may be exponential in the number of edges of $\bd(X_i)$, but
since that is bounded, the dimension of $v_i$ is also bounded.  It is
straightforward to compute $v_{i+1}$ from $v_i$ when we pass from $X_i$
to $X_{i+1}$.  If $X_{i+1} \setminus X_i$ is an edge, then $v_{i+1}$
consists of $|G|$ copies of $v_i$.  If $X_{i+1} \setminus X_i$ is a triangle
and $\bd(X_{i+1})$ has the same edges as $\bd(X_i)$, then $v_{i+1}$ is a
subvector of $v_i$.   If $\bd(X_{i+1})$ has fewer edges than $\bd(X_i)$,
then $v_{i+1}$ is obtained from $v_i$ by taking local sums of entries.
\end{proof}

\chapter{Group theory lemmas}
\label{ch:group}

In this chapter we collect some group theory results.  We do not consider any
of these results to be especially new, although we found it challenging
to prove \Thm{th:rubik}.

\section{Generating alternating groups}
\label{ss:alt}

\begin{lemma}[\Cf\ {\cite[Lem. 7]{DGK:shuffles}}] Let $S$ be a finite set
and let $T_1,T_2,\dots,T_n \subseteq S$ be a collection of subsets with at
least 3 elements each, whose union is $S$, and that form a connected graph
under pairwise intersection.  Then the permutation groups $\Alt(T_i)$
together generate $\Alt(S)$.
\label{l:alt} \end{lemma}

\begin{proof} We argue by induction on $|S \setminus T_1|$.  If $T_1 =
S$, then there is nothing to prove.  Otherwise, we can assume (possibly
after renumbering the sets) that there is an element $a \in T_1 \cap T_2$
and an element $b \in T_2 \setminus T_1$.  Let $\alpha \in \Alt(T_2)$ be a
3-cycle such that $\alpha(a) = b$.  Then the 3-cycles in $\Alt(T_1)$, and
their conjugates by $\alpha$, and $\alpha$ itself if it lies in $\Alt(T_1
\cup \{b\})$, include all 3-cycles in $\Alt(T_1 \cup \{b\})$.  Thus we
generate $\Alt(T_1 \cup \{b\})$ and we can replace $T_1$ by $T_1 \cup \{b\}$.
\end{proof}

\section{Joint surjectivity}
\label{ss:joint}

Recall the existence half of the Chinese remainder theorem:  If 
$d_1,d_2,\ldots,d_n$ are pairwise relatively prime integers, then the
canonical homomorphism
\[ f:\Z \to \Z/d_1 \times \Z/d_2 \times \dots \times \Z/d_n \]
is (jointly) surjective.   The main hypothesis is ``local" in the
sense that it is a condition on each pair of divisors $d_i$ and $d_j$,
namely $\mathrm{gcd}(d_i,d_j) = 1$.  For various purposes, we will need
non-commutative joint surjectivity results that resemble the classic Chinese
remainder theorem.  (But we will not strictly generalize the Chinese
remainder theorem, although such generalizations exist.)  Each version
assumes a group homomorphism
\[ f:K \to G_1 \times G_2 \times \dots \times G_n \]
that surjects onto each factor $G_i$, and assumes certain other local
hypotheses, and concludes that $f$ is jointly surjective.  Dunfield-Thurston
\cite[Lem. 3.7]{DT:random} and Kuperberg \cite[Lem. 3.5]{K:zdense}
both have results of this type and call them ``Hall's lemma", but Hall
\cite[Sec. 1.6]{Hall:eulerian} only stated without proof a special case
of Dunfield and Thurston's lemma.  Ribet \cite[Lem. 3.3]{Ribet:adic} also
has such a lemma with the proof there attributed to Serre.  In this paper,
we will start with a generalization of Ribet's lemma.

We define a group homomorphism
\[ f:K \to G_1 \times G_2 \times \dots \times G_n \]
to be \emph{$k$-locally surjective} for some integer $1 \le k \le n$ if it
surjects onto every direct product of $k$ factors.  Recall also that if $G$
is a group, then $G' = [G,G]$ is a notation for its commutator subgroup.

\begin{lemma}[After Ribet-Serre {\cite[Lem. 3.3]{Ribet:adic}}] 
Let 
\[ f:K \to G_1 \times G_2 \times \dots \times G_n \]
be a 2-locally surjective group homomorphism, such that also
its abelianization
\[ f_\ab:K \to (G_1)_\ab \times (G_2)_\ab \times \dots \times (G_n)_\ab \]
is $\ceil{(n+1)/2}$-locally surjective.  Then
\[ f(K) \ge G_1' \times G_2' \times \dots \times G_n'. \]
\label{l:ribet} \end{lemma}

\begin{proof} We argue by induction on $n$.  If $n=2$, then there is
nothing to do.  Otherwise let $t = \ceil{(n+1)/2}$ and note
that $n > t > n/2$.  Let 
\[ \pi:G_1 \times G_2 \times \dots \times G_n \to
    G_1 \times G_2 \times \dots \times G_t \]
be the projection onto the first $t$ factors.  Then $\pi \circ f$ satisfies
the hypotheses, so
\[ \pi(f(K)) \ge G_1' \times G_2' \times \dots \times G_t'. \]
Morever, $(\pi \circ f)_\ab$ is still $t$-locally surjective, which is to
say that
\[ \pi(f(K))_\ab = (G_1)_\ab \times (G_2)_\ab \times
    \dots \times (G_t)_\ab. \]
Putting these two facts together, we obtain
\[ \pi(f(K)) = G_1 \times G_2 \times \dots \times G_t. \]
Repeating this for any $t$ factors, we conclude that $f$ is $t$-locally
surjective.

Given any two elements $g_t,h_t \in G_t$, we can use $t$-local surjectivity
to find two elements
\begin{multline*} (g_1,g_2,\dots,g_{t-1},g_t,1,1,\dots,1), \\
    (1,1,\dots,1,h_t,h_{t+1},\dots,h_n) \in f(K).
\end{multline*}
Their commutator then is $[g_t,h_t] \in G_t \cap f(K)$.   Since $g_t$
and $h_t$ are arbitrary, we thus learn that $G'_t \le f(K)$, and since
this construction can be repeated for any factor, we learn that
\[ f(K) \ge G_1' \times G_2' \times \dots \times G_n', \]
as desired.
\end{proof}

We will also use a complementary result, Goursat's lemma, which can be used
to establish 2-local surjectivity.  (Indeed, it is traditional in some papers
to describe joint surjectivity results as applications of Goursat's lemma.)

\begin{lemma}[Goursat \cite{Goursat:divisions,BSZ:goursat}] Let $G_1$ and
$G_2$ be groups and let $H \leq G_1 \times G_2$ be a subgroup that
surjects onto each factor $G_i$.  Then there exist normal subgroups $N_i
\normaleq G_i$ such that $N_1 \times N_2 \le H$ and $H/(N_1 \times N_2)$
is the graph of an isomorphism $G_1/N_1 \cong G_2/N_2$.
\label{l:goursat} \end{lemma}

For instance, if $G_1$ is a simple group, then either $H = G_1 \times G_2$
or $H$ is the graph of an isomorphism $G_1 \cong G_2$.  In other words,
given a joint homomorphism
\[ f = f_1 \times f_2:K \to G_1 \times G_2 \]
which surjects onto each factor, either $f$ is surjective or $f_1$ and $f_2$
are equivalent by an isomorphism $G_1 \cong G_2$.  We can combine this
with the perfect special case of \Lem{l:ribet} to obtain exactly Dunfield
and Thurston's version.

\begin{lemma}[{\cite[Lem. 3.7]{DT:random}}] If 
\[ f:K \to G_1 \times G_2 \times \dots \times G_n \]
is a group homomorphism to a direct product of non-abelian simple groups,
and if no two factor homomorphism $f_i:K \to G_i$ and $f_j:K \to G_j$
are equivalent by an isomorphism $G_i \cong G_j$, then $f$ is surjective.
\label{l:sjoint}\end{lemma}

\begin{cor} Let $K$ be a group and let
\[ N_1,N_2,\dots,N_n \normal K \]
be distinct maximal normal subgroups with non-abelian simple quotients
$G_i = K/N_i$.  Then
\[ G_1 \cong (N_2 \cap N_3 \cap \dots \cap N_n)/
    (N_1 \cap N_2 \cap \dots \cap N_n). \]
\label{c:sjoint} \end{cor}

\begin{proof} We can take the product of the quotient maps to obtain
a homomorphism
\[ f:K \to G_1 \times G_2 \times \dots \times G_n \]
that satisfies \Lem{l:sjoint}.  Thus we can restrict $f$ to
\[ f^{-1}(G_1) = N_2 \cap N_3 \cap \dots \cap N_n \]
to obtain a surjection
\[ f:N_2 \cap N_3 \cap \dots \cap N_n \onto G_1. \]
This surjection yields the desired isomorphism.
\end{proof}

We will use a more direct corollary of \Lem{l:goursat}.  We say that
a group $G$ is \emph{normally Zornian} if every normal subgroup of $G$
is contained in a maximal normal subgroup.  Clearly every finite group is
normally Zornian, and so is every simple group.  A more interesting result
implied by Neumann \cite[Thm. 5]{Neumann:remarks} is that every finitely
generated group is normally Zornian.  (Neumann's stated result is that
every subgroup is contained in a maximal subgroup, but the proof works
just as well for normal subgroups.  He also avoided the axiom of choice
for this result, despite our reference to Zorn's lemma.)  Recall also the
standard concept that a group $H$ is \emph{involved} in another group $G$
if $H$ is a quotient of a subgroup of $G$.

\begin{lemma} Suppose that 
\[ f:K \to G_1 \times G_2 \]
is a group homomorphism that surjects onto the first factor $G_1$, and
that $G_1$ is normally Zornian.  Then:
\begin{enumerate}
\item If no simple quotient of $G_1$ is involved in $G_2$, then 
$f(K)$ contains $G_1$.
\item If $f$ surjects onto $G_2$, and no simple quotient of $G_1$
is a quotient of $G_2$, then $f$ is surjective.
\end{enumerate}
\label{l:zorn} \end{lemma}

\begin{proof} Case 1 reduces to case 2, since we can replace $G_2$ by the
projection of $f(K)$ in $G_2$.  In case 2, \Lem{l:goursat} yields isomorphic
quotients $G_1/N_1 \cong G_2/N_2$.  Since $G_1$ is normally Zornian, we
may further quotient $G_1/N_1$ to produce a simple quotient $Q$, and we
can quotient $G_2/N_2$ correspondingly.
\end{proof}

Finally, we have a lemma to calculate the simple quotients of a direct
product of groups.

\begin{lemma} If
\[ f:G_1 \times G_2 \times \cdots \times G_n \onto Q \]
is a group homomorphism from a direct product  to a non-abelian simple
quotient, then it factors through a quotient map $f_i:G_i \to Q$ for a
single value of $i$.
\label{l:prodquo} \end{lemma}

\begin{proof} The lemma clearly reduces to the case $n=2$ by induction.  If
\[ f:G_1 \times G_2 \onto Q \]
is a simple quotient, then $f(G_1)$ and $f(G_2)$ commute with each other,
so they are normal subgroups of the group that they generate, which by
hypothesis is $Q$.  So each of $f(G_1)$ and $f(G_2)$ is either trivial
or equals $Q$.  Since $Q$ is non-commutative, then $f(G_1)$ and $f(G_2)$
cannot both be $Q$, again because they commute with each other.  Thus one
of $G_1$ and $G_2$ is in the kernel of $f$, and $f$ factors through a
quotient of the other one.
\end{proof}

\section{Integer symplectic groups}

Recall that for any integer $g \ge 1$ and any commutative ring $A$,
there is an integer symplectic group $\Sp(2g,A)$, by definition the set of
automorphisms of the free $A$-module $A^{2g}$ that preserves a symplectic
inner product.  Likewise the projective symplectic group $\PSp(2g,A)$
is the quotient of $\Sp(2g,A)$ by its center (which is trivial in
characteristic 2 and consists of $\pm I$ otherwise).  For each prime
$p$ and each $g \ge 1$, the group $\PSp(2g,\Z/p)$ is a finite simple
group, except for $\PSp(2,\Z/2)$, $\PSp(2,\Z/3)$, and $\PSp(4,\Z/2)$
\cite[Thm. 11.1.2]{Carter:simple}.  Moreover, $\PSp(2g,\Z/p)$ is never
isomorphic to an alternating group when $g \ge 2$ (because it has the
wrong cardinality).

We want to apply \Lem{l:zorn} to the symplectic group $\Sp(2g,\Z)$,
since it is the quotient of the mapping class group $\MCG_*(\Sigma_g)$
by the Torelli group $\Tor_*(\Sigma_g)$.  To this end, we can
describe its simple quotients when $g \ge 3$.

\begin{lemma} If $g \ge 3$, then the simple quotients of $\Sp(2g,\Z)$
are all of the form $\PSp(2g,\Z/p)$, where $p$ is prime and the quotient
map is induced by the ring homomorphism from $\Z$ to $\Z/p$.
\label{l:symplectic} \end{lemma}

As the proof will indicate, \Lem{l:symplectic} is a mutual corollary of
two important results due to others:  the congruence subgroup property of
Mennicke and Bass-Lazard-Serre, and the Margulis normal subgroup theorem.

Note that the finite simple quotients of $\Sp(4,\Z)$ are only slightly
different.  The best way to repair the result in this case is to replace
both $\Sp(4,\Z)$ and $\Sp(4,\Z/2)$ by their commutator subgroups of index 2.
Meanwhile given the well-known fact that $\PSp(2,\Z) \cong C_2 * C_3$,
any simple group generated by an involution and an element of order 3 is
a simple quotient of $\Sp(2,\Z)$, and this is a very weak restriction.
However, we only need \Lem{l:symplectic} for large $g$.

\begin{proof} We note first that $\Sp(2g,\Z)$ is a perfect group when
$g \ge 3$, so every possible simple quotient is non-abelian, and every
such quotient is also a quotient of $\PSp(2g,\Z)$.  It is a special
case of the Margulis normal subgroup theorem \cite{Margulis:discrete}
that $\PSp(2g,\Z)$ is \emph{just infinite} for $g \ge 2$, meaning that
all quotient groups are finite.   Meanwhile, a theorem of Mennicke and
Bass-Lazard-Serre \cite{Mennicke:zur,BLS:fini} says that $\Sp(2g,\Z)$ has
the \emph{congruence subgroup property}, meaning that all finite quotients
factor through $\Sp(2g,\Z/n)$ for some integer $n > 1$.  Every finite
quotient of $\PSp(2g,\Z)$ likewise factors through $\PSp(2g,\Z/n)$, so we
only have to find the simple quotients of $\PSp(2g,\Z/n)$.

Clearly if a prime $p$ divides $n$, then $\PSp(2g,\Z/p)$ is a simple
quotient of $\PSp(2g,\Z/n)$.  We claim that there are no others.  Let $N$
be the kernel of the joint homomorphism
\[f:\PSp(2g,\Z/n) \to \prod_{p|n \text{ prime}} \PSp(2g,\Z/p). \]
If $\PSp(2g,\Z/n)$ had another simple quotient, necessarily non-abelian, then
by \Cor{c:sjoint}, it would also be a simple quotient of $N$.  It is easy
to check that $N$ is nilpotent, so all of its simple quotients are abelian.
\end{proof}

\section{Rubik groups}
\label{ss:rubik}

Recall the notation that $J' = [J,J]$ is the commutator subgroup of a
group $J$.

If $J$ is a group and $X$ is a $J$-set, then we define the $J$-set
symmetric group $\Sym_J(X)$ to be the group of permutations of $X$ that
commute with the action of $J$.  (Equivalently, $\Sym_J(X)$ is the group
of automorphisms of $X$ as a $J$-set.)  In the case that there are only
finitely many orbits, we define the \emph{Rubik group} $\Rub_J(X)$ to be
the commutator subgroup $\Sym_J(X)'$.  (For instance, the actual Rubik's
Cube group has a subgroup of index two of the form $\Rub_J(X)$, where $J =
C_6$ acts on a set $X$ with 12 orbits of order 2 and 8 orbits of order 3.)

If every $J$-orbit of $X$ is free and $X/J$ has $n$ elements, then we can
recognize $\Sym_J(X)$ as the restricted wreath product
\[ \Sym_J(X) \cong J \wr_{X/J} \Sym(X/J) \cong J \wr_n \Sym(n). \]
We introduce the more explicit notation
\begin{align*}
\Sym(n,J) &\defeq J \wr_n \Sym(n) \\
\Alt(n,J) &\defeq J \wr_n \Alt(n) \\
\Rub(n,J) &\defeq \Sym(n,J)'.
\end{align*}
We can describe $\Rub(n,J)$ as follows.  Let $J_\ab$ be the abelianization
of $J$, and define a map $\sigma:J^n \to J_\ab$ by first abelianizing $J^n$
and then multiplying the $n$ components in any order. Let $\AD(n,J) \le J^n$
($\AD$ as in ``anti-diagonal") be the kernel of $\sigma$.  Then:

\begin{prop} For any integer $n > 1$ and any group $J$, the commutator
subgroup of $\Sym(n,J)$ is given by
\[ \Rub(n,J) = \AD(n,J) \rtimes \Alt(n). \]
\end{prop}

\begin{proof} It is easy to check that $(\ker \sigma) \rtimes \Alt(n)$
is a normal subgroup of $\Sym(n,J)$ and that the quotient is the abelian
group $J_\ab \times C_2$. This shows that
\[ \AD(n,J) \rtimes \Alt(n) \supseteq \Rub(n,J). \]

To check the opposite inclusion, note that $\AD(n,J) \rtimes \Alt(n)$
is generated by the union of $(J')^n$, $\Alt(n) = \Sym(n)'$, and all
permutations of elements of the form
\[ (x,x^{-1},1,\dots,1) \in J^n.\]
Clearly $\Rub(n,J)$ contains the former two subsets.  Since
\[ (x,x^{-1},1,\dots,1) = [(x,1,1,\dots,1),(1 \ 2)]\]
(and similarly for other permutations),  we see
\[ \AD(n,J) \rtimes \Alt(n) \le \Rub(n,J).\]
We conclude with the desired equality.
\end{proof}

The main result of this section is a condition on a group homomorphism
to $\Rub(n,J)$ that guarantees that it is surjective.  We say that a group
homomorphism
\[ f:K \to \Sym(n,J)\]
is \emph{$J$-set $i$-transitive} if it acts transitively on ordered lists
of $i$ elements that all lie in distinct $J$-orbits.

\begin{thm} Let $J$ be a group and let $n \ge 7$ be an integer such
that $\Alt(n-2)$ is not a quotient of $J$.  Suppose that a homomorphism
\[ f:K \to \Rub(n,J) \]
is $J$-set 2-transitive and that its composition with the projection
$\Rub(n,J) \to \Alt(n)$ is surjective.  Then $f$ is surjective.
\label{th:rubik} \end{thm}

\begin{proof}
In the proof we will mix Cartesian
product notation for elements of $J^n$ with cycle notation for permutations.
The proof is divided into three steps.

Step 1: We let $H = f(K)$, and we consider its normal subgroup
\[ D \defeq H \cap J^n. \]
We claim that $D$ is $2$-locally surjective.  To this end, we look at the
subgroup $\Alt(n-2) \le \Alt(n)$ that fixes the last two letters (say).
Then there is a projection
\[ \pi:J^n \rtimes \Alt(n-2) \to J^2 \times \Alt(n-2) \]
given by retaining only the last two coordinates of $g \in J^n$.  We let
\[ P = \pi(H \cap (J^n \rtimes \Alt(n-2))). \]
Since $H$ is $J$-set 2-transitive, the group $P$ surjects onto $J^2$;
since $H$ surjects onto $\Alt(n)$, $P$ surjects onto $\Alt(n-2)$.  Thus we
can apply \Lem{l:zorn} to the inclusion
\[ P \leq J^2 \times \Alt(n-2). \]
Since $\Alt(n-2)$ is not a quotient of $J$ and therefore not $J^2$ either
(by \Lem{l:prodquo}), we learn that
\[ P = J^2 \times \Alt(n-2) \]
and that 
\[ J^2 \leq H \cap (J^n \rtimes \Alt(n-2)). \]
So the group $D = H \cap J^n$ surjects onto the last two coordinates
of $J^n$.  Since we can repeat the argument for any two coordinates, $D$
is 2-locally surjective.

Step 2: Suppose that $J$ is abelian.  Then $D$ is a subgroup of $J^n$
which is 2-locally surjective.  Since $J^n$ is abelian, conjugation of
elements of $D$ by elements of $H$ that surject onto $\Alt(n)$ coincides
with conjugation by $\Alt(n)$; thus $D$ is $\Alt(n)$-invariant.  By step 1,
for each $g_1 \in J$, there exists an element
\[ d_1 \defeq (x_1,1,x_3,x_4,\dots,x_n) \in D. \]
We now form a commutator with elements in $\Alt(n)$ to obtain
\begin{align*}
d_2 &\defeq [d_1,(1\;2)(3\;4)]
    = (x_1,x_1^{-1},x_3x_4^{-1},x_3^{-1}x_4,1,\dots,1) \in D. \\
d_3 &\defeq [d_2,(1 \ 2 \ 5)(3 \ 4)(6 \ 7)]
    = (x,1,1,1,x^{-1},1,\dots,1) \in D.
\end{align*}
The $\Alt(n)$-orbit of $d_3$ generates $\AD(n,J)$, thus $D = \AD(n,J)$.

Step 3: In the general case, step 2 tells us that $D_\ab = \AD(n,J_\ab)$
is $(n-1)$-locally surjective.  This together with step 1 tells us that $D
\le J^n$ satisfies the hypothesis of \Lem{l:ribet}, which tells us that $D
= \AD(n,J)$.  It remains only to show that $\Alt(n) \le H$.  It suffices
to show that $H/D$ contains (indeed is) $\Alt(n)$ in the quotient group
\[ \Alt(n,J)/D \cong (J^n/D) \rtimes \Alt(n)
    \cong (J^n/D) \times \Alt(n). \]
Now let $D_0 = (J^n \cap H)/D$, so that $H$ surjects onto $D_0 \times
\Alt(n)$.  Since $\Alt(n)$ is not a quotient of $D_0$ (for one reason,
because $J^n/D$ is abelian), we can thus apply \Lem{l:zorn} to conclude
that $H/D$ contains $\Alt(n)$.
\end{proof}

\begin{lemma} If $J$ is a group and $n \ge 5$, then $\Rub(n,J)/\AD(n,J)
\cong \Alt(n)$ is the unique simple quotient of $\Rub(n,J)$.
\label{l:rubikquo} \end{lemma}

\begin{proof} We first claim that $\Rub(n,J)$ is a perfect group.  For any
two elements $g,h \in J$, we can take commutators such as
\[
[(x_1,x_1^{-1},1,1,\dots,1), (x_2,1,x_2^{-1},1,\dots,1)]
    = ([x_1,x_2],1,1,\dots,1) \in \AD(n,J)', \]
to conclude that
\[ (J')^n = \AD(n,J)' \le \Rub(n,J)'. \]
We can thus quotient $\Rub(n,J)$ by $(J')^n$ and replace $J$ by $J_\ab$,
or equivalently assume that $J$ is abelian.  In this case, we can take
commutators such as
\[
[(x,1,x^{-1},1,1,\dots,1),(1\; 2)(4\; 5)]
    = (x,x^{-1},1,1,\dots,1) \in \Rub(n,J)'
\]
to conclude that $\AD(n,J) \le \Rub(n,J)'$.  Meanwhile $\Alt(n) \le
\Rub(n,J)'$ because it is a perfect subgroup.  Thus $\Rub(n,J)$ is perfect.

Suppose that
\[ f:\Rub(n,J) \onto Q \]
is a second simple quotient map, necessarily non-abelian.  Then \Cor{c:sjoint}
tells us that $f$ is also surjective when restricted to $\AD(n,J)$.  If $J$
is abelian, then so is $\AD(n,J)$ and this is immediately impossible.
Otherwise we obtain that the restriction of $f$ to $\AD(n,J)' = (J')^n$ is
again surjective, and we can apply \Lem{l:prodquo} to conclude that
$f|_{(J')^n}$ factors through a quotient $h:J' \to Q$ on a single factor.
But then $(\ker f) \cap (J')^n$ would not be invariant under conjugation
by $\Alt(n)$ even though it is the intersection of two normal subgroups
of $\Rub(n,J)$, a contradiction.
\end{proof}

\chapter{Mapping class group actions}
\label{ch:mcg}
\section{Closed surfaces}
\label{s:closed}
In this section, we let $G$ be a fixed finite simple group, and we use ``eventually" to mean ``when the genus $g$ is sufficiently large".

Recall from \Sec{ss:sketch1} that we consider several sets of homomorphisms
of the fundamental group of the surface $\Sigma_g$ to $G$:
\begin{align*}
\hR_g(G) &\defeq \{f:\pi_1(\Sigma_g) \to G\} \\
R_g(G) &\defeq \{f: \pi_1(\Sigma_g) \onto G\} \subseteq \hR_g(G) \\
R^s_g(G) &\defeq \{f \in R_g \mid \sch(f) = s\}.
\end{align*}
For convenience we will write $R_g = R_g(G)$, etc., and only give the
argument of the representation set when the target is some group other
than $G$.

The set $\hR_g$ has an action of $J = \Aut(G)$ and a commuting action of
$\MCG_*(\Sigma_g)$, so we obtain a representation map
\[ \rho:\MCG_*(\Sigma_g) \to \Sym_J(\hR_g). \]
Since $\MCG_*(\Sigma_g)$ is perfect for $g \ge 3$ \cite[Thm. 5.2]{FM:primer}
(and we are excluding small values of $g$), we can let the target be
$\Rub_J(\hR_g)$ instead.  Now $R_g$ and $R_g^0$ are both invariant subsets
under both actions; in particular the representation map projects to maps
to $\Sym_J(\hR_g \setminus R_g)$ and $\Sym_J(R^0_g)$.  At the same time,
$\MCG_*(\Sigma_g)$ acts on $H_1(\Sigma_g) \cong \Z^{2g}$, and we get a
surjective representation map
\[ \tau:\MCG_*(\Sigma_g) \to \Sp(2g,\Z), \]
whose kernel is by definition the Torelli group $\Tor_*(\Sigma_g)$.

The goal of this subsection is the following theorem.

\begin{thm} The image of the joint homomorphism
\[
\rho_{R^0_g} \times \rho_{\hR_g \setminus R_g} \times \tau:\MCG_*(\Sigma_g) 
    \to \Rub_J(R^0_g) \times \Rub_J(\hR_g \setminus R_g) \times \Sp(2g,\Z)
\]
eventually contains $\Rub_J(R^0_g)$.
\label{th:dtrefine} \end{thm}

Comparing \Thm{th:dtrefine} to the second part of \Thm{th:dt}, it says that
\Thm{th:dt} still holds for the smaller Torelli group $\Tor_*(\Sigma_g)$,
and after that the action homomorphism is still surjective if we lift
from $\Alt(R^0_g/J)$ to $\Rub_J(R^0_g)$.   Its third implication is
that we can restrict yet further to the subgroup of $\Tor_*(\Sigma_g)$
that acts trivially on $\hR_g \setminus R_g$, the set of non-surjective
homomorphisms to $G$.

The proof uses a lemma on relative sizes of representation sets.

\begin{lemma} Eventually
\[ |R^0_g/J| > |\hR_g \setminus R_g|.\]
\label{l:outgrow} \end{lemma}

\begin{proof}
Informally, if $g$ is large and we choose a homomorphism $f \in \hR_g$
at random, then it is a surjection with very high probability;
if it is a surjection, then its Schur invariant $\sch(f)$ is
approximately equidistributed.  In detail, Dunfield-Thurston
\cite[Lems.~6.10~\&~6.13]{DT:random} show that
\[ \lim_{g \to \infty} \frac{|R_g|}{|\hR_g|} = 1 \qquad
    \lim_{g \to \infty} \frac{|R^0_g|}{|R_g|} = \frac{1}{|H_2(G)|}. \]
Thus
\[ \lim_{g \to \infty} \frac{|\hR_g \setminus R_g|}{|R^0_g/J|} =
    |H_2(G)|\cdot|J|\cdot\bigg(\lim_{g \to \infty}
    \frac{|\hR_g|}{|R_g|}-1\bigg) = 0. \qedhere \]
\end{proof}

\begin{proof}[Proof of \Thm{th:dtrefine}] We first claim that $\rho_{R^0_g}$
by itself is eventually surjective. Note that the action of $J$ on $R^0_g$
is free; thus we can apply \Thm{th:rubik} if $\rho_{R^0_g}$ satisfies
suitable conditions.  By part 2 of \Thm{th:dt}, $\rho_{R^0_g}$ is eventually
surjective when composed with the quotient $\Rub_J(R^0_g) \to \Sym(R^0_g/J)$.
Meanwhile, part 1 of \Thm{th:dt} says that $\MCG_*(\Sigma_g)$ eventually acts
transitively on $R^0_g(G^2)$.  Since $G$ is simple, \Lem{l:sjoint} tells us
that the homomorphisms $f \in R^0_g(G^2)$ correspond to pairs of surjections
\[ f_1,f_2:\Sigma_g \onto G \]
that are inequivalent under $J = \Aut(G)$.  This eventuality is thus the
condition that the action of $\MCG_*(\Sigma_g)$ is $J$-set 2-transitive
in its action on $R^0_g$.  (\Cf\ Lemma 7.2 in \cite{DT:random}.)
Thus $\rho_{R^0_g}$ eventually satisfies the hypotheses of \Thm{th:rubik}
and is surjective.

The map $\tau$ also surjects $\MCG_*(\Sigma_g)$ onto $\Sp(2g,\Z)$.
Lemmas~\ref{l:symplectic} and \ref{l:rubikquo} thus imply that
$\Rub_J(R^0_g)$ and $\Sp(2g,\Z)$ do not share any simple quotients.
By \Lem{l:zorn}, $\MCG_*(\Sigma_g)$ surjects onto $\Rub_J(R^0_g) \times
\Sp(2g,\Z)$.  Equivalently, $\ker \tau = \Tor_*(\Sigma_g)$ surjects onto
$\Rub_J(R^0_g)$.

Finally we consider 
\[ \rho_{R^0_g} \times \rho_{\hR_g \setminus R_g}:\Tor_*(\Sigma_g)
    \to \Rub_J(R^0_g) \times \Rub_J(\hR_g \setminus R_g), \]
which we have shown surjects onto the first factor. The unique simple
quotient $\Alt(R^0_g/J)$ of $\Rub_J(R^0_g)$ is eventually not involved in
$\Rub_J(\hR_g \setminus R_g)$ because it is too large.  More precisely,
\Lem{l:outgrow} implies that eventually
\[ |\Alt(R^0_g/J)| > |\Alt(\hR_g \setminus R_g)| >
    |\Rub_J(\hR_g \setminus R_g)|. \]
Thus we can apply \Lem{l:zorn} to conclude that the image of
$\Tor_*(\Sigma_g)$ contains $\Rub_J(R^0_g)$, which is equivalent to the
conclusion.
\end{proof}

\section{Punctured disks}
\label{s:disks}
The main goal of this section is Theorem \ref{th:rvrefine}, which is a refinement of the portion of the ``full monodromy theorem" of Roberts and Venkatesh extracted in Theorem \ref{th:rv}.  Aside from notation introduced herein, Theorem \ref{th:rvrefine} is the only result necessary for the reduction in Section \ref{s:knotreduction}.

\subsection{Actions of interest}
\label{ss:actions}
For this subsection and the following, except where stated otherwise, $G$ can be any finite group, and $c$ any element such that its conjugacy class $C$ generates $G$.  (In other words, $c$ normally generates $G$.). We define the relevant braid subgroups and their actions, and recall the theorem of Conway-Parker characterizing the orbits of these actions in the many puncture limit.

For any positive integer $k$, let
\[ v = (v_1,v_2,\dots,v_{2k-1},v_{2k}) = (C,C^{-1},\dots,C,C^{-1}) \]
be a $2k$-tuple with entries alternating between the symbols $C$ and $C^{-1}$.  We use $v$ to color the $i^\text{th}$ puncture of the $2k$-punctured disk
\[ D_{2k} = D^2 \setminus \{p_1,\dots,p_{2k}\} \]
with the color $v_i$.  Similarly, we alternately color the $2k$ strands of the braid group $B_{2k}$.  Define $B_v \leq B_{2k}$ to be the subgroup of braids that preserve this coloring of strands by $v$.  Note: if $C=C^{-1}$, then both of these colorings are just 1-colorings and $B_v = B_{2k}$.

As in Subsection \ref{ss:sketch2}, let
\[ \hat{R}_{2k}(G) \defeq \{f: \pi_1(D_{2k}) \to G \}. \]
As in the previous section, we will suppress the dependence on $G$ when it is clear, and just write $\hat{R}_{2k}$.  Choose a set of generators of $\pi_1(D_{2k})$ represented by simple closed curves $\gamma_1,\dots,\gamma_{2k}$, where each $\gamma_i$ winds once, counterclockwise, around the puncture $p_i$, and zero times around the other punctures.  The sets that interest us are
\[ \hat{R}_v \defeq \{ f \in \hat{R}_{2k} \mid f(\gamma_i) \in v_i, \prod_{i=1}^{2k}f(\gamma_i) = 1 \} \subset (C \times C^{-1})^k \subset \hat{R}_{2k}\]
and
\[ R_v \defeq \{f \in \hat{R}_v \mid f \text{ is onto}\}. \]
The colored braid group $B_v$ acts on $\hat{R}_v$, and $R_v$ is a $B_v$-invariant subset.  The technical goal of the proof of Theorem \ref{th:main2} is to give a precise enough description of this action so that we can do a gadget construction with it.

\subsection{Schur-type braid invariants and the Conway-Parker theorem}
\label{ss:CP}
As a first step, we describe the orbits of the $B_v$ action on $R_v$ in the limit that $k$---and, hence, $v$---is large enough.  In this context, we use the word ``eventually" to mean ``for all $k$ large enough."  Recall that in this subsection and the previous, we allow $G$ to be any finite group, and $C$ any conjugacy class that generates $G$.  The only exception is Lemma \ref{l:props}(3), where we make the further requirement that $G$ is perfect.

Our main tool is a certain $B_v$-invariant called the \emph{Conway-Parker (universal) lifting invariant}.  The most general definition of this invariant is due to Ellenberg, Venkatesh and Westerland \cite{EVW:hurwitz2}, although the ideas go back to unpublished work of Conway and Parker \cite{CP:hurwitz}, which were first relayed in a publication by Fried and Volklein \cite{FV:inverse}.  Since \cite{EVW:hurwitz2} was never published, we also refer the reader to \cite{RV:hurwitz} for an exposition.  We remark that the construction of \cite{EVW:hurwitz2} and \cite{RV:hurwitz} is carried out for more general colorings $v$ than we consider, but, for expediency's sake, we restrict to our specific setting.

We follow \cite[Sec.~4]{RV:hurwitz}, including their notation.  Let $H_2(G)$ be the Schur multiplier of $G$.  Equivalently, $H_2(G)$ is the second integral homology of the classifying space $K(G,1)$.  Let $H_2(G)_{C\cup C^{-1}} \leq H_2(G)$ denote the subgroup generated by homology classes represented by maps from tori where a meridian maps to an element of $C\cup C^{-1}$.  Then the \emph{reduced Schur multiplier} is defined as the quotient
\[ H_2(G,C\cup C^{-1}) \defeq H_2(G)/H_2(G)_{C\cup C^{-1}}. \]
We emphasize that $H_2(G,C)$ is \emph{not} a relative homology group, despite the unfortunate notation.  The \emph{Conway-Parker invariant} is a $B_v$-invariant map
\[ \inv_v: \hat{R}_v \to H_2(G,C\cup C^{-1},v) \]
where $H_2(G,C\cup C^{-1},v)$ is a torsor over $H_2(G,C\cup C^{-1})$, defined as follows.

Fix a Schur cover $\tilde{G} \to G$, \ie, a central extension with kernel in the derived subgroup of $\tilde{G}$ such that the order of this extension is maximal among all such extensions.  Then the kernel is canonically isomorphic to $H_2(G)$.  We form the \emph{reduced Schur cover} $\tilde{G}_{C\cup C^{-1}} = \tilde{G}/H_2(G)_{C\cup C^{-1}}$, which fits into an exact sequence
\[ H_2(G,C\cup C^{-1}) \into \tilde{G}_{C\cup C^{-1}} \onto G. \]
Define $H_2(G,C\cup C^{-1},C)$ to be the set of conjugacy classes of $\tilde{G}_{C \cup C^{-1}}$ that lie in the preimage of $C$.  Let $\tilde{1} \in H_2(G,C\cup C^{-1})$ and if $\tilde{c}$ is a lift of an element $c \in C$, denote the conjugacy class of $\tilde{c}$ by $[\tilde{c}] \in H_2(G,C\cup C^{-1},C)$.  Then the action $\tilde{1} \cdot [\tilde{c}] = [\tilde{1}\cdot \tilde{c}]$ makes $H_2(G,C\cup C^{-1},C)$ into a torsor over $H_2(G,C\cup C^{-1})$.

We define $H_2(G,C\cup C^{-1},C^{-1})$ similarly, also making it a torsor over $H_2(G,C\cup C^{-1})$.

Torsors over the same abelian group can be multiplied.  If $T_1$ and $T_2$ are torsors over $A$, then their product, as a set, is $(T_1 \times T_2)/\sim$, where $(t_1,t_2) \sim (at_1,a^{-1}t_2)$ for all $a \in A$.  We define
\[ H_2(G,C\cup C^{-1},v) = \prod_{i=1}^{2k} H_2(G,C\cup C^{-1},v_i), \]
where the product symbol indicates the product of $H_2(G,C\cup C^{-1})$-torsors.

The invariant $\inv_v(f)$ is now defined by writing
\[ f = (f_1,\dots,f_{2k}) \in (C \times C^{-1})^{2k} \]
then arbitrarily picking preimages $\tilde{f}_i \in \tilde{G}_{C \cup C^{-1}}$ so that
\[ \tilde{f}_1\cdots \tilde{f}_{2k} = 1 \]
and letting
\[ \inv_v(f) \defeq \prod_{i=1}^{2k} [\tilde{f}_i] \in H_2(G,C\cup C^{-1},v). \]
It is straightforward to verify that $\inv_v(f)$ does not depend on the choices made, and is $B_v$ invariant.  More importantly, the next theorem shows $\inv_v$ is quite useful.

\begin{thm}[Conway-Parker theorem {\cite[Thm.~7.5.1]{EVW:hurwitz2}}; see {\cite[Prop.~4.1]{RV:hurwitz}}]
For any finite group $G$ and conjugacy class $C$ that generates $G$, eventually:
\[ \inv_v: R_v/B_v \leftrightarrow H_2(G,C,v).\]
In other words, the Conway-Parker invariant is eventually a complete invariant for the orbits of the $B_v$ action on $R_v$.
\label{th:CP}
\end{thm}

Strictly speaking, \cite[Prop.~4.1]{RV:hurwitz} only applies to the orbits of the action of $B_v$ on $R_v/G$, where $G$ acts by postcomposition with inner automorphisms.  However, because the elements of $R_v$ are surjective homomorphisms, one can easily show that the $B_v$-equivalence relation on $R_v$ refines the $G$-equivalence relation.  That is, if two maps $f,h \in R_v$ are equivalent after applying an inner automorphism of $G$ to $f$, then they are also equivalent after applying some braid to $f$.  Alternatively, \cite[Thm.~7.5.1]{EVW:hurwitz2} contains our Theorem \ref{th:CP} as a special case.

Later, we need a few basic properties of the Conway-Parker invariant, one of which requires a definition: we say $f\in R_v$ \emph{bounds a plat} if there is an inclusion $D_{2k}^2 \to B^3$ so that $f$ extends to a homomorphism from the fundamental group of the complement of a trivial tangle in $B^3$.

\begin{lemma}
The Conway-Parker invariant has the following properties:
\begin{enumerate}
\item Given our choice of $v$, the torsor $H_2(G,C,v)$ has a natural basepoint that allows us to identify it with the group $H_2(G,C)$.
\item If $f \in (C \times C^{-1})^{k}$ bounds a plat, then $\inv_v(f) = 0$.
\item If $G$ is perfect, then for all positive integers $i$,
\[ H_2(G^i,C^i) \cong H_2(G,C)^i.\]
\end{enumerate}
\label{l:props}
\end{lemma}

\begin{proof}
\begin{enumerate}
\item The product of torsors equipped with basepoints has a natural basepoint.  So it suffices to show the product torsor
\[ H_2(G,C\cup C^{-1},C) \times H_2(G,C\cup C^{-1},C^{-1}) \]
has a natural basepoint.  If $[\tilde{c}]$ is the conjugacy class of an arbitrary lift of an element $c \in C$ to $\tilde{G}_{C \cup C^{-1}}$, then $[\tilde{c}] \times [\tilde{c}]^{-1}$ does not depend on the choice of $c$, and thus provides a natural basepoint.

\item If $f$ bounds a plat, then by isotoping the plat, we can show $f$ is in the same braid orbit as a map of the form
\[ h=(h_1,h_1^{-1},h_2,h_2^{-1},\dots,h_k,h_k^{-1}). \]
Clearly $\inv_v(h) = 0$, hence $\inv_v(f) = \inv_v(h) = 0$.

\item This follows from the definition of $H_2(G,C)$, the K\"{u}nneth theorem, and the fact that
\[ H_2(G^i)_{C^i} \cong \left( H_2(G)_C\ \right)^{\oplus i}. \]
\end{enumerate}
\end{proof}

In our final reduction, we will be especially interested in homomorphisms that bound plats, and, hence, by the lemma, homomorphisms with $\inv_C(f) = 0$.  We collect them into the subsets
\[ \begin{aligned}
\hat{R}_v^0 &\defeq \{ f \in \hat{R}_v \mid \inv_v(f) = 0 \}, \\
R_v^0 &\defeq \{ f \in R_v \mid \inv_v(f) = 0 \}.
\end{aligned} \]

\subsection{Refining the theorem of Roberts-Venkatesh}
\label{ss:rvrefine}
We now refine Theorem \ref{th:rv}(2) using the Rubik group Theorem \ref{th:rubik}.  The latter is relevant for the following reason.  Let $\Aut(G,C)$ be the group of automorphisms of $G$ fixing $C$ setwise.  Then $\Aut(G,C)$ acts on $R_v^0$.  This action is free, because the homomorphisms in $R_v^0$ are all surjective.  Moreover, the actions of $B_v$ and $\Aut(G,C)$ on $R_v^0$ commute.  In other words, the image of
\[ \rho: B_v \to \Sym(R_v^0) \]
is contained inside $\Sym_{\Aut(G,C)}(R_v^0)$.  In particular, there is an induced action of $B_v$ on the quotient $R_v^0/\Aut(G,C)$, which is precisely the subject of Theorem \ref{th:rv}(2).

\begin{lemma}
Let $G$ be a nonabelian simple group and let $C\subset G$ be a conjugacy class.  Then for every $i>0$, $B_v$ eventually acts $\Aut(G,C)$-set $i$-transitivity on $R_v^0$.
\label{l:trans}
\end{lemma}

One can almost recover this lemma from Roberts and Venkatesh's proof of the full monodromy theorem.  They show eventual $\Aut(G,C)$-set $i$-transitivity on $R_v^0/G$.  We could use some tricks to squeeze what we need out of this, but, at this point, it is just as much work (and hopefully more enlightening) to repeat their argument.  We note that this argument is identical to Dunfield and Thurston's proof of Theorem \ref{th:dt}.

\begin{proof}
We begin by choosing $k$ large enough so that the conclusion of Theorem \ref{th:CP} holds for the finite group $G^i$ and the conjugacy class $C^i$.  Let $f_1\dots,f_i \in R_v^0$ lie in distinct $\Aut(G,C)$ orbits and consider the product homomorphism
\[ f = f_1 \times \cdots \times f_i: \pi_1(D_{2k}^2) \to G \times \cdots \times G. \]
Similarly, let $g_1\dots,g_i \in R_v^0$ lie in distinct $\Aut(G,C)$ orbits, and form the product homomorphism $g = g_1 \times \cdots\times g_i$.  By Lemma \ref{l:sjoint}, $f$ and $g$ are both surjective.  Lemma \ref{l:props} shows that $\inv(f) = \inv(g) = 0$.  By Theorem \ref{th:CP}, $f$ and $g$ are in the same braid orbit, which is equivalent to the conclusion of the lemma.  
\end{proof}

With this lemma at our disposal, we can reprove Theorem \ref{th:rv}(2).

\begin{proof}[Proof of Theorem \ref{th:rv}(2)]
The previous lemma shows $B_v$ eventually acts $\Aut(G,C)$-set $6$-transitively on $R_v^0$.  It follows that $B_v$ eventually acts $6$-transitively (in the usual sense) on $R_v^0/\Aut(G,C)$.  Conclude by using the following corollary of the classification of finite simple groups: if a group acts $6$-transitively on a finite set $X$, then the image of that group in $\Sym(X)$ contains $\Alt(X)$.
\end{proof}

Our refinement of Theorem \ref{th:rv} occurs in three ways:
\begin{enumerate}
\item We lift the conclusion to the $B_v$ action on $R_v^0$, and not just the $\Aut(G,C)$-quotient $\R_v^0/\Aut(G,C)$.
\item We ``disentangle" the action of $B_v$ on $R_v^0$ from its action on $\hat{R}_v \setminus R_v$.
\item We replace $B_v$ with the pure braid subgroup $PB_{2k} \leq B_v \leq B_{2k}$.
\end{enumerate}
Precisely, let 
\[ \rho: B_v \to \Sym_{\Aut(G,C)}(R_v^0) \]
and
\[ \hat{\rho}: B_v \to \Sym(\hat{R}_v \setminus R_v) \]
be the pertinent permutation representations, and let
\[ F: B_v \to \Sym(2k) \]
be the forgetful map that only remembers how braid strands are permuted.

\begin{thm}
Let $G$ be a finite, nonabelian, simple group, with conjugacy class $C \subset G$. Then the image of $B_v$ under the joint homomorphism
\[ \rho \times \hat{\rho} \times F: B_v \to \Sym_{\Aut(G,C)}(R_v^0) \times \Sym(\hat{R}_v \setminus R_v) \times \Sym(2k) \]
eventually contains $\Rub_{\Aut(G,C)}(R_v^0)$.
\label{th:rvrefine}
\end{thm}

In other words, we can find a set of \emph{pure} braids that act by the full Rubik subgroup on $R_v^0$, while simultaneously acting trivially on the non-surjective maps in $\hat{R}_v \setminus R_v$.

\begin{proof}
Let $k$ be large enough for the conclusion of Lemma \ref{l:trans} to hold with $i=6$.  Then we see the image of $\rho$ contains $\Rub_{\Aut(G,C)}(R_v^0)$ by combining Theorem \ref{th:rv}, Lemma \ref{l:trans}, and Theorem \ref{th:rubik}.

Finite groups are always normally Zornian, so Lemma \ref{l:zorn} implies it is enough to show that $\Rub_{\Aut(G,C)}(R_v^0)$ does not have any simple quotients that are subquotients of $\Sym(\hat{R}_v \setminus R_v) \times \Sym(2k)$.  In fact, by Lemmas \ref{l:rubikquo} and \ref{l:prodquo}, it suffices to show that
\[ \Alt(R_v^0/\Aut(G,C)) \]
is not a subquotient of $\Sym(\hat{R}_v \setminus R_v)$ or $\Sym(2k)$.  Finally, by cardinality considerations, it is enough to show 
\[ |R_v^0/\Aut(G,C)| = \frac{R_v^0}{|\Aut(G,C)|} > |\hat{R}_v \setminus R_v| \]
eventually.

Elements of $\hat{R}_v$ are overwhelmingly likely to be surjective,
\[ \lim_{k \to \infty} \frac{|R_v|}{|\hat{R}_v|} \to 1 \]
and, in the large $k$ limit, the value of the Conway-Parker invariant is equidistributed in $R_v$. Thus
\[ \lim_{k \to \infty} \frac{|R_v^0|}{|\hat{R}_v|} \to \frac{1}{|H_2(G,C)|} \]
and we conclude $|R_v^0/\Aut(G,C)|$ grows faster than $|\hat{R}_v \setminus R_v|$ as a function of $k$.
\end{proof}

\chapter[CSPs for reversible circuits]{Constraint satisfaction problems for equivariant reversible circuits}
\label{ch:CSP}
We begin this chapter with an introduction to planar, reversible circuits in Section \ref{s:revcircuits}.  In Section \ref{s:csp} we define various complete decision and counting problems for reversible circuits.  None of the results are new, but we include proofs because we could not find references.  The main goal of this chapter is to introduce the problem $\#\ZSAT$ and to prove that $\#\ZSAT$ is $\shP$-complete via almost parsimonious reduction.  See Section \ref{s:zombies}, in particular Lemma \ref{l:zsat}. 

\section{Reversible circuits}
\label{s:revcircuits}

We will need two variations of the circuit model that still satisfy
\Thm{th:csat}:  Reversible circuits and planar circuits.

A \emph{reversible circuit} \cite{FT:logic} is a circuit $Z$ in which
every gate $\alpha:A^k \to A^k$ in the gate set $\Gamma$ is a bijection;
thus the evaluation of $Z$ is also a bijection.  We say that $\Gamma$ is
\emph{reversibly universal} if for any sufficiently large $n$, the gates of
$\Gamma$ in different positions generate either $\Alt(A^n)$ or $\Sym(A^n)$.
(If $|A|$ is even, then we cannot generate any odd permutations when $n$
is larger than the size of any one gate in $\Gamma$.)

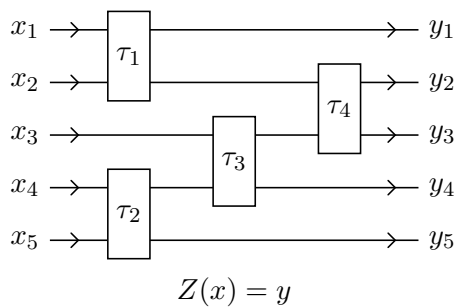
\begin{figure}[htb] \begin{center}
\begin{tikzpicture}[scale=.7,semithick,decoration={markings,
    mark=at position 0.08 with {\arrow{angle 90}},
    mark=at position 0.94 with {\arrow{angle 90}}}]
\draw (0,0) node[anchor=east] {$x_5$};
\draw[postaction={decorate}] (0,0) -- (7,0);
\draw (7,0) node[anchor=west] {$y_5$};
\draw (0,1) node[anchor=east] {$x_4$};
\draw[postaction={decorate}] (0,1) -- (7,1);
\draw (7,1) node[anchor=west] {$y_4$};
\draw (0,2) node[anchor=east] {$x_3$};
\draw[postaction={decorate}] (0,2) -- (7,2);
\draw (7,2) node[anchor=west] {$y_3$};
\draw (0,3) node[anchor=east] {$x_2$};
\draw[postaction={decorate}] (0,3) -- (7,3);
\draw (7,3) node[anchor=west] {$y_2$};
\draw (0,4) node[anchor=east] {$x_1$};
\draw[postaction={decorate}] (0,4) -- (7,4);
\draw (7,4) node[anchor=west] {$y_1$};
\draw[fill=white] (1.1,-.35) rectangle (1.9,1.35);
\draw (1.5,.5) node {$\tau_2$};
\draw[fill=white] (1.1,2.65) rectangle (1.9,4.35);
\draw (1.5,3.5) node {$\tau_1$};
\draw[fill=white] (3.1,.65) rectangle (3.9,2.35);
\draw (3.5,1.5) node {$\tau_3$};
\draw[fill=white] (5.1,1.65) rectangle (5.9,3.35);
\draw (5.5,2.5) node {$\tau_4$};
\draw (3.5,-.5) node[anchor=north] {$Z(x) = y$};
\end{tikzpicture} \end{center}
\caption{A planar, reversible circuit.}
\label{f:planar} \end{figure}

A circuit $Z$ is \emph{planar} if its graph is a planar graph placed in a
rectangle in the plane, with the inputs on one edge and the output on an
opposite edge.  The definition of a universal gate for general circuits
carries over to planar circuits; likewise the definition for reversible
circuits carries over to reversible planar circuits.  (See \Fig{f:planar}.)
We can make a circuit or a reversible circuit planar using reversible
$\SWAP$ gates that take $(a,b)$ to $(b,a)$.  Likewise, any universal gate
set becomes planar-universal by adding the $\SWAP$ gate.  Thus, the planar
circuit model is equivalent to the general circuit model.

The reduction from general circuits to reversible circuits is more
complicated.

\begin{lemma} Let $A$ be an alphabet for reversible circuits.
\begin{enumerate}
\item[1.] If $|A| \ge 3$, then $\Gamma = \Alt(A^2)$ is a universal set of
binary gates.
\item[2.] If $|A| = 2$, then $\Gamma = \Alt(A^3)$ is a universal set
of ternary gates.
\item[3.] If $|A|$ is even, then $\Sym(A^n) \subseteq \Alt(A^{n+1})$.
\end{enumerate}
\label{l:rev} \end{lemma}

Different versions of \Lem{l:rev} are standard in the reversible
circuit literature.  For instance, when $A = \Z/2$, the foundational
paper \cite{FT:logic} defines the Fredkin gate and the Toffoli gate,
each of which is universal together with the $\NOT$ gate.  Nonetheless,
we did not find a proof for all values of $|A|$, so we give one.

\begin{proof} Case 3 of the lemma is elementary, so we
concentrate on cases 1 and 2.  We will show by induction on $n$ that $\Gamma$
generates $\Alt(A^n)$.  The hypothesis hands us the base of induction $n=3$
when $|A| = 2$ and $n=2$ when $|A| \ge 3$.  So, we assume a larger value
of $n$ and we assume by induction that the case $n-1$ is already proven.

We consider the two subgroups in $\Alt(A^n)$ that are given by
$\Gamma$-circuits that act respectively on the left $n-1$ symbols or
the right $n-1$ symbols.  By induction, both subgroups are isomorphic to
$\Alt(A^{n-1})$, and we call them $\Alt(A^{n-1})_L$ $\Alt(A^{n-1})_R$.  They
in turn have subgroups $\Alt(A^{n-2})^{|A|}_L$ and $\Alt(A^{n-2})^{|A|}_R$
which are each isomorphic to $\Alt(A^{n-2})^{|A|}$ and each act on the
middle $n-2$ symbols; but in one case the choice of permutation $\alpha
\in \Alt(A^{n-2})$ depends on the leftmost symbol, while in the other case
it depends on the rightmost symbol.  By taking commutators between these
two subgroups, we obtain all permutations in $\Alt(\{a\} \times A^{n-2}
\times \{b\})$ for every pair of symbols $(a,b)$.  Moreover, we can repeat
this construction for every subset of $n-2$ symbols.  Since $n \ge 3$,
and since $n \ge 4$ when $|A| = 2$, we know that $|A^{n-2}| \ge 3$.
We can thus apply \Lem{l:alt} in the next section to the alternating
subgroups that we have obtained.
\end{proof}

\Lem{l:rev} motivates the definition of a canonical reversible gate set
$\Gamma$ for each alphabet $A$.  (Canonical in the sense that it is both
universal and constructed intrinsically from the finite set $A$.)  If $|A|$
is odd, then we let $\Gamma = \Alt(A^2)$.  If $|A| \ge 4$ is even, then
we let $\Gamma = \Sym(A^2)$.  Finally, if $|A| = 2$, then we let $\Gamma
= \Sym(A^3)$.  By \Lem{l:rev}, each of these gate sets is universal.
Moreover, each of these gate sets can be generated by any universal gate
set, possibly with the aid of an ancilla in the even case.

\section{Constraint satisfaction problems}
\label{s:csp}
In one version of reversible circuit satisfiability, we choose two subsets
$I,F \subseteq A$, interpreted as initialization and finalization constraint alphabets.
We define the problem $\RSAT_{A,I,F}$ as follows:  The input $x$ represents
a reversible circuit $Z_x$ of some width $n$ over the alphabet $A$,
with gates taken from some universal gate set $\Gamma$.   Then $Z_x$
is said to be satisfied if there is a circuit input $y \in I^n \subseteq
A^n$ such that $Z_x(y) \in F^n \subseteq A^n$.  The satisfiability problem
$\RSAT_{A,I,F}(x)$ asks whether such a witness $y$ exists, while as usual the
counting problem $\shRSAT_{A,I,F}(x)$ asks for the number of witnesses $y$.
Note that if either $F = A$ or $|I| = 1$, then $\RSAT_{A,I,F}$ is trivial.
Since $Z_x$ is a reversible circuit, it is just as easy to construct its
inverse $Z^{-1}_x$, so likewise $\RSAT_{A,I,F}$ is also trivial if either
$I = A$ or $|F| = 1$.

\begin{theorem} Consider $A$, $I$, $F$, and $\Gamma$ with $\Gamma$ universal
and $2 \le |I|, |F| < |A|$.  Then $\RSAT_{A,I,F}$ is Karp $\NP$-hard
and $\shRSAT_{A,I,F}$ is parsimoniously $\shP$-hard.
\label{th:rsat} \end{theorem}

\Thm{th:rsat} is also a standard result in reversible circuit theory,
but we again give a proof because we did not find one.

\begin{proof} We consider a sequence $\RSAT_i$ of versions of the reversible
circuit problem.  We describe the satisfiability version for each one, and
implicitly define the counting version $\shRSAT_i$ using the same predicate.

\begin{itemize}
\item $\RSAT_1$ uses the binary alphabet $A = \Z/2$ and does not have 
$I$ or $F$.  Instead, some of the input bits are set to $0$ while others
are variable, and the decision output of a circuit is simply the value
of the first bit.

\item $\RSAT_2$ also has $A = \Z/2$ with an even number of input and
output bits.  Half of the input bits and output bits are set to $0$,
while the others are variable.  A circuit $Z$ is satisfied by finding an
input/output pair $x$ and $Z(x)$ that satisfy the constraints.
 
\item $\RSAT_3$ is $\RSAT_{A,I,F}$ with $I$ and $F$ disjoint and $|A
\setminus (I \cup F)| \ge 2$.

\item $\RSAT_4$ is $\RSAT_{A,I,F}$ with the stated hypotheses of the theorem.
\end{itemize}
We claim parsimonious reductions from $\shCSAT$ to $\shRSAT_1$,
and from $\shRSAT_i$ to $\shRSAT_{i+1}$ for each $i$.

Step 1: We can reduce $\CSAT$ to $\RSAT_1$ through the method of \emph{gate
dilation} and ancillas.  Here an \emph{ancilla} is any fixed input to
the circuit that is used for scratch space; the definition of $\RSAT_1$
includes ancillas.  To define gate dilation, we can let $A$ be any alphabet
with the structure of an abelian group.  If $\alpha:A^k \to A$ is a gate,
then we can replace it with the reversible gate
\[ \beta:A^{k+1} \to A^{k+1} \qquad \beta(x,a) = (x,\alpha(x)+a), \]
where $x \in A^k$ is the input to $\alpha$ and $a \in A$ is an ancilla
which is set to $a=0$ when $\beta$ replaces $\alpha$.  The gate $\beta$
is called a \emph{reversible dilation} of $\alpha$.  We can similarly
replace every irreversible $\COPY$ gate with the reversible gate
\[ \COPY:A^2 \to A^2 \qquad \COPY(x,a) = (x,x+a), \]
where again $a$ is an ancilla set to $a=0$.  Dilations also leave extra
output symbols, but under the satisfiability rule of $\RSAT_1$, we
can ignore them.

In the Boolean case $A = \Z/2$, the reversible $\COPY$ gate is denoted
$\CNOT$ (controlled $\NOT$), while the dilation of $\AND$ is denoted $\CCNOT$
(doubly controlled $\NOT$) and is called the Toffoli gate.  We can add to
this the uncontrolled $\NOT$ gate
\[ \NOT(x) = x+1. \]
These three gates are clearly enough to dilate irreversible Boolean circuits.
(They are also a universal gate set for reversible computation.)

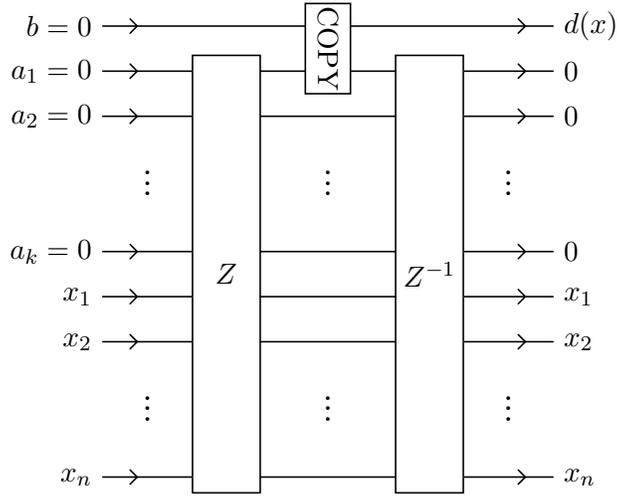
\begin{figure}[htb] \begin{center}
\begin{tikzpicture}[scale=.6,semithick,decoration={markings,
    mark=at position 0.08 with {\arrow{angle 90}},
    mark=at position 0.94 with {\arrow{angle 90}}}]
\draw (0,10) node[anchor=east] {$b=0$};
\draw[postaction={decorate}] (0,10) -- (10,10);
\draw (10,10) node[anchor=west] {$d(x)$};
\draw (0,9) node[anchor=east] {$a_1=0$};
\draw[postaction={decorate}] (0,9) -- (10,9);
\draw (10,9) node[anchor=west] {$0$};
\draw (0,8) node[anchor=east] {$a_2=0$};
\draw[postaction={decorate}] (0,8) -- (10,8);
\draw (10,8) node[anchor=west] {$0$};
\draw (0,5) node[anchor=east] {$a_k=0$};
\draw[postaction={decorate}] (0,5) -- (10,5);
\draw (10,5) node[anchor=west] {$0$};
\draw (0,4) node[anchor=east] {$x_1$};
\draw[postaction={decorate}] (0,4) -- (10,4);
\draw (10,4) node[anchor=west] {$x_1$};
\draw (0,3) node[anchor=east] {$x_2$};
\draw[postaction={decorate}] (0,3) -- (10,3);
\draw (10,3) node[anchor=west] {$x_2$};
\draw (0,0) node[anchor=east] {$x_n$};
\draw[postaction={decorate}] (0,0) -- (10,0);
\draw (10,0) node[anchor=west] {$x_n$};
\draw (1,6.7) node {\Large $\vdots$};
\draw (1,1.7) node {\Large $\vdots$};
\draw[fill=white] (2,-.35) rectangle (3.5,9.35);
\draw (2.75,4.5) node {$Z$};
\draw[fill=white] (4.5,8.5) rectangle (5.5,10.5);
\draw (5,9.5) node[rotate=270] {$\COPY$};
\draw (5,6.7) node {\Large $\vdots$};
\draw (5,1.7) node {\Large $\vdots$};
\draw[fill=white] (6.5,-.35) rectangle (8,9.35);
\draw (7.25,4.5) node {$Z^{-1}$};
\draw (9,6.7) node {\Large $\vdots$};
\draw (9,1.7) node {\Large $\vdots$};
\end{tikzpicture}
\end{center}
\caption{Using uncomputation to reset ancilla values.}
\label{f:uncomp} \end{figure}

Step 2: We can reduce $\RSAT_1$ to $\RSAT_2$ using the method of
uncomputation.  Suppose that a circuit $Z$ in the $\RSAT_1$ problem has
an $n$-bit variable input register $x = (x_1,x_2,\dots,x_n)$ and a $k$-bit
ancilla register $a = (a_1,a_2,\dots,a_k)$.  Suppose that $Z$ calculates
decision output $d(x)$ in the $a_1$ position (when $a_1 = 0$ since it is
an ancilla).   Then we can make a new circuit $Z_1$ with the same $x$
and $a$ and one additional ancilla bit $b$, defined by applying $Z$,
then copying the output to $b$ and negating $b$, then applying $Z^{-1}$,
as in \Fig{f:uncomp}.  If $n = k+1$, then $Z_1$ is a reduction of $Z$ from
$\RSAT_1$ to $\RSAT_2$.  If $n > k+1$, then we can pad $Z_1$ with $n-k-1$
more ancillas and do nothing with them to produce a padded circuit $Z_2$.
If $n < k+1$, then we can pad $Z_1$ with $k+1-n$ junk input bits, and at
the end of $Z_1$, copy of these junk inputs to $k+1-n$ of the first $k$
ancillas; again to produce $Z_2$.  (Note that $k+1-n \le k$ since we can
assume that $Z$ has at least one variable input bit.)  In either of these
cases, $Z_2$ is a reduction of $Z$ from $\RSAT_1$ to $\RSAT_2$.

Step 3: We can reduce $\RSAT_2$ to $\RSAT_3$ by grouping symbols and
permuting alphabets.   As a first step, let $A_1 = \Z/2 \times \Z/2$
with $I_1 = F_1 = \{(0,0),(1,0)\}$.  Then we can reduce $\RSAT_2$ to
$\RSAT_{A_1,I_1,F_1}$ by pairing each of input or output bit with an
ancilla; we can express each ternary gate over $\Z/2$ in terms of binary
gates over $A_1$.  Now let $A_2$ be any alphabet with disjoint $I_2$ and
$F_2$, and with at least two symbols not in $I_2$ or $F_2$.  Then we can
embed $(A_1,I_1,F_1)$ into $(A_2,I_2,F_2)$ arbitrarily, and extend any gate
$\alpha:A_1^k \to A_1^k$ (with $k \in \{1,2\}$, say) arbitrarily to a gate
$\beta:A_2^k \to A_2^k$ which is specifically an even permutation.  This
reduces $\RSAT_2 = \RSAT_{A_1,I_1,F_1}$ to $\RSAT_3 = \RSAT_{A_2,I_2,F_2}$.

Step 4: Finally, $(A_3,I_3,F_3)$ is an alphabet that is not of our choosing,
and we wish to reduce $\RSAT_3 = \RSAT_{A_2,I_2,F_2}$ to $\RSAT_4 =
\RSAT_{A_3,I_3,F_3}$.  We choose $k$ such that
\[ |A_3|^k \ge |I_3|^k + |F_3|^k + 2. \]
We then let $A_2 = A_3^k$ and $I_2 = I_3^k$, and we choose $F_2 \subseteq A_2
\setminus I_2$ with $|F_2| = |F_3|^k$.  A circuit in $\RSAT_{A_2,I_2,F_2}$
can now be reduced to a circuit in $\RSAT_{A_3,I_3,F_3}$ by grouping
together $k$ symbols in $A_3$ to make a symbol in $A_2$.   Since $I_2 =
I_3^k$, the initialization is the same.  At the end of the circuit, we
convert finalization in $F_2$ to finalization in $F_3^k$ with some unary
permutation of the symbols in $A_2$.
\end{proof}

\section{Circuits with a $J$-set alphabet}
\label{s:zombies}

Let $J$ be a non-trivial finite group and let $A$ be an alphabet which
is a $J$-set with a single fixed point $z$, the \emph{zombie symbol},
and otherwise with free orbits.  We choose two $J$-invariant alphabets
$I,F \subsetneq A \setminus \{z\}$, and we assume that
\begin{equation}
|I|,|F| \ge 2|J| \qquad I \ne F \qquad |A| \ge 2|I \cup F| + 3|J| + 1.
\label{e:zineq}
\end{equation}
(The second and third conditions are for convenience rather than necessity.)
With these parameters, we define a planar circuit model that we denote
$\ZSAT_{J,A,I,F}$ that is the same as $\RSAT_{A,I \cup \{z\},F \cup
\{z\}}$ as defined in \Sec{s:circuits}, \emph{except} that the gate set
is $\Rub_J(A^2)$.  This gate set is not universal in the sense of $\RSAT$
because every gate and thus every circuit is $J$-equivariant.  (One can
show that it is universal for $J$-equivariant circuits by establishing an
analogue of \Lem{l:rev} with the aid of \Thm{th:rubik}, but we will not need
this.)  More explicitly, in the $\ZSAT$ model we consider $J$-equivariant
planar circuits $Z$ that are composed of binary gates in $\Rub_J(A^2)$,
and satisfiability is defined by the equation $Z(x) = y$ with $x \in
(I \cup \{z\})^n$ and $y \in (F \cup \{z\})^n$.

\begin{lemma} $\shZSAT_{J,A,I,F}$ is almost parsimoniously $\shP$-complete.
More precisely, if $c \in \shP$, then there is an $f \in \FP$ such that
\begin{equation}\shZSAT_{J,A,I,F}(f(x)) = |J|c(x)+1. \label{e:zk}\end{equation}
\label{l:zsat} \end{lemma}

Equation \ref{e:zk} has the same form as equation \ref{e:hq}, and for
an equivalent reason: The input $(z,z,\dots,z)$ trivially satisfies any
$\ZSAT$ circuit (necessarily at both ends), while $J$ acts freely on the
set of other circuit solutions.

\begin{proof} We take the convention that $A$ is a left $J$-set.  We choose
a subset $A_0 \subsetneq A$ that has one representative from each free
$J$-orbit of $A$.  (In other words, $A_0$ is a section of the free orbits.)

We say that a data state $(a_1,a_2,\dots,a_n)$ of a $\ZSAT$ circuit of
width $n$ is \emph{aligned} if it has no zombie symbols and if there is a
single element $j \in J$ such that $ja_i \in A_0$ for all $i$.   The idea
of the proof is to keep zombie symbols unchanged (which is why they are
called zombies) and preserve alignment in the main reduction, and then add
a postcomputation that converts zombies and misaligned symbols into warning
symbols in a separate warning alphabet.  The postcomputation cannot work
if all symbols are zombies, but it can work in all other cases.

More precisely, we let $W \subseteq A \setminus (I \cup F \cup \{z\})$
be a $J$-invariant subalphabet of size $|I \cup F|+2|J|$ which we call the
\emph{warning alphabet}, and we distinguish two symbols $z_1,z_2 \in W$ in
distinct orbits.  Using \Thm{th:rsat} as a lemma, we will reduce a circuit
$\overline{Z}$ in the planar, reversible circuit model $\RSAT_{(I \cup F)/J,I/J,F/J}$
with binary gates to a circuit $Z$ in $\ZSAT_{J,A,I,F}$.  To describe the
reduction, we identify each element of $(I \cup F)/J$ with its lift in $A_0$.

We let $Z$ have the same width $n$ as $\overline{Z}$.  To make $Z$, we convert each
binary gate $\gamma$ of the circuit $\overline{Z}$ in $\RSAT_{(I \cup F)/J,I/J,F/J}$
to a gate $\delta$ in $\ZSAT_{J,A,I,F}$ in sequence.  After all of these
gates, $Z$ will also have a postcomputation stage.  Given $\gamma$,
we define $\delta$ as follows:
\begin{enumerate}
\item Of necessity,
\[ \delta(z,z) = (z,z). \]
\item If $a \in I \cup F$, then
\[ \delta(z,a) = (z,a) \qquad \delta(a,z) = (a,z). \]
\item If $a_1,a_2 \in (I \cup F) \cap A_0$, $g_1,g_2 \in J$, and 
\[ \gamma(a_1,a_2) = (b_1,b_2), \]
then
\[ \delta(g_1a_1,g_2a_2) = (g_1b_1,g_2b_2). \]
\item We extend $\delta$ to the rest of $A^2$ so that $\delta \in
\Rub_J(A^2)$.
\end{enumerate}
By cases 1 and 2, zombie symbols stay unchanged.  Cases 1, 2, and 3 together
keep the computation within the subalphabet $I \cup F \cup \{z\}$, while
case 3 preserves alignments, as well as misalignments.

The postcomputation uses a gate $\alpha:A^2 \to A^2$ such that:
\begin{enumerate}
\item Of necessity,
\[ \alpha(z,z) = (z,z). \]
\item If $a \in I \cap A_0$, then
\[ \alpha(z,a) = (z_1,a) \qquad \alpha(a,z) = (z_2,a) \]
\item If $a_1,a_2 \in I \cup F$ are misaligned, then
\[ \alpha(a_1,a_2) = (\beta(a_1),a_2) \]
for some $J$-equivariant bijection 
\[ \beta: I \cup F \to W \setminus (Jz_1 \cup Jz_2). \]
\item If $a_1,a_2 \in I \cup F$ are aligned, then
\[ \alpha(a_1,a_2) = (a_1,a_2). \]
\item We extend $\alpha$ to the rest of $A^2$ so that $\alpha \in
\Rub_J(A^2)$.
\end{enumerate}
We apply this gate $\alpha$ to each adjacent pair of symbols $(a_i,a_{i+1})$
for $i$ ranging from $1$ to $n-1$ in order.  The final effect is that,
if some (but not all) of the input symbols are zombies, or if any two
symbols are misaligned, then the postcomputation in $Z$ creates symbols
in the warning alphabet $W$.

Any input to $Z$ with either zombies or misaligned symbols cannot finalize,
since the main computation preserves these syndromes and the postcomputation
then produces warning symbols that do not finalize.  The only spurious input
that finalizes is the all-zombies state $(z,z,\dots,z)$, and otherwise
each input that $\overline{Z}$ accepts yields a single aligned $J$-orbit.  Thus we
obtain the relation
\[ \#Z = |J|\#\overline{Z} + 1 \]
between the number of inputs that satisfy $\overline{Z}$ and the number that satisfy
$Z$, as desired.
\end{proof}

\chapter{Reductions}
\label{ch:reductions}
\section{Reduction to homology spheres}
\label{s:homologyreduction}

\subsection{Mapping class gadgets}
\label{ss:mcg}

\begin{figure*}[t] \begin{center}
\begin{tikzpicture}[scale=.4,semithick]
\foreach \x in {0,6,12,24} {
\begin{scope}[shift={(\x,0)}]
\draw (1,0) .. controls (1,1) and (0,2) .. (0,3)
    .. controls (0,4) and (1,5) .. (1,7)
    .. controls (1,9) and (0,10) .. (0,11)
    .. controls (0,12) and (1,13) .. (2,13)
    .. controls (3,13) and (4,12) .. (4,11)
    .. controls (4,10) and (3,9) .. (3,7)
    .. controls (3,5) and (4,4) .. (4,3)
    .. controls (4,2) and (3,1) .. (3,0);
\draw (2,7.75) node {\scalebox{3}{$\vdots$}};
\draw (1.2,3) .. controls (1.6,3.4) and (2.4,3.4) .. (2.8,3);
\draw (1.1,3.1) -- (1.2,3) .. controls (1.6,2.6) and (2.4,2.6) .. (2.8,3)
    -- (2.9,3.1);
\draw[darkblue,thick] (2,3) ellipse (1.5 and .7);
\draw[darkred,thick,dashed] (1.2,3) arc (0:180:.6 and .3);
\draw[darkred,thick] (0,3) arc (-180:0:.6 and .3);
\draw (1.2,11) .. controls (1.6,11.4) and (2.4,11.4) .. (2.8,11);
\draw (1.1,11.1) -- (1.2,11) .. controls (1.6,10.6) and (2.4,10.6) .. (2.8,11)
    -- (2.9,11.1);
\draw[darkblue,thick] (2,11) ellipse (1.5 and .7);
\draw[darkred,thick,dashed] (1.2,11) arc (0:180:.6 and .3);
\draw[darkred,thick] (0,11) arc (-180:0:.6 and .3);
\draw[thick,dashed] (3,0) arc (0:180:1 and .5);
\draw[thick] (1,0) arc (-180:0:1 and .5);
\fill[medgreen] (2,-.5) circle (.175);
\end{scope} }
\draw (1,0) .. controls (1,-2) and (-1,-2) .. (-1,-4)
    .. controls (-1,-7) and (5,-8) .. (14,-8)
    .. controls (23,-8) and (29,-7) .. (29,-4)
    .. controls (29,-2) and (27,-1) .. (27,0);
\draw (3,0) .. controls (3,-2) and (7,-2) .. (7,0);
\draw (9,0) .. controls (9,-2) and (13,-2) .. (13,0);
\draw (15,0) .. controls (15,-1) and (16,-1.5) .. (17,-1.5) -- (18,-1.5);
\draw (25,0) .. controls (25,-1) and (24,-1.5) .. (23,-1.5) -- (22,-1.5);
\draw (20.1,-1.5) node {\scalebox{3}{$\cdots$}};
\draw[thick] (11,-1.5) .. controls (13,-1.5) and (9,-4) .. (5,-4)
    .. controls (1,-4) and (-1,-3) .. (0,-2);
\draw[thick,dashed] (0,-2) .. controls (1,-1) and (2,-2.5) .. (5,-2.5)
    .. controls (8,-2.5) and (9,-1.5) .. (11,-1.5);
\draw[thick] (5,-1.5) .. controls (3,-1.5) and (7,-4) .. (11,-4)
    .. controls (15,-4) and (19,-1.5) .. (17,-1.5);
\draw[thick,dashed] (5,-1.5) .. controls (6,-1.5) and (8,-2.5) .. (11,-2.5)
    .. controls (14,-2.5) and (16,-1.5) .. (17,-1.5);
\fill[medgreen] (14,-6) circle (.175);
\fill[medgreen] (11,-4) circle (.175);
\fill[medgreen] (5,-4) circle (.175);
\draw[medgreen] (14,-6) .. controls (10,-6) and (5,-6) .. (5,-4);
\draw[medgreen] (5,-4) .. controls (5,-2.5) and (2,-2) .. (2,-.5);
\draw[medgreen] (5,-4) .. controls (5,-2.5) and (8,-2) .. (8,-.5);
\draw[medgreen] (14,-6) .. controls (12,-6) and (11,-5) .. (11,-4);
\draw[medgreen] (11,-4) .. controls (11,-2.5) and (8,-2) .. (8,-.5);
\draw[medgreen] (11,-4) .. controls (11,-2.5) and (14,-2) .. (14,-.5);
\draw[medgreen] (14,-6) -- (18,-6);
\draw[medgreen] (20.1,-6) node {\scalebox{3}{$\cdots$}};
\draw[medgreen] (22,-6) .. controls (26,-6) and (26,-2) .. (26,-.5);
\draw[medgreen] (14,-6.2) node[anchor=north] {$p_0$};
\draw (2,14) node {$(\Sigma_{g}^1)_1$};
\draw (8,14) node {$(\Sigma_{g}^1)_2$};
\draw (14,14) node {$(\Sigma_{g}^1)_3$};
\draw (26,14) node {$(\Sigma_{g}^1)_n$};
\draw[darkblue] (29.5,11) node {$(H_g)_{F,n}$};
\draw[darkred] (22.5,11) node {$(H_g)_{I,n}$};
\draw (-3,3) node {\scalebox{1.5}{$\Sigma_{ng}$}};
\draw (3.5,-4) node[anchor=north east] {$(\Sigma_{2g}^1)_{1,2}$};
\draw (14,-3.5) node[anchor=north west] {$(\Sigma_{2g}^1)_{2,3}$};
\end{tikzpicture} \end{center}
\caption{The Heegaard surface $\Sigma_{ng}$ with disjoint subsurfaces
$(\Sigma_{g}^1)_i$.  Circles that contract in $(H_g)_{I,i} \subset (H_{ng})_I$
are in red, while circles that contract in $(H_g)_{F,i} \subset (H_{ng})_F$
are in blue.  The subsurfaces $(\Sigma_{2g}^1)_{i,i+1}$ are also indicated.
The system of basepoints and connecting paths is in green.}
\label{f:heegaard} \end{figure*}
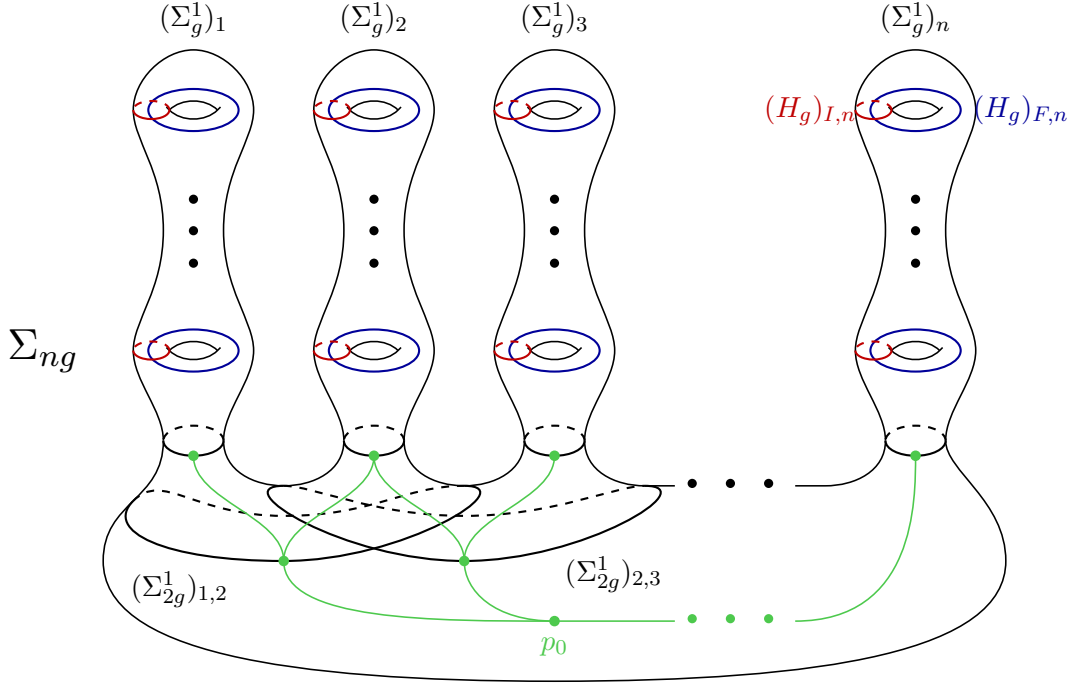

In this subsection and the next one, we will finish the proof of
\Thm{th:main1}.  We want to convert a suitable $\ZSAT$ circuit $Z$ of width
$n$ to a homology 3-sphere $M$.  To this end, we choose some sufficiently
large $g$ that depends only on the group $G$, and we let $\Sigma_{ng}$
be a Heegaard surface of $M$.  This Heegaard surface will be decorated in
various ways that we summarize in \Fig{f:heegaard}.  We use the additional
notation that $\Sigma_{g}^b$ is a surface of genus $g$ with $b$ boundary
circles, with a basepoint on one of its circles.  We give $\Sigma_{g}^b$
the representation set
\[ \hR_{g,b} \defeq \{f:\pi_1(\Sigma_{g}^b) \to G\}. \]
We let $\MCG_*(\Sigma_{g}^b)$ be the relative mapping class group (that
fixes $\partial \Sigma_{g}^b$); it naturally acts on $\hR_{g,b}$.

We attach two handlebodies $(H_{ng})_I$ and $(H_{ng})_F$ to $\Sigma_{ng}$
so that
\[ (H_{ng})_I \cup (H_{ng})_F \cong S^3.\]
Although an actual sphere $S^3$ is not an interesting homology sphere
for our purposes, our goal is to construct a homeomorphism $\phi \in
\MCG_*(\Sigma_{ng})$ so that
\[ M \defeq (H_{ng})_I \sqcup_\phi (H_{ng})_F \]
is the 3-manifold that we will produce to prove \Thm{th:main1}.  (We could
let $\phi$ be an element of the unpointed mapping class group here, but
it is convenient to keep the basepoint.)

\begin{figure}[t]\begin{center} \begin{tabular}{l|l}
$\ZSAT_{J,A,I,F}$ & $H(M,G)$ \\ \hline
$n$-symbol memory & Heegaard surface $\Sigma_{ng}$ \\
1-symbol memory & computational subsurface $\Sigma_{g}^1$ \\
binary gate & element of $\MCG_*(\Sigma_{2g}^1)$ \\
circuit: $Z$ & mapping class $\phi \in \MCG_*(\Sigma_{ng})$ \\
alphabet: $A$ & homomorphisms $\pi_1(\Sigma_g^1) \to G$\\
alphabet symmetry: $J$ & automorphisms $\Aut(G)$ \\
zombie symbol: $z \in A$ & trivial map $z:\pi_1(\Sigma_g^1) \to G$ \\
memory state: $x\in A^n$ & homomorphism $f: \pi_1(\Sigma_{ng}) \to G$ \\
initialization: $x \in (I \cup \{z\})^n$ & $f$ extends to $\pi_1((H_{ng})_I)$ \\
finalization: $y \in (F \cup \{z\})^n$ & $f$ extends to $\pi_1((H_{ng})_F)$ \\
solution: $Z(x) = y$ & homomorphism $f:\pi_1(M) \to G$
\end{tabular} \end{center}
\caption{A correspondence between $\ZSAT$ and $H(M,G)$.}
\label{f:corresp} \end{figure}

We identify $n$ disjoint subsurfaces
\[ (\Sigma_{g}^1)_1,(\Sigma_{g}^1)_2,\dots,
    (\Sigma_{g}^1)_n \subseteq \Sigma_{ng} \]
which are each homeomorphic to $\Sigma_{g}^b$.  The handlebodies
$(H_{ng})_I$ and $(H_{ng})_F$ likewise have sub-handlebodies $(H_g)_{I,i}$
and $(H_g)_{F,i}$ of genus $g$ associated with $(\Sigma_{g}^1)_i$ and
positioned so that
\[ (H_g)_{I,i} \cup (H_g)_{F,i} \cong B^3. \]
We also choose another set of subsurfaces
\[ (\Sigma_{2g}^1)_{1,2},(\Sigma_{2g}^1)_{2,3},\dots,
    (\Sigma_{2g}^1)_{n-1,n} \subseteq \Sigma_{ng} \]
such that
\[ (\Sigma_{g}^1)_i,(\Sigma_{g}^1)_{i+1} \subseteq (\Sigma_{2g}^1)_{i,i+1}. \]
Finally we mark basepoints for each subsurface $(\Sigma_{g}^1)_i$ and
$(\Sigma_{2g}^1)_{i,i+1}$, and one more basepoint $p_0 \in \Sigma_{ng}$,
and we mark a set of connecting paths as indicated in \Fig{f:heegaard}.

The circuit conversion is summarized in \Fig{f:corresp}.  We will use the
computational alphabet
\[ A \defeq R^0_g \cup \{z\} \subseteq \hR_g \subseteq \hR_{g,1}, \]
where $z:\pi_1(\Sigma_g) \to G$ is (as first mentioned in \Sec{ss:sketch1})
the trivial homomorphism and the zombie symbol, and the inclusion $\hR_g
\subseteq \hR_{g,1}$ comes from the inclusion of surfaces $\Sigma_{g}^1
\subseteq \Sigma_g$.  We let $J=\Aut(G)$ be the finite group acting on $A$.
Note that $\hR^0_g$ is precisely the subset of
$\hR_{g,1}$ consisting of homomorphisms
\[ f:\pi_1(\Sigma_{g}^1) \to G \]
that are trivial on the peripheral subgroup $\pi_1(\partial \Sigma_{g}^1)$.

Each subsurface $(\Sigma_{g}^1)_i$ is interpreted as the ``memory unit" of
a single symbol $x_i \in A$.  Using the connecting paths in $\Sigma_{ng}$
between the basepoints of its subsurfaces, and since each $x_i$ is trivial
on $\pi_1(\partial \Sigma_{g}^1)$, a data register
\[ x = (x_1,x_2,\dots,x_n) \in A^n \]
combines to form a homomorphism
\[ f:\pi_1(\Sigma_{ng}) \to G. \]
In particular, if $x \ne (z,z,\dots,z)$, then $f \in R_{ng}$.  In other
words, $f$ is surjective in this circumstance because one of its components
$x_i$ is already surjective.  (Note that the converse is not true: we can
easily make a surjective $f$ whose restriction to each $(\Sigma_{g}^1)_i$
is far from surjective.)

For every subgroup $H \le G$, we define $I(H)$ to be the set of surjections
\[ x:\pi_1(\Sigma_{g}^1) \onto H \]
that come from a homomorphism 
\[ x:\pi_1((H_g)_I) \onto H. \]
We define $F(H)$ in the same way using $(H_g)_F$. A priori we know that
$I(H), F(H) \subseteq R_{g,1}(H)$.  This inclusion can be sharpened in
two significant respects.

\begin{lemma} The sets $I(H)$ and $F(H)$ are subsets of $R^0_g(H)$.
If $H$ is non-trivial, then they are disjoint.
\label{l:ifr0} \end{lemma}

\begin{proof} First, since $\partial \Sigma_{g}^1$ bounds a disk in
$(H_g)_I$, we obtain that $I(H), F(H) \subseteq R_g(H)$.  Second,
since any $x$ in $I(H)$ or $F(H)$ extends to a handlebody, the cycle
$x_*([\Sigma_g])$ is null-homologous in $BG$ and therefore $\sch(x) =
0.$  Third, since $(H_g)_I \cup (H_g)_F \cong B^3$ is simply connected, a
surjective homomorphism $x \in R_g(H)$ cannot extend to both handlebodies if
$H$ is non-trivial.  Therefore $I(H)$ and $F(H)$ are disjoint in this case.
\end{proof}

The gadgets that serve as binary gates are mapping class elements $\alpha
\in \MCG_*((\Sigma_{2g}^1)_{i,i+1})$ that act on two adjacent memory units
$(\Sigma_{g}^1)_i$ and $(\Sigma_{g}^1)_{i+1}$.  We summarize the effect
of the local subsurface inclusions on representation sets.  In order to
state it conveniently, if $X$ and $Y$ are two pointed spaces, we define a
modified wedge $X \vee_\lambda Y$, where $\lambda$ is a connecting path
between the basepoint of $X$ and the basepoint of $Y$.  \Fig{f:pinched}
shows a surjection from $\Sigma_{2g}$ to $\Sigma_g \vee_\lambda \Sigma_g$,
while \Fig{f:heegaard} has copies of $\Sigma_{g}^1 \vee_\lambda \Sigma_{g}^1$
(which has a similar surjection from $\Sigma_{2g}^1$).

\begin{figure*}[htb] \begin{center}
\begin{tikzpicture}[semithick,scale=.4]
\draw (0,-1) .. controls (1,-1) and (2,-2) .. (3,-2)
    .. controls (4,-2) and (5,-1) .. (7,-1)
    .. controls (9,-1) and (10,-2) .. (11,-2)
    .. controls (12,-2) and (13,-1) .. (13,0)
    .. controls (13,1) and (12,2) .. (11,2)
    .. controls (10,2) and (9,1) .. (7,1)
    .. controls (5,1) and (4,2) .. (3,2)
    .. controls (2,2) and (1,1) .. (0,1)
    .. controls (-1,1) and (-2,2) .. (-3,2)
    .. controls (-4,2) and (-5,1) .. (-7,1)
    .. controls (-9,1) and (-10,2) .. (-11,2)
    .. controls (-12,2) and (-13,1) .. (-13,0)
    .. controls (-13,-1) and (-12,-2) .. (-11,-2)
    .. controls (-10,-2) and (-9,-1) .. (-7,-1)
    .. controls (-5,-1) and (-4,-2) .. (-3,-2)
    .. controls (-2,-2) and (-1,-1) .. (0,-1);
\draw[thick] (0,1) arc (90:270:.5 and 1);
\draw[thick,dashed] (0,-1) arc (-90:90:.5 and 1);
\foreach \x in {-11,-3,3,11} {
\begin{scope}[shift={(\x,0)}]
\draw (-.8,0) .. controls (-.4,.4) and (.4,.4) .. (.8,0);
\draw (-.9,.1) -- (-.8,0) .. controls (-.4,-.4) and (.4,-.4) .. (.8,0)
    -- (.9,.1);
\end{scope} }
\draw (7.1,-.05) node {\scalebox{3}{$\cdots$}};
\draw (-6.9,-.05) node {\scalebox{3}{$\cdots$}};
\foreach \x in {-8,8} {
\begin{scope}[shift={(\x,-9)}]
\draw (-6,0) .. controls (-6,-1) and (-5,-2) .. (-4,-2)
    .. controls (-3,-2) and (-2,-1) .. (0,-1)
    .. controls (2,-1) and (3,-2) .. (4,-2)
    .. controls (5,-2) and (6,-1) .. (6,0)
    .. controls (6,1) and (5,2) .. (4,2)
    .. controls (3,2) and (2,1) .. (0,1)
    .. controls (-2,1) and (-3,2) .. (-4,2)
    .. controls (-5,2) and (-6,1) .. (-6,0);
\draw (.1,-.05) node {\scalebox{3}{$\cdots$}};
\end{scope} }
\foreach \x in {-12,-4,4,12} {
\begin{scope}[shift={(\x,-9)}]
\draw (-.8,0) .. controls (-.4,.4) and (.4,.4) .. (.8,0);
\draw (-.9,.1) -- (-.8,0) .. controls (-.4,-.4) and (.4,-.4) .. (.8,0)
    -- (.9,.1);
\end{scope} }
\draw (-2,-9) -- (2,-9); \draw (0,-9) node[anchor=north] {$\lambda$};
\fill (-2,-9) circle (.175); \fill (2,-9) circle (.175);
\draw (-15,0) node[anchor=east] {\scalebox{1.25}{$\Sigma_{2g}$}};
\draw (-15,-9) node[anchor=east]
    {\scalebox{1.25}{$\Sigma_g \vee_\lambda \Sigma_g$}};
\draw (0,-4.5) node {\scalebox{2}{\rotatebox{270}{
    $\relbar\joinrel\twoheadrightarrow$}}};
\end{tikzpicture}
\end{center}
\caption{From $\Sigma_{2g}$ to $\Sigma_g \vee_\lambda \Sigma_g$}
\label{f:pinched} \end{figure*}
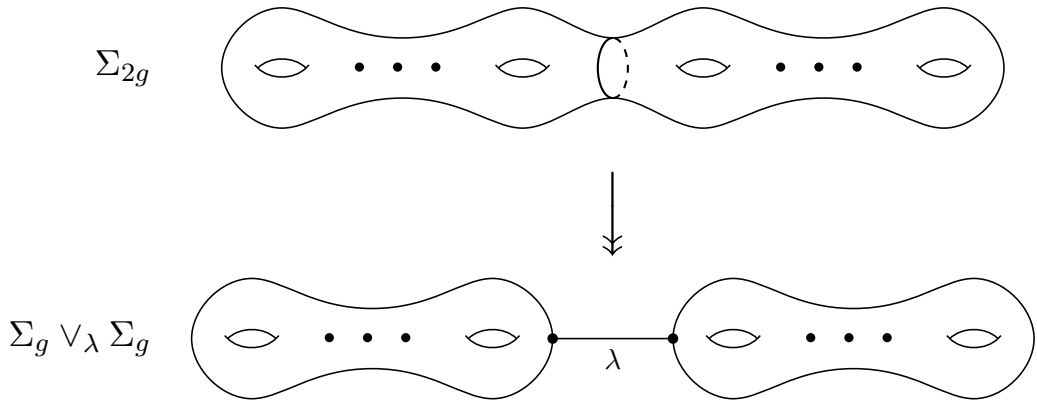

\begin{lemma} The inclusions and surjections
\[ \begin{array}{ccc}
\Sigma_{2g}^1 & \subseteq & \Sigma_{2g} \\[-1.5ex]
\vonto  & & \vonto \\[2ex]
\Sigma_{g}^1 \vee_\lambda \Sigma_{g}^1 & \subseteq
    & \Sigma_g \vee_\lambda \Sigma_g 
\end{array} \]
yield the inclusions
\begin{equation} \begin{array}{ccccccc}
\hR_{2g,1} & \supseteq & \hR_{2g} & \supseteq & R_{2g}
    & \supseteq & R^0_{2g} \\[.5ex]
\veq & & \vsubseteq & & \vsubseteq & & \vsubseteq \\
\hR_{g,1} \times \hR_{g,1} & \supseteq & \hR_g \times \hR_g & \supseteq &
    R_g \times R_g & \supseteq & R^0_g \times R^0_g
\end{array}.\label{e:include1} \end{equation}
For every pair of subgroups $H_1, H_2 \le G$ that generate $H \le G$,
they also yield
\begin{equation} R^0_g(H_1) \times R^0_g(H_2)
    \subseteq R^0_{2g}(H). \label{e:include2}\end{equation}
Finally, they yield
\begin{equation}A \times A \subseteq R^0_{2g}
    \cup \{z_{2g}\}, \label{e:include3} \end{equation}
where $z_g \in R_g$ is the trivial map in genus $g$ and $z_{2g} = (z_g,z_g)$.
 \label{l:include} \end{lemma}

\begin{proof} The horizontal inclusions are all addressed above; the real
issue is the vertical inclusions and equalities.   We consider the vertical
inclusions from left to right in diagram \eqref{e:include1}.  The surjection
\[ \sigma_1:\Sigma_{2g}^1 \onto \Sigma_{g}^1 \vee_\lambda \Sigma_{g}^1 \]
is an isomorphism of $\pi_1$, while the surjection
\[ \sigma_0:\Sigma_{2g} \onto \Sigma_g \vee_\lambda \Sigma_g \]
is a surjection in $\pi_1$.  This implies the first two vertical
relations.  Then, if two homomorphisms
\[ f_1,f_2:\pi_1(\Sigma_g) \onto G \]
are each surjective, then they are certainly jointly surjective; this
implies the third relation.  Finally, the surjection $\sigma_0$
yields the formula 
\begin{equation} \sch((f_1,f_2)) = \sch(f_1) + \sch(f_2). \label{e:sch}\end{equation}
The reason is that the image $\sigma_0([\Sigma_{2g}])$ of the fundamental
class of $\Sigma_{2g}$ is the sum of the fundamental classes of the 
two $\Sigma_g$ components.  This yields the fourth, leftmost inclusion
because equation \eqref{e:sch} then reduces to $0 = 0+0$.

To treat \eqref{e:include2}, we claim that if $\sch_{K_i}(f_i) = 0$,
then $\sch_K(f_i) = 0$.  This follows from the fact that each map from
$\Sigma_g$ to the classifying space $BH_i$ and $BH$ forms a commutative
triangle with the map $BH_i \to BH$.  With this remark, inclusion
\eqref{e:include2} can be argued in the same way as the inclusion $R^0_g
\times R^0_g \subseteq R^0_{2g}$.

Finally for inclusion \eqref{e:include3}, recall that $A = R^0_g \cup
\{z_g\}$, and that $z_{2g} = (z_g,z_g)$ since in each case $z$ is the
trivial homomorphism.  The inclusions
\[ R^0_g \times \{z_g\}, \{z_g\} \times R^0_g \subseteq R^0_{2g} \]
can be argued the same way as before: Given the two homomorphisms $f_1,f_2$,
even if one of them is the trivial homomorphism $z_g$, the surjectivity
of the other one gives us joint surjectivity.  Moreover, the trivial
homomorphism $z_g$ has a vanishing Schur invariant $\sch_G(z_g) = 0$
relative to the target group $G$.
\end{proof}

\subsection{End of the proof}
\label{ss:end}
We combine \Thm{th:dtrefine} with Lemmas \ref{l:include} and \ref{l:zsat}
to convert a circuit $Z$ in $\ZSAT_{J,A,I,F}$ to a mapping class $\phi \in
\MCG_*(\Sigma_{ng})$ using mapping class gadgets, where $J,A,I,F$ are as specified in the previous 
subsection. To apply \Lem{l:zsat}, we need to verify the conditions in \eqref{e:zineq}.  These follow
easily from asymptotic estimates on the cardinality of $A$ and $I$
\cite[Lems.~6.10~\&~6.11]{DT:random}.

For each $\tau \in \Rub_J(A \times A)$, we choose an $\alpha \in
\Tor_*(\Sigma_{2g}^1)$ such that:
\begin{enumerate}
\item $\alpha$ acts by $\tau$ on $A \times A$.
\item $\alpha$ acts by an element of $\Rub_J(R^0_{2g})$
that fixes $R^0_{2g} \setminus (A \times A)$.
\item $\alpha$ fixes $\hR_{2g} \setminus R_{2g}$.
\end{enumerate}
Given a circuit $Z$ in $\ZSAT_{J,A,I,F}$, we can replace each
gate $\tau \in \Rub_J(A \times A)$ that acts on symbols $i$
and $i+1$ by the corresponding local mapping class $\alpha \in
\Tor_*((\Sigma_{2g,1})_{(i,i+1)})$.  Then we let $\phi$ be the composition
of the gadgets $\alpha$.

\begin{lemma} Let
\[ M \defeq (H_{ng})_I \sqcup_\phi (H_{ng})_F. \]
Then
\begin{enumerate}
\item $M$ is a homology 3-sphere.
\item If $1 \lneq H \lneq G$ is a non-trivial, proper subgroup of $G$,
then $Q(M,H) = \emptyset$.
\item $\#H(M,G) = \#Z.$
\end{enumerate}
\label{l:glue} \end{lemma}

\begin{proof} Point 1 holds because by construction,
$\phi \in \Tor(\Sigma_{2g})$.

To address points 2 and 3, we decompose $\phi$ as a composition of 
local gadgets,
\begin{equation} \phi = \alpha_m \circ \alpha_{m-1} \circ \dots \circ
    \alpha_2 \circ \alpha_1, \label{e:phi}\end{equation}
and we insert parallel copies $(\Sigma_{ng})_j$ of the Heegaard surface
with $0 \le j \le m$, so the $i$th gadget $\alpha_j$ yields a map
\[ \alpha_j:(\Sigma_{ng})_{j-1} \to (\Sigma_{ng})_j \]
from the $(j-1)$-st to the $j$-th surface.  Each $\alpha_j$ is a non-trivial
homeomorphism
\[ \alpha_j:(\Sigma_g)_{j-1,(i,i+1)} \to (\Sigma_{ng})_{j,(i,i+1)} \]
for some $i$, and is the identity elsewhere.  We use this decomposition
to analyze the possibilities for a group homomorphism
\[ f:\pi_1(M) \to G. \]
The map $f$ restricts to a homomorphism
\[ f_j:\pi_1((\Sigma_{ng})_j) \to G, \]
and then further restricts to a homomorphism
\[ f_{j,i}:\pi_1((\Sigma_{g,1})_{j,i}) \to G \]
for the $i$th memory unit for each $i$.   It is convenient to interpret
$\hR_{g,1} \supseteq A$ as the superalphabet of all possible symbols that
could in principle arise as the state of a memory unit.

By construction, each initial symbol $f_{0,i}$ extends to the handlebody
$(H_g)_{I,i}$.   Thus $f_{0,i} \in I(H)$ for some subgroup $1 \le H \le
G$, and all cases are disjoint from $A$ other than $H = 1$ and $H=G$.
Likewise at the end, each $f_{m,i} \in F(H)$ for some $H$.  By construction,
each $\alpha_j$ fixes both $R^0_{2g} \setminus (A \times A)$ and $\hR_{2g}
\setminus R_{2g}$.  This fixed set includes all cases $R^0(H_1) \times
R^0(H_2)$, and therefore all cases $I(H_1) \times I(H_2)$, other than
$H_1,H_2 \in \{1,G\}$.  Thus every initial symbol $f_{0,i} \in I(H)
\not\subseteq A$ is preserved by every gadget $\alpha_j$, and then can't
finalize because $I(H) \cap F(H) = \emptyset$.  Among other things, this
establishes point 2 of the lemma.

This derivation also restricts the initial state $f_0$ to $A^n$.  In this
case, each $\alpha_j$ acts in the same way on $A^n$ as the corresponding gate
$\tau_j$.  Consequently, $\alpha_j$ leaves the set $A^n$ invariant.  Considering both
the circuit action and initialization and finalization, these states exactly
match the behavior of the circuit $Z$ under the rules of $\ZSAT_{J,A,I,F}$.
\end{proof}

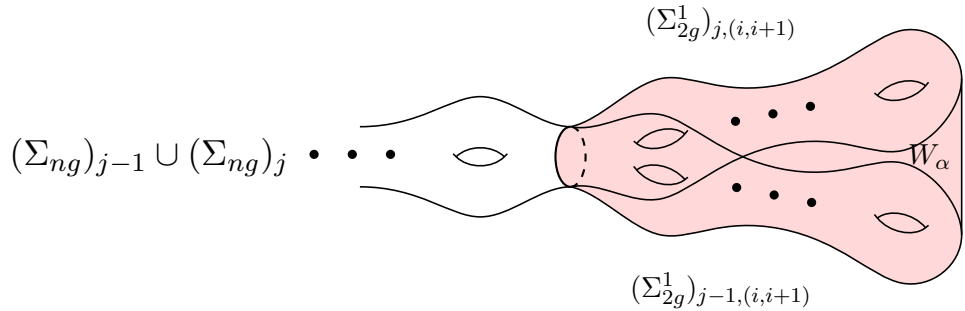
\begin{figure*}[htb] \begin{center}
\begin{tikzpicture}[semithick,scale=.4,even odd rule]
\fill[white!85!red] (0,-1) .. controls (1,-1.2) and (2,-2.4) .. (3,-2.6)
    .. controls (4,-2.8) and (5,-2) .. (7,-2.4)
    .. controls (9,-2.8) and (10,-4) .. (11,-4.2)
    .. controls (12,-4.4) and (13,-3.6) .. (13,-2.6) -- (13,2.6)
    .. controls (13,3.6) and (12,4.4) .. (11,4.2)
    .. controls (10,4) and (9,2.8) .. (7,2.4)
    .. controls (5,2) and (4,2.8) .. (3,2.6)
    .. controls (2,2.4) and (1,1.2) .. (0,1)
    arc  (90:270:.5 and 1);
\draw (13,-2.6) -- (13,2.6);
\draw (0,-1) .. controls (1,-.8) and (2,-1.6) .. (3,-1.4)
    .. controls (4,-1.2) and (5,0) .. (7,.4)
    .. controls (9,.8) and (10,0) .. (11,.2)
    .. controls (12,.4) and (13,1.6) .. (13,2.6)
    .. controls (13,3.6) and (12,4.4) .. (11,4.2)
    .. controls (10,4) and (9,2.8) .. (7,2.4)
    .. controls (5,2) and (4,2.8) .. (3,2.6)
     .. controls (2,2.4) and (1,1.2) .. (0,1);
\draw (2.2,.44) .. controls (2.6,.92) and (3.4,1.08) .. (3.8,.76);
\draw (2.1,.52) -- (2.2,.44) .. controls (2.6,.12) and (3.4,.28)
    .. (3.8,.76) -- (3.9,.88);
\draw (10.2,2.04) .. controls (10.6,2.52) and (11.4,2.68) .. (11.8,2.36);
\draw (10.1,2.12) -- (10.2,2.04) .. controls (10.6,1.72) and (11.4,1.88)
    .. (11.8,2.36) -- (11.9,2.48);
\draw (7,1.4) node {\rotatebox{11.3}{\scalebox{3}{$\cdots$}}};
\draw (0,-1) .. controls (1,-1.2) and (2,-2.4) .. (3,-2.6)
    .. controls (4,-2.8) and (5,-2) .. (7,-2.4)
    .. controls (9,-2.8) and (10,-4) .. (11,-4.2)
    .. controls (12,-4.4) and (13,-3.6) .. (13,-2.6)
    .. controls (13,-1.6) and (12,-.4) .. (11,-.2)
    .. controls (10,0) and (9,-.8) .. (7,-.4)
    .. controls (5,0) and (4,1.2) .. (3,1.4)
     .. controls (2,1.6) and (1,.8) .. (0,1);
\draw (2.2,-.44) .. controls (2.6,-.12) and (3.4,-.28) .. (3.8,-.76);
\draw (2.1,-.32) -- (2.2,-.44) .. controls (2.6,-.92) and (3.4,-1.08)
    .. (3.8,-.76) -- (3.9,-.68);
\draw (10.2,-2.04) .. controls (10.6,-1.72) and (11.4,-1.88) .. (11.8,-2.36);
\draw (10.1,-1.92) -- (10.2,-2.04) .. controls (10.6,-2.52) and (11.4,-2.68)
    .. (11.8,-2.36) -- (11.9,-2.28);
\draw (7,-1.4) node {\rotatebox{-11.3}{\scalebox{3}{$\cdots$}}};
\draw (0,1) .. controls (-1,1) and (-2,2) .. (-3,2)
    .. controls (-4,2) and (-5,1) .. (-7,1);
\draw (-7,-1) .. controls (-5,-1) and (-4,-2) .. (-3,-2)
    .. controls (-2,-2) and (-1,-1) .. (0,-1);
\draw (-2.2,0) .. controls (-2.6,.4) and (-3.4,.4) .. (-3.8,0);
\draw (-2.1,.1) -- (-2.2,0) .. controls (-2.6,-.4) and (-3.4,-.4) .. (-3.8,0)
    -- (-3.9,.1);
\draw (-7,0) node {\scalebox{3}{$\cdots$}};
\draw[thick] (0,1) arc (90:270:.5 and 1);
\draw[thick,dashed] (0,-1) arc (-90:90:.5 and 1);
\draw (13,0) node[anchor=east] {$W_\alpha$};
\draw (5,-3.5) node[anchor=north] {$(\Sigma_{2g}^1)_{j-1,(i,i+1)}$};
\draw (5,3.5) node[anchor=south] {$(\Sigma_{2g}^1)_{j,(i,i+1)}$};
\draw (-9,0) node[anchor=east] {\scalebox{1.25}{
    $(\Sigma_{ng})_{j-1} \cup (\Sigma_{ng})_j$}};
\end{tikzpicture}
\end{center}
\caption{The blister $W_\alpha$ between $(\Sigma_{2g,1})_{j-1,(i,i+1)}$ and
    $(\Sigma_{2g,1})_{j,(i,i+1)}$.}
\label{f:blister} \end{figure*}

To complete the proof of \Thm{th:main1}, we only need to efficiently
triangulate the 3-manifold $M \defeq (H_{ng})_I \sqcup_\phi (H_{ng})_F$.
The first step is to refine the decoration of $\Sigma_{ng}$ shown in
\Fig{f:heegaard} to a triangulation.   It is easy to do this with polynomial
complexity in $n$ (or in $ng$, but recall that $g$ is fixed).  We can also
give each subsurface $(\Sigma_{g}^1)_i$ the same triangulation for all $i$,
as well as each subsurface $(\Sigma_{2g}^1)_{i,i+1}$.  It is also routine
to extend any such triangulation to either $(H_{ng})_I$ or $(H_{ng})_F$ with
polynomial (indeed linear) overhead:  Since by construction the triangulation
of each $(\Sigma_{g}^1)_i$ is the same, we pick some extension to $(H_g)_I$
and $(H_g)_F$ and use it for each $(H_g)_{I,i}$ and each $(H_g)_{F,i}$.
The remainder of $(H_{ng})_I$ and $(H_{ng})_F$ is a 3-ball whose boundary
has now been triangulated; any triangulation of the boundary of a 3-ball can
be extended to the interior algorithmically and with polynomial complexity.

We insert more triangulated structure in between $(H_{ng})_I$ and
$(H_{ng})_F$ to realize the homeomorphism $\phi$.  Recalling equation
\eqref{e:phi} in the proof of \Lem{l:glue}, $\phi$ decomposes into
local mapping class gadgets $\alpha_j$.  Only finitely many $\alpha \in
\MCG_*(\Sigma_{g,1})$ are needed, since we only need one representative
for each $\tau \in \Rub_J(A \times A)$.  At this point it is convenient
to use a blister construction.  We make a 3-manifold $W_\alpha$ whose
boundary is two copies of $\Sigma_{2g,1}$ (with its standard triangulation)
that meet at their boundary circle, and so that $W_\alpha$ is a relative
mapping cylinder for the homeomorphism $\alpha$.   If $\alpha_j$ acts
on $(\Sigma_{2g}^1)_{i,i+1}$, then we can have $(\Sigma_{ng})_{j-1}$ and
$(\Sigma_{ng})_j$ coincide outside of $(\Sigma_{2g}^1)_{j-1,(i,i+1)}$ and
$(\Sigma_{2g}^1)_{j,(i,i+1)}$, so that their union $(\Sigma_{ng})_{j-1}
\cup (\Sigma_{ng})_j$ is a branched surface.  We insert $W_\alpha$
and its triangulation in the blister within $(\Sigma_{ng})_{j-1} \cup
(\Sigma_{ng})_j$; see \Fig{f:blister}. \qed

\section{Reduction to knot complements}
\label{s:knotreduction}
In this section, we complete the proof of Theorem \ref{th:main2}.

\subsection{A convenient equivariant alphabet}
\label{ss:alphabet}
We choose specific $J$, $A$, $z$, $I$ and $F$ that are both topologically inspired and satisfy the conditions of \Thm{l:zsat}.  Our choices provide a convenient $\shP$-complete problem which we will parsimoniously reduce to $\#H(-,G,c)$ in \Sec{ss:reduction}.

Let $c \in G$ be nontrivial and let $C$ be its conjugacy class.  Fix $k$ large enough for the conclusion of \Thm{th:rvrefine} to hold.  Let $J = \Aut(G,c)$, the group of automorphisms of $G$ fixing $c$.  We let the zombie symbol be
\[ z = (c,c^{-1},c,c^{-1},\dots,c,c^{-1}) \in (C\times C^{-1})^k \subset \hat{R}_{2k}. \]
The total alphabet is
\[ A = \{z\} \cup \{ (x_1,\dots,x_{2k}) \in R_v^0 \mid x_1 = c, x_{2k} = c^{-1} \}. \]
That is, the non-zombie symbols in $A$ are surjections with trivial Conway-Parker invariant such that the leftmost (resp. rightmost) puncture maps to $c$ (resp. $c^{-1}$), and not some arbitrary element of $C$ (resp. $C^{-1}$).  The initialization and finalization conditions are specified by restricting to homomorphisms that factor through the two trivial tangles in Figure \ref{f:plats}, respectively.  Precisely, the initialization sub-alphabet is
\[ I =\{ (x_1,\dots,x_{2k}) \in R_v^0 \mid  x_1 = c, x_{2k} = c^{-1}, x_{2i} = x_{2i-1}^{-1} \forall i=1,\dots,k \}\]
and the finalization sub-alphabet is
\[ F =\{ (x_1,\dots,x_{2k}) \in R_v^0 \mid  x_1 = c, x_{2k} = c^{-1}, x_{2i} = x_{2i+1}^{-1} \forall i=1,\dots,k-1 \} \]
For the rest of the section, $J, A, z, I$ and $F$ denote these specific sets.

\begin{figure*}
\begin{tikzpicture}[scale = 0.10, semithick, decoration={markings,
    mark=at position 0.25 with {\arrow{angle 90}}}]
\foreach \x in {-5,15,45}
	{
	\draw[darkred,thick] (\x,25-2) arc [start angle = 180, end angle = 360, radius = 5];
	\draw[darkred,thick] (\x,25) -- (\x,25-2);
	\draw[darkred,thick] (\x+10,25) -- (\x+10,25-2);
	}
\foreach \i in {1,2,3,4}
	{
	\coordinate (p) at (-15+10*\i,25);	
	\draw[fill] (p) circle [radius=.5];	

	\coordinate (w) at ($ (0,0)!50pt!210-30*\i:(1,0) $);
	\coordinate (v) at ($ (25,0)+8*(w) $);	
	
	\coordinate (b) at ($ (p)+(0,5) $);
	\coordinate (bE) at ($ (b) + (1,0)$);	
	\coordinate (bW) at ($ (b) - (1,0)$);	

	\coordinate (a) at ($ (p) + (3,0) $);	
	\coordinate (aN) at ($ (a) + (0,5) $);
	\coordinate (aS) at ($ (a) - (0,12) $);	
	\coordinate (c) at ($ (p) - (3,0) $);	
	\coordinate (cN) at ($ (c) + (0,5) $); 
	\coordinate (cS) at ($ (c) - (0,12) $); 	
	
	\draw[white, line width = 4] (25,0) .. controls (v) and (aS) .. (a) .. controls (aN) and (bE) .. (b) .. controls (bW) and (cN) .. (c) .. controls (cS) and (v) .. (25,0);
	}
\foreach \i in {6}
	{
	\coordinate (p) at (-15+10*\i,25);	
	\draw[fill] (p) circle [radius=.5];	

	\coordinate (w) at ($ (0,0)!50pt!210-30*\i:(1,0) $);
	\coordinate (v) at ($ (25,0)+6*(w) $);	
	
	\coordinate (b) at ($ (p)+(0,5.5) $);
	\coordinate (bE) at ($ (b) + (1,0)$);	
	\coordinate (bW) at ($ (b) - (1,0)$);	

	\coordinate (a) at ($ (p) + (4.25,0) $);	
	\coordinate (aN) at ($ (a) + (0,6) $);
	\coordinate (aS) at ($ (a) - (0,12) $);	
	\coordinate (c) at ($ (p) - (4.3,0) $);	
	\coordinate (cN) at ($ (c) + (0,6) $); 
	\coordinate (cS) at ($ (c) - (0,12) $); 	
	
	\draw[white, line width = 4] (25,0) .. controls (v) and (aS) .. (a) .. controls (aN) and (bE) .. (b) .. controls (bW) and (cN) .. (c) .. controls (cS) and (v) .. (25,0);
	}
\foreach \i in {7}
	{
	\coordinate (p) at (-15+10*\i,25);	
	\draw[fill] (p) circle [radius=.5];	

	\coordinate (w) at ($ (0,0)!50pt!210-30*\i:(1,0) $);
	\coordinate (v) at ($ (25,0)+10*(w) $);	
	
	\coordinate (b) at ($ (p)+(0,5) $);
	\coordinate (bE) at ($ (b) + (1,0)$);	
	\coordinate (bW) at ($ (b) - (1,0)$);	

	\coordinate (a) at ($ (p) + (3,0) $);	
	\coordinate (aN) at ($ (a) + (0,5) $);
	\coordinate (aS) at ($ (a) - (0,12) $);	
	\coordinate (c) at ($ (p) - (3,0) $);	
	\coordinate (cN) at ($ (c) + (0,5) $); 
	\coordinate (cS) at ($ (c) - (0,12) $); 	
	
	\draw[white, line width = 4] (25,0) .. controls (v) and (aS) .. (a) .. controls (aN) and (bE) .. (b) .. controls (bW) and (cN) .. (c) .. controls (cS) and (v) .. (25,0);
	}
\foreach \i in {1}
	{
	\coordinate (p) at (-15+10*\i,25);	
	\draw[fill] (p) circle [radius=.5];	
	\node[above] at (p) {$p_{\i}$}; 
	\node[above, medgreen] at ($ (p)+(0,4.5) $) {\large{$c$}}; 

	\coordinate (w) at ($ (0,0)!50pt!210-30*\i:(1,0) $);
	\coordinate (v) at ($ (25,0)+8*(w) $);	
	
	\coordinate (b) at ($ (p)+(0,5) $);
	\coordinate (bE) at ($ (b) + (1,0)$);	
	\coordinate (bW) at ($ (b) - (1,0)$);	

	\coordinate (a) at ($ (p) + (3,0) $);	
	\coordinate (aN) at ($ (a) + (0,5) $);
	\coordinate (aS) at ($ (a) - (0,12) $);	
	\coordinate (c) at ($ (p) - (3,0) $);	
	\coordinate (cN) at ($ (c) + (0,5) $); 
	\coordinate (cS) at ($ (c) - (0,12) $); 	
	
	\draw[medgreen,postaction={decorate}] (25,0) .. controls (v) and (aS) .. (a) .. controls (aN) and (bE) .. (b) .. controls (bW) and (cN) .. (c) .. controls (cS) and (v) .. (25,0);
	}		
\foreach \i in {2,3,4}
	{
	\coordinate (p) at (-15+10*\i,25);	
	\draw[fill] (p) circle [radius=.5];	
	\node[above] at (p) {$p_{\i}$}; 
	\node[above, medgreen] at ($ (p)+(0,4.5) $) {\large{$\gamma_{\i}$}}; 

	\coordinate (w) at ($ (0,0)!50pt!210-30*\i:(1,0) $);
	\coordinate (v) at ($ (25,0)+8*(w) $);	
	
	\coordinate (b) at ($ (p)+(0,5) $);
	\coordinate (bE) at ($ (b) + (1,0)$);	
	\coordinate (bW) at ($ (b) - (1,0)$);	

	\coordinate (a) at ($ (p) + (3,0) $);	
	\coordinate (aN) at ($ (a) + (0,5) $);
	\coordinate (aS) at ($ (a) - (0,12) $);	
	\coordinate (c) at ($ (p) - (3,0) $);	
	\coordinate (cN) at ($ (c) + (0,5) $); 
	\coordinate (cS) at ($ (c) - (0,12) $); 	
	
	\draw[medgreen,postaction={decorate}] (25,0) .. controls (v) and (aS) .. (a) .. controls (aN) and (bE) .. (b) .. controls (bW) and (cN) .. (c) .. controls (cS) and (v) .. (25,0);
	}	
\foreach \i in {6}
	{
	\coordinate (p) at (-15+10*\i,25);	
	\draw[fill] (p) circle [radius=.5];	
	\node[above, medgreen] at ($ (p)+(0,4.5) $) {\large{$\gamma_{2n-1}$}}; 

	\coordinate (w) at ($ (0,0)!50pt!210-30*\i:(1,0) $);
	\coordinate (v) at ($ (25,0)+6*(w) $);	
	
	\coordinate (b) at ($ (p)+(0,5.5) $);
	\coordinate (bE) at ($ (b) + (1,0)$);	
	\coordinate (bW) at ($ (b) - (1,0)$);	

	\coordinate (a) at ($ (p) + (4.25,0) $);	
	\coordinate (aN) at ($ (a) + (0,6) $);
	\coordinate (aS) at ($ (a) - (0,12) $);	
	\coordinate (c) at ($ (p) - (4.3,0) $);	
	\coordinate (cN) at ($ (c) + (0,6) $); 
	\coordinate (cS) at ($ (c) - (0,12) $); 
	
	\draw[medgreen,postaction={decorate}] (25,0) .. controls (v) and (aS) .. (a) .. controls (aN) and (bE) .. (b) .. controls (bW) and (cN) .. (c) .. controls (cS) and (v) .. (25,0);
	}
\foreach \i in {7}
	{
	\coordinate (p) at (-15+10*\i,25);	
	\draw[fill] (p) circle [radius=.5];	
	\node[above, medgreen] at ($ (p)+(1.5,4.5) $) {\large{$c^{-1}$}}; 

	\coordinate (w) at ($ (0,0)!50pt!210-30*\i:(1,0) $);
	\coordinate (v) at ($ (25,0)+10*(w) $);	
	
	\coordinate (b) at ($ (p)+(0,5) $);
	\coordinate (bE) at ($ (b) + (1,0)$);	
	\coordinate (bW) at ($ (b) - (1,0)$);	

	\coordinate (a) at ($ (p) + (3,0) $);	
	\coordinate (aN) at ($ (a) + (0,5) $);
	\coordinate (aS) at ($ (a) - (0,12) $);	
	\coordinate (c) at ($ (p) - (3,0) $);	
	\coordinate (cN) at ($ (c) + (0,5) $); 
	\coordinate (cS) at ($ (c) - (0,12) $); 	
	
	\draw[medgreen,postaction={decorate}] (25,0) .. controls (v) and (aS) .. (a) .. controls (aN) and (bE) .. (b) .. controls (bW) and (cN) .. (c) .. controls (cS) and (v) .. (25,0);
	}
	
\node[above] at (45,25) {$p_{2n-1}$};
\node[above] at (55,25) {$p_{2n}$};
\draw (25,25) circle [x radius=40, y radius = 25];	
\node at (34.5,25) {\huge \dots};
\draw[fill] (25,0) circle [radius=.5];
\node[below] at (25,0) {\huge $*$};
\end{tikzpicture}
\hfill
\begin{tikzpicture}[scale = 0.10, semithick, decoration={markings,
    mark=at position 0.35 with {\arrow{angle 90}}}]
\foreach \i in {1,2,3}
	{
	\coordinate (p) at (-15+10*\i,25);	
	\node[below] at (p) {$p_{\i}$}; 
			
	\coordinate (w) at ($ (0,0)!50pt!210-30*\i:(1,0) $);
	\coordinate (v) at ($ (25,0)+8*(w) $);	
	
	\coordinate (b) at ($ (p)+(0,2.5) $);
	\coordinate (bE) at ($ (b) + (.1,0)$);	
	\coordinate (bW) at ($ (b) - (.1,0)$);	

	\coordinate (a) at ($ (p) + (3,0) $);	
	\coordinate (aN) at ($ (a) + (0,2.5) $);
	\coordinate (aS) at ($ (a) - (0,12) $);	
	\coordinate (c) at ($ (p) - (3,0) $);	
	\coordinate (cN) at ($ (c) + (0,2.5) $); 
	\coordinate (cS) at ($ (c) - (0,12) $); 	
	
	\draw[medgreen,postaction={decorate}] (25,0) .. controls (v) and (aS) .. (a) .. controls (aN) and (bE) .. (b) .. controls (bW) and (cN) .. (c) .. controls (cS) and (v) .. (25,0);
	}	
\foreach \i in {5,6}
	{
	\coordinate (p) at (-15+10*\i,25);	

	\coordinate (w) at ($ (0,0)!50pt!210-30*\i:(1,0) $);
	\coordinate (v) at ($ (25,0)+8*(w) $);	
	
	\coordinate (b) at ($ (p)+(0,2.5) $);
	\coordinate (bE) at ($ (b) + (1,0)$);	
	\coordinate (bW) at ($ (b) - (1,0)$);	

	\coordinate (a) at ($ (p) + (4.5,0) $);	
	\coordinate (aN) at ($ (a) + (0,2.5) $);
	\coordinate (aS) at ($ (a) - (0,12) $);	
	\coordinate (c) at ($ (p) - (4.3,0) $);	
	\coordinate (cN) at ($ (c) + (0,2.5) $); 
	\coordinate (cS) at ($ (c) - (0,12) $); 	
	
	\draw[medgreen,postaction={decorate}] (25,0) .. controls (v) and (aS) .. (a) .. controls (aN) and (bE) .. (b) .. controls (bW) and (cN) .. (c) .. controls (cS) and (v) .. (25,0);
	}
\foreach \i in {7}
	{
	\coordinate (p) at (-15+10*\i,25);	

	\coordinate (w) at ($ (0,0)!50pt!210-30*\i:(1,0) $);
	\coordinate (v) at ($ (25,0)+8*(w) $);	
	
	\coordinate (b) at ($ (p)+(0,2.5) $);
	\coordinate (bE) at ($ (b) + (1,0)$);	
	\coordinate (bW) at ($ (b) - (1,0)$);	

	\coordinate (a) at ($ (p) + (3,0) $);	
	\coordinate (aN) at ($ (a) + (0,2.5) $);
	\coordinate (aS) at ($ (a) - (0,12) $);	
	\coordinate (c) at ($ (p) - (3,0) $);	
	\coordinate (cN) at ($ (c) + (0,2.5) $); 
	\coordinate (cS) at ($ (c) - (0,12) $); 	
	
	\draw[medgreen,postaction={decorate}] (25,0) .. controls (v) and (aS) .. (a) .. controls (aN) and (bE) .. (b) .. controls (bW) and (cN) .. (c) .. controls (cS) and (v) .. (25,0);
	}(25,25) circle [x radius=40, y radius = 25];		
\draw[white, line width = 4] (5,25) to (5,25+2) arc [start angle = 180, end angle = 0, radius = 5] to (5+10,25);
\draw[darkblue,thick] (5,25) to (5,25+2) arc [start angle = 180, end angle = 0, radius = 5] to (5+10,25);
\draw[white, line width = 4] (45,25) to (45,25+2) arc [start angle = 0, end angle = 180, radius = 5] to (35,25);
\draw[darkblue,thick] (45,25) to (45,25+2) arc [start angle = 0, end angle = 180, radius = 5] to (35,25);
\draw[white, line width = 4] (-5,25) arc [start angle = 180, end angle = 0, x radius = 30, y radius = 20];
\draw[darkblue,thick] (-5,25) arc [start angle = 180, end angle = 0, x radius = 30, y radius = 20];
\foreach \i in {1}
	{
	\coordinate (p) at (-15+10*\i,25);	
	\draw[fill] (p) circle [radius=.5];	
	\node[above, medgreen] at ($ (p)+(-3,1.5) $) {\large $c$}; 
	}
\foreach \i in {2}
	{
	\coordinate (p) at (-15+10*\i,25);	
	\draw[fill] (p) circle [radius=.5];	
	\node[above, medgreen] at ($ (p)+(-3,1.5) $) {\large $\gamma_{\i}$}; 
	}
\foreach \i in {3}
	{
	\coordinate (p) at (-15+10*\i,25);	
	\draw[fill] (p) circle [radius=.5];	
	\node[above, medgreen] at ($ (p)+(3,1.5) $) {\large $\gamma_{\i}$}; 
	}
\draw[fill] (35,25) circle [radius=.5];	
\node[below] at (35,25) {$p_{2n-2}$};
\node[above, medgreen] at ($ (-15+10*5,25) + (-5,1.5) $) {\large $\gamma_{2n-2}$};
\draw[fill] (45,25) circle [radius=.5];	
\node[below] at (45,25) {$p_{2n-1}$};
\node[above, medgreen] at ($ (-15+10*6,25)+(5,1.5) $) {\large $\gamma_{2n-1}$};
\draw[fill] (55,25) circle [radius=.5];	
\node[below] at (55,25) {$p_{2n}$};
\node[above, medgreen] at ($ (-15+10*7,25)+(5,1.5) $) {\large $c^{-1}$};
\node at (24.5,25) {\huge \dots};
\draw[fill] (25,0) circle [radius=.5];
\node[below] at (25,0) {\huge $*$};
\draw (25,25) circle [x radius=40, y radius = 25];	
\end{tikzpicture}
\caption{The initialization and finalization constraints.}
\label{f:plats}
\end{figure*}
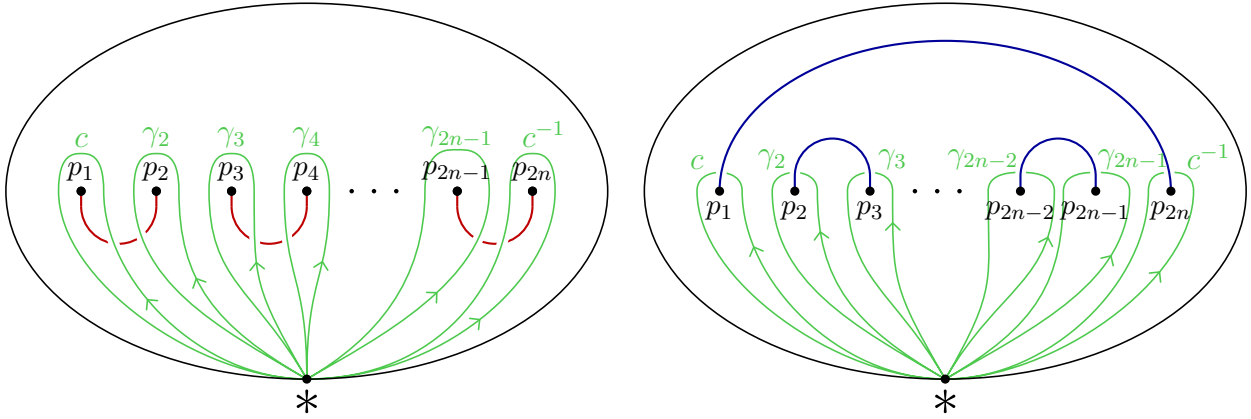

It is straightforward to verify that these choices satisfy the conditions of Lemma \ref{l:zsat}.  In particular, $I$ and $F$ are $J$-invariant.  We note $C$ generates $G$ because $G$ is simple and $c$ is nontrivial.  This implies that $A\setminus \{z\}$ has a large, nonzero cardinality, and is a free $J$-set.  Hence we have

\begin{lemma}
With these choices of $J, A, z, I$ and $F$, $\ZSAT_{J,A,I,F}$ is $\shP$-complete via almost parsimonious reduction. \qed
\end{lemma}

Our readers may have the impression that our choices of $A, I$ and $F$ are somewhat contrived.  They would not be wrong, since restricting $A$ to consist only of homomorphisms where the first and last punctures map to $c$ and $c^{-1}$ is not natural from a topological perspective.  One could argue that the choices
\[ J' = \Aut(G,C), \]
\[ A' = \{z\} \cup R_v^0, \]
\[ I' =\{ (x_1,\dots,x_{2k}) \in R_v^0 \mid  x_{2i} = x_{2i-1}^{-1} \forall i=1,\dots,k \}, \]
and
\[ F' =\{ (x_1,\dots,x_{2k}) \in R_v^0 \mid  x_{2i} = x_{2i+1}^{-1} \forall i=1,\dots,k-1 \} \]
are more natural.  While it is possible to define a $\shP$-complete version of $\ZSAT$ with these choices, it would not be possible to reduce every instance of this model to a \emph{knot}.  Instead, if $Z$ is a reversible circuit of width $n$, one could only hope to construct a link complement with $n$ components.  

To build a reduction to knot diagrams, we need some way of ``coupling" the input and output strands of $Z$ to each other.  There are various ways to achieve this.  One could generalize the definition of $\ZSAT$ so the initialization and finalization conditions are \emph{2-local} instead of \emph{1-local}, meaning they are subsets of $A^2$ instead of $A$.  The downside to this approach is that we would have to generalize Theorem \ref{l:zsat} to the 2-local setting.  This is possible, but to keep a proliferation of circuit models from taking over this dissertation, we proceed by an alternate route that exploits some topological tricks.  Roughly, our choice of $A, I$ and $F$ made above builds a ``trivial" coupling into $\ZSAT_{J,A,I,F}$ itself.  We make this precise in Section \ref{ss:reduction}, although the reader might look ahead at Figure \ref{f:surgery} now in order to get a sense of what is to come.

\subsection{Pure braid gadgets}
\label{ss:gadgets}
We now construct braid gadgets that simulate gates in $\Rub_{J}(A\times A)$.  Consider a pointed disk $D_{4k}$ with $4k$ punctures.  Choose two smaller disks with $2k$ punctures, each of which contains the basepoint and half of the $4k$ punctures of $D_{4k}$, and whose intersection is contractible.  This allows us to identify $\hat{R}_{2k} \times \hat{R}_{2k} = G^{2k} \times G^{2k}$  with $\hat{R}_{4k} = G^{4k}$.  It is straightforward to verify that this identification takes $R_v^0 \times R_v^0$ to a subset of $R_{v\#v}^0$, where $v\#v$ denotes the concatenation of two copies of $v$.  In particular, we identify $A \times A$ with a subset of $R_{v\#v}^0 \cup \{(z,z)\}$.

For every gate $\tau \in \Rub_J(A\times A)$, fix a braid $b_\tau \in B_{v \# v} \leq B_{4k}$ with the following properties:
\begin{enumerate}
\item $b_\tau$ acts on $A\times A$ as $\tau$,
\item $b_\tau$ acts trivially on $R_{v\#v}^0 \setminus \Aut(G,C) \cdot (A \times A)$,
\item $b_\tau$ acts trivially on $\hat{R}_{v\#v} \setminus R_{v\#v}$, and
\item $b_\tau$ is a pure braid.
\end{enumerate}
We elaborate on properties 1 and 2.  On one hand, property 1 specifies how $b_\tau$ should act on $A\times A$.  On the other hand, $B_{v \#v}$ acts on $R_{v\#v}^0$ by $\Aut(G,C)$-set automorphisms, and $\Aut(G,C) \setminus \Aut(G,c)$ is nonempty, hence $A\times A \leq R_{v\#v}^0$ is not closed under the $\Aut(G,C)$ action on $P_{2k}^C$.  However, there is a natural embedding
\[ \Rub_{\Aut(G,c)}(A\times A) \into \Rub_{\Aut(G,C)}(\Aut(G,C) \cdot(A\times A)) \]
where
\[ \Rub_{\Aut(G,C)}(\Aut(G,C) \cdot(A\times A)) \leq \Rub_{\Aut(G,C)}(R_{v\#v}^0) \]
is an honest subgroup.  The embedding extends an element of $\Rub_{\Aut(G,c)}(A\times A)$ to an element of $\Rub_{\Aut(G,C)}(\Aut(G,C) \cdot(A\times A))$ by acting on $A\times A \subset \Aut(G,C) \cdot(A\times A)$ as before, and acting on each orbit in 
\[ [\Aut(G,C) \cdot(A\times A)]/ \Aut(G,c) \]
in an isomorphic fashion.  Conflating $\Rub_{\Aut(G,c)}(A\times A)$ with its image under this embedding, we see that the support of any $\tau \in \Rub_{\Aut(G,c)}(A\times A)$ is restricted to \emph{aligned} states.  In particular, $\tau$ acts trivially on
\[ \Aut(G,C) \cdot A \times \Aut(G,C) \cdot A \setminus \Aut(G,C)\cdot(A\times A). \]
Moreover, $\tau$ \emph{preserves alignment} in $\Aut(G,C) \cdot (A \times A)$, meaning that for every $\alpha \in \Aut(G,C)$ and $(f_1,f_2) \in A\times A$, there exists $(g_1,g_2) \in A \times A$ such that $\tau \cdot (\alpha\cdot f_1, \alpha \cdot f_2) = (\alpha \cdot g_1,\alpha \cdot g_2)$.

Theorem \ref{th:rvrefine} implies a choice of $b_\tau$ satisfying all four properties exists.  For every $\tau$, we fix an expression of $b_\tau$ as a product of elementary braid generators and their inverses.  This is equivalent to picking a diagram of $b_\tau$ in general position.  These choices of diagrams, which we also call $b_\tau$, are our pure braid gadgets.

We record here a useful property of the $b_\tau$ that follows immediately from their definition:

\begin{lemma}
For each $\tau \in \Rub_J(A^2)$, $b_\tau \in PB_{4k}$ preserves the subset $\hat{R}_v^0 \times \hat{R}_v^0 \subset \hat{R}_{v\# v}$. \qed
\label{l:preserve}
\end{lemma}

\subsection{The reduction}
\label{ss:reduction}
Let $Z$ be an instance of $\ZSAT_{J,A,I,F}$, with $J, A, I, F$ as in \Sec{ss:alphabet}. Recall this means $Z$ is a planar $\Aut(G,c)$-equivariant reversible circuit over the alphabet $A$.  Suppose the width of $Z$ is $n$.

Consider the disk $D_{2kn}$ with $2kn$ punctures and basepoint $* \in D_{2kn}$.  For $i=1,\dots,n$, let $D_{2k,i} \subset D_{nk}$ denote the $n$ different $2k$-punctured disks indicated in Figure \ref{f:disks}.  Note each $D_{2k,i}$ contains the basepoint $*$.  We pick generators $\gamma_{1,i},\dots,\gamma_{2k,i}$ for $\pi_1(D_{2k,i},*)$ as indicated in the figure.

\begin{figure*}[t]
\begin{center}
\begin{tikzpicture}[scale = 0.07, semithick, decoration={markings,
    mark=at position 0.45 with {\arrow{angle 90}}}]
\coordinate (*) at (100,0); 
\coordinate (a) at ($ (*)!.9!344:(0,0) $);
\coordinate (b) at ($ (*)!.65!309.7:(0,0) $);
\coordinate (c) at ($ (*)!.65!270:(0,0) $);
\coordinate (c') at ($ (*)!.4875!270:(0,0) $);
\coordinate (d) at ($ (*)!.65!230.3:(0,0) $);
\coordinate (d') at ($ (*)!.52!230.3:(0,0) $);
\coordinate (e) at ($ (*)!.9!196:(0,0) $);

\foreach \i in {1,2,4}
	{
	\coordinate (ap) at ($ (a)!0.2*\i!(b) $);
	\coordinate (aE) at ($ (a)!0.2*\i + 0.06!(b) $);
	\coordinate (aW) at ($ (a)!0.2*\i - 0.06!(b) $);	
	\draw[fill] (ap) circle [radius=.5];
	\coordinate (F) at ($ 4*(aE)-\i*(aE) - 4*(ap)+\i*(ap) $);
	\coordinate (aN) at ($ (*)!1.05!0:(ap) + 0.2*(F) $);
	\coordinate (t) at 	($ (*)!.6! -10*\i:(0,0) $);
	\coordinate (u1) at ($ (aE) - .85*(aN) + .85*(ap) $);
	\coordinate (u2) at ($ (aE) + .85*(aN) - .85*(ap) $);	
	\coordinate (v1) at ($ (aN) + .3*(aE)-.3*(ap) $);
	\coordinate (v2) at ($ (aN) - .3*(aE)+.3*(ap)  $);	
	\coordinate (w1) at ($ (aW) + .85*(aN) - .85*(ap) $);
	\coordinate (w2) at ($ (aW) - .85*(aN) + .85*(ap) $);	
	\draw[medgreen,postaction={decorate}] (*) .. controls (t) and	(u1) .. (aE) .. controls (u2) and (v1) .. (aN) .. controls (v2) and (w1) .. (aW) .. controls (w2) and (t) .. (*);	
	}

\node[above left, medgreen] at ($ (a)!0.2*1!(b) + (-1,1)$) {$\gamma_{1,1}$};
\node[above left, medgreen] at ($ (a)!0.2*2!(b) + (-1,1)$) {$\gamma_{2,1}$};
\node[above left, medgreen] at ($ (a)!0.2*4!(b) + (-1,1)$) {$\gamma_{2k,1}$};
	
\draw[fill] ($ (a)!0.2*3-0.03!(b) $) circle [radius=.2];
\draw[fill] ($ (a)!0.2*3!(b) $) circle [radius=.2];
\draw[fill] ($ (a)!0.2*3+0.03!(b) $) circle [radius=.2];

\foreach \i in {1,2,4}
	{
	\coordinate (bp) at ($ (b)!0.2*\i!(c) $);
	\coordinate (bE) at ($ (b)!0.2*\i + 0.07!(c) $);
	\coordinate (bW) at ($ (b)!0.2*\i - 0.07!(c) $);	
	\draw[fill] (bp) circle [radius=.5];
	\coordinate (bN) at ($ (bp)!100pt!90:(c) $);
	\coordinate (bt) at ($ (*)!0.8!(bp) $);
	\coordinate (bu1) at ($ (bE) - 0.65*(bN) + 0.65*(bp) $);
	\coordinate (bu2) at ($ (bE) + 0.65*(bN) - 0.65*(bp) $);	
	\coordinate (bv1) at ($ (bN) - 0.45*(bW) + 0.45*(bp)  $);
	\coordinate (bv2) at ($ (bN) + 0.45*(bW) - 0.45*(bp)  $);	
	\coordinate (bw1) at ($ (bW) + 0.65*(bN) - 0.65*(bp) $);
	\coordinate (bw2) at ($ (bW) - 0.65*(bN) + 0.65*(bp) $);	
	\draw[medgreen,postaction={decorate}] (*) .. controls (bt) and (bu1) .. (bE) .. controls (bu2) and (bv1) .. (bN) .. controls (bv2) and (bw1) .. (bW) .. controls (bw2) and (bt) .. (*);
	}

\node[above left, medgreen] at ($ (b)!0.2*1!(c) + (-1,1)$) {$\gamma_{1,2}$};
\node[above left, medgreen] at ($ (b)!0.2*2!(c) + (-1,1)$) {$\gamma_{2,2}$};
\node[above left, medgreen] at ($ (b)!0.2*4!(c) + (-1,1)$) {$\gamma_{2k,2}$};

\draw[fill] ($ (b)!0.2*3-0.03!(c) $) circle [radius=.2];
\draw[fill] ($ (b)!0.2*3!(c) $) circle [radius=.2];
\draw[fill] ($ (b)!0.2*3+0.03!(c) $) circle [radius=.2];

\draw[fill] ($ (c')!0.5-0.1!(d') $) circle [radius=.5];
\draw[fill] ($ (c')!0.5!(d') $) circle [radius=.5];
\draw[fill] ($ (c')!0.5+0.1!(d') $) circle [radius=.5];

\foreach \i in {1,3,4}
	{
	\coordinate (p) at ($ (e)!0.2*\i!(d) $);
	\coordinate (E) at ($ (e)!0.2*\i - 0.06!(d) $);
	\coordinate (W) at ($ (e)!0.2*\i + 0.06!(d) $);	
	\draw[fill] (p) circle [radius=.5];
	\coordinate (F) at ($ 0.2*4*(W)-0.2*\i*(W) - 0.2*4*(p)+0.2*\i*(p) $);
	\coordinate (N) at ($ (*)!1.05!0:(p) + (F) $);
	\coordinate (t) at ($ (*)!.6! 180+10*\i:(0,0)$);
	\coordinate (u1) at ($ (E) - .85*(N) + .85*(p) $);
	\coordinate (u2) at ($ (E) + .85*(N) - .85*(p) $);	
	\coordinate (v1) at ($ (N) + .3*(E) - .3*(p) $);
	\coordinate (v2) at ($ (N) - .3*(E) + .3*(p) $);	
	\coordinate (w1) at ($ (W) + .85*(N) - .85*(p) $);
	\coordinate (w2) at ($ (W) - .85*(N) + .85*(p) $);	
	\draw[medgreen,postaction={decorate}] (*) .. controls (t) and	(u1) .. (E) .. controls (u2) and (v1) .. (N) .. controls (v2) and (w1) .. (W) .. controls (w2) and (t) .. (*);	
	}
	
\node[above right, medgreen] at ($ (d)!0.2*1!(e) + (1,1)$) {$\gamma_{1,n}$};
\node[above right, medgreen] at ($ (d)!0.2*2!(e) + (1,1)$) {$\gamma_{2,n}$};
\node[above right, medgreen] at ($ (d)!0.2*4!(e) + (1,1)$) {$\gamma_{2k,n}$};	

\draw[fill] ($ (d)!0.2*3-0.03!(e) $) circle [radius=.2];
\draw[fill] ($ (d)!0.2*3!(e) $) circle [radius=.2];
\draw[fill] ($ (d)!0.2*3+0.03!(e) $) circle [radius=.2];

\draw (100,50) circle [x radius=100, y radius = 50];    
\draw[fill] (*) circle [radius=.5];
\draw (*) arc [x radius=100, y radius = 50, start angle = 270, end angle = 135 ] to (*);
\draw (*) arc [x radius=100, y radius = 50, start angle = 270, end angle = 90 ] to (*);
\draw (*) arc [x radius=100, y radius = 50, start angle = 270, end angle = 45 ] to (*);
\node[below] at (*) {\huge $*$};
\node at (20,55) {\LARGE $D_{2k,1}$};
\node at (70,78) {\LARGE $D_{2k,2}$};
\node at (180,55) {\LARGE $D_{2k,n}$};
\node at (200,80) {\huge $D_{2kn}$};
\end{tikzpicture}

\caption{The punctured disks encoding the data register of a width $n$ reversible circuit.}
\label{f:disks}
\end{center}
\end{figure*}
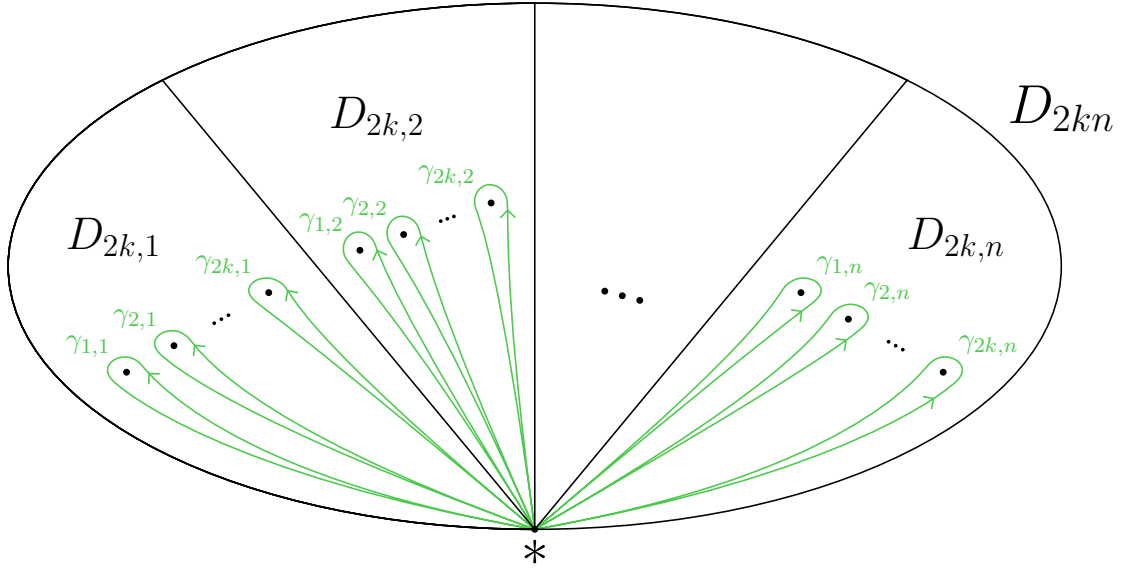

Convert $Z$ into a braid diagram $b_Z$ by replacing each strand in $Z$ with $2k$ parallel strands and each gate $\tau_i$ in $Z$ with the diagram of the braid gadget $b_{\tau_i}$ as in Figure \ref{f:reduction}.  Let $K_Z$ be the oriented link diagram formed by the plat closure of $b_Z$ indicated in the figure, and let $\gamma_Z \in \pi_1(S^3\setminus K_Z)$ be the indicated meridian.

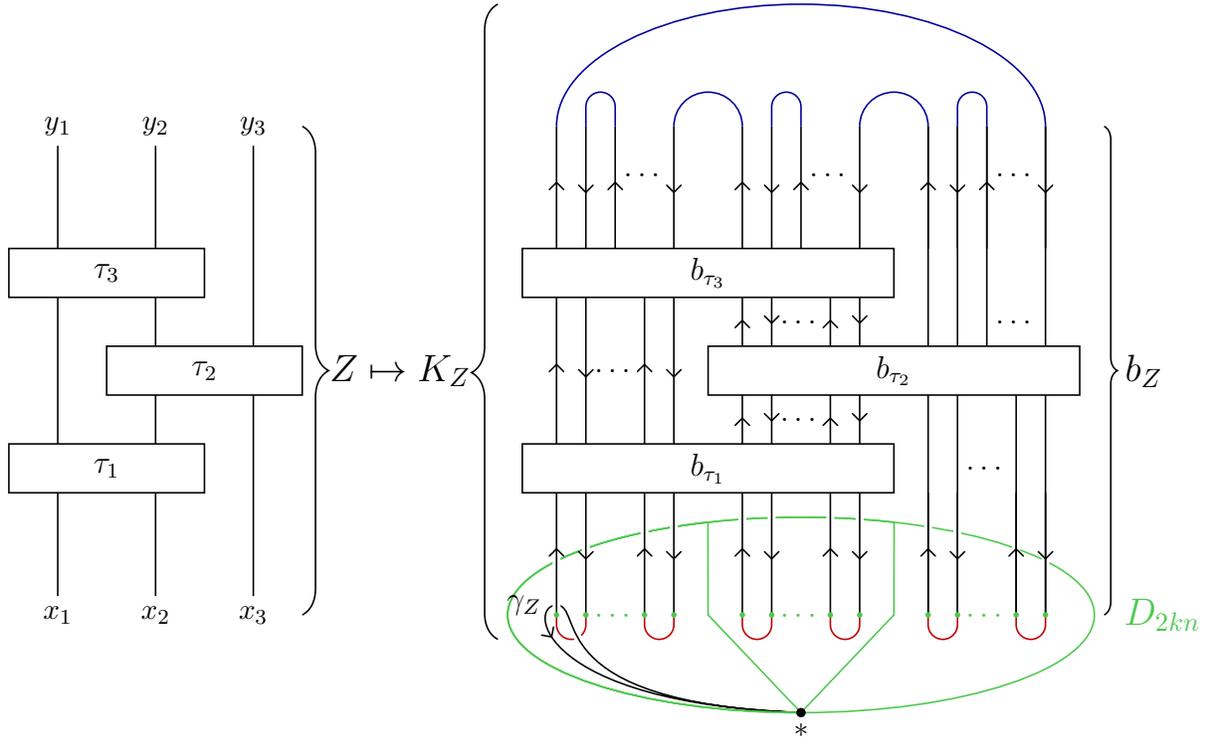
\begin{figure*}[t]
\begin{center}
\begin{tikzpicture}[scale = 0.13,semithick,decoration={markings,
    mark=at position 0.55 with {\arrow{angle 90}}}]
\coordinate (*) at (0,0);

\draw[darkred] (3,10) -- (3,9) arc [radius = 1.5, start angle = 180, end angle = 360] -- (6,10);
\draw[darkred] (-3,10) -- (-3,9) arc [radius = 1.5, start angle = 0, end angle = -180] -- (-6,10);
\draw[darkred] (3*7+1,10) -- (3*7+1,9) arc [radius = 1.5, start angle = 180, end angle = 360] -- (3*8+1,10);
\draw[darkred] (3*4+1,10) -- (3*4+1,9) arc [radius = 1.5, start angle = 180, end angle = 360] -- (3*5+1,10);
\draw[darkred] (-3*7-1,10) -- (-3*7-1,9) arc [radius = 1.5, start angle = 0, end angle = -180] -- (-3*8-1,10);
\draw[darkred] (-3*4-1,10) -- (-3*4-1,9) arc [radius = 1.5, start angle = 0, end angle = -180] -- (-3*5-1,10);

\draw[white, line width = 3] (*) .. controls (-27,1) and (-3*8+2,10+1) .. (-3*8-1,10+1) .. controls (-3*8-4,10+1) and (-27,1) .. (*);
\draw[postaction={decorate}] (*) .. controls (-27,1) and (-3*8+2,10+1) .. (-3*8-1,10+1) .. controls (-3*8-4,10+1) and (-27,1) .. (*);

\draw[medgreen] (*) arc [x radius = 30, y radius = 10, start angle = 270, end angle = -90];
\draw[medgreen] (*) arc [x radius = 30, y radius = 10, start angle = 270, end angle = 108.5] -- (-9.5,10) -- (*);
\draw[medgreen] (*) arc [x radius = 30, y radius = 10, start angle = 270, end angle = 71.5] -- (9.5,10) -- (*);
\node[right,medgreen] at (32,10) {\Large $D_{2kn}$};
\draw[fill] (*) circle [radius = 0.4];
\node[below] at (*) {*};

\foreach \i in {-2,1}
	{
	\coordinate (p) at (3*\i,10);
	\draw[white, line width = 3] (p) -- ($(p) + (0,12.5)$);
	\draw[postaction={decorate}] (p) -- ($(p) + (0,12.5)$);
	\draw[postaction={decorate}] ($(p) + (0,17.5)$) -- ($(p) + (0,22.5)$);
	\draw[postaction={decorate}] ($(p) + (0,27.5)$) -- ($(p) + (0,32.5)$);		
	\draw[fill,medgreen] (p) circle [radius = 0.2];
	}

\foreach \i in {-1,2}
	{
	\coordinate (p) at (3*\i,10);
	\draw[white, line width = 3] (p) -- ($(p) + (0,12.5)$);
	\draw[postaction={decorate}] ($(p) + (0,12.5)$) -- (p);
	\draw[postaction={decorate}] ($(p) + (0,22.5)$) -- ($(p)+(0,17.5)$);
	\draw[postaction={decorate}] ($(p) + (0,32.5)$) -- ($(p)+(0,27.5)$);	
	\draw[fill,medgreen] (p) circle [radius = 0.2];
	}
	
\foreach \i in {5,8}
	{
	\coordinate (q) at (-3*\i-1,10);
	\draw[white,line width = 3] (q) -- ($(q) + (0,12.5)$);
	\draw[postaction={decorate}] (q) -- ($(q) + (0,12.5)$);
	\draw[postaction={decorate}] ($(q)+(0,17.5)$) -- ($(q) + (0,32.5)$);	
	\draw[fill,medgreen] (q) circle [radius = 0.2];	
	}
\foreach \i in {4,7}
	{
	\coordinate (q) at (-3*\i-1,10);
	\draw[white,line width = 3] (q) -- ($(q) + (0,12.5)$);
	\draw[postaction={decorate}] ($(q) + (0,12.5)$)--(q);
	\draw[postaction={decorate}] ($(q) + (0,32.5)$) -- ($(q)+(0,17.5)$);
	\draw[fill,medgreen] (q) circle [radius = 0.2];	
	}

\foreach \i in {4,7}
	{
	\coordinate (p) at (3*\i+1,10);
	\draw[white,line width = 3] (p) -- ($(p) + (0,35)$);	
	\draw[postaction={decorate}] (p) -- ($(p) + (0,12.5)$);	
	\draw (p) -- ($(p) + (0,22.5)$);		
	\draw[fill,medgreen] (p) circle [radius = 0.2];
	}
\foreach \i in {5,8}
	{
	\coordinate (q) at (3*\i+1,10);
	\draw[white,line width = 3] (q) -- ($(q) + (0,12.5)$);
	\draw[postaction={decorate}] ($(q) + (0,12.5)$) --(q);
	\draw (q) -- ($(q) + (0,22.5)$);			
	\draw[fill,medgreen] (q) circle [radius = 0.2];	
	}
	
\node[medgreen] at (0,10) {\large  \dots};
\node[medgreen] at (3*6+1,10) {\large \dots};
\node[medgreen] at (-3*6-1,10) {\large \dots};
\node[left] at (-3*8-1.5,10+1) {\large $\gamma_Z$};
\node at (-3*6-1,10+12.5+5+5+2.5) {\large \dots};
\node at (3*6+1,12.5+5+5+2.5) {\large \dots};
\node at (0,12.5+5+5+2.5+5) {\large \dots};
\node at (0,12.5+5+5+2.5+5+10) {\large \dots};

\foreach \i in {0,-2}
	{
	\coordinate (q) at (3*\i,60);
	\draw[postaction={decorate}] ($(q) - (0,12.5)$) -- (q);	
	}
\foreach \i in {-1,2}
	{
	\coordinate (q) at (3*\i,60);
	\draw[postaction={decorate}] (q) -- ($(q) - (0,12.5)$);	
	}

\node at (3,55) {\large \dots};

\foreach \i in {4,5,7,8}
	{
	\coordinate (q) at (3*\i+1,60);
	\draw[white, line width = 2] (q) -- ($(q) - (0,12.5+10)$);	
	}

\foreach \i in {4,5,6,8}
	{
	\coordinate (q) at (3*\i+1,60);
	\draw (q) -- ($(q) - (0,12.5+10)$);	
	}
\foreach \i in {5,8}
	{
	\coordinate (q) at (3*\i+1,60);
	\draw[postaction={decorate}] (q) -- ($(q) - (0,12.5)$);	
	}	
\foreach \i in {4,6}
	{
	\coordinate (q) at (3*\i+1,60);
	\draw[postaction={decorate}] ($(q) - (0,12.5)$) -- (q);	
	}	
	
\foreach \i in {4,7}
	{
	\coordinate (p) at (-3*\i-1,60);	
	\draw[postaction={decorate}] (p) -- ($(p) - (0,12.5)$);
	}
\foreach \i in {6,8}
	{
	\coordinate (p) at (-3*\i-1,60);	
	\draw[postaction={decorate}] ($(p) - (0,12.5)$) -- (p);
	}

\node at (3*7+1,40) {\large  \dots};	
\node at (3*7+1,55) {\large \dots};
\node at (-3*5-1,55) {\large \dots};

\draw[fill=white] (-28.5,22.5) rectangle (9.5,27.5);
\node at ($ (-28.5,22.5)!0.5!(9.5,27.5) $) {\large $b_{\tau_1}$};
\draw[fill=white] (28.5,32.5) rectangle (-9.5,37.5);
\node at ($ (28.5,32.5)!0.5!(-9.5,37.5) $) {\large $b_{\tau_2}$};
\draw[fill=white] (-28.5,42.5) rectangle (9.5,47.5);
\node at ($ (-28.5,42.5)!0.5!(9.5,47.5) $) {\large $b_{\tau_3}$};

\draw[darkblue] (-25,60) arc [x radius=25, y radius = 12.5, start angle = 180, end angle = 0];
\draw[darkblue] (-22,60) -- (-22,62) arc [radius = 1.5, start angle = 180, end angle = 0] -- (-19,60);
\draw[darkblue] (-13,60) arc [radius = 3.5, start angle = 180, end angle = 0];
\draw[darkblue] (-3,60) -- (-3,62) arc [radius = 1.5, start angle = 180, end angle = 0] -- (0,60);
\draw[darkblue] (13,60) arc [radius = 3.5, start angle = 0, end angle = 180];
\draw[darkblue] (16,60) -- (16,62) arc [radius = 1.5, start angle = 180, end angle = 0] -- (19,60);

\draw [decorate,decoration={brace,amplitude=5pt}] (31,60) -- (31,10) node [midway,xshift=15] {\Large $b_Z$};
\draw[mybrace=0.42] (-31,7.5) -- (-31,72.5) ;
\draw [decorate,decoration={brace,amplitude=10pt}] (-51,60) -- (-51,10) node [midway,xshift=37] {\Large  $Z \mapsto  K_Z$};

\node (x3) at (-56,10) {$x_3$};
\node (x2) at (-66,10) {$x_2$};
\node (x1) at (-76,10) {$x_1$};
\node (y3) at (-56,60) {$y_3$};
\node (y2) at (-66,60) {$y_2$};
\node (y1) at (-76,60) {$y_1$};
\draw (x1) -- (y1);
\draw (x2) -- (y2);
\draw (x3) -- (y3);
\draw[fill=white] ($(x1)+(-5,12.5)$) rectangle ($(x2)+(5,17.5)$);
\node at ($(-81,22.5)!0.5!(-61,27.5)$) {\large $\tau_1$};
\draw[fill=white] ($(x2)+(-5,22.5)$) rectangle ($(x3)+(5,27.5)$);
\node at ($ (-71,32.5)!0.5!(-51,37.5) $) {\large $\tau_2$};
\draw[fill=white] ($(x1)+(-5,32.5)$) rectangle ($(x2)+(5,37.5)$);
\node at ($(-81,42.5)!0.5!(-61,47.5)$) {\large $\tau_3$};
\end{tikzpicture}
\caption{The reduction takes the circuit $Z$ to the knot $K_Z$ and meridian $\gamma_Z \in \pi_1(S^3\setminus K)$.  The labels $\tau_1,\tau_2$ and $\tau_3$ denote gates in $\Rub_J(A^2)$, and $b_{\tau_1}, b_{\tau_2}$ and $b_{\tau_3}$ denote the pure braid gadgets in $PB_{4k}$ simulating the respective gates.  The green punctured disk $D_{2kn}$ is included to indicate how $K_Z$ and $\gamma_Z$ are constructed, but it is not part of the reduction's output.}
\label{f:reduction}
\end{center}
\end{figure*}

The inclusion of the disk $(D_{2kn},*)$ into the knot complement $(S^3 \setminus K_Z,*)$ induces a surjection on fundamental groups, so we specify homomorphisms $f: \pi_1(S^3 \setminus K) \to G$ by listing the image in $G$ of each of the generators $\gamma_{j,i} \in \pi_1(D_{2kn})$, $j=1,\dots,2k$, $i=1,\dots,n$.  Let
\[ f_i = (f(\gamma_{1,i}),f(\gamma_{2,i}),\dots,f(\gamma_{2k,i})) \in G^{\times 2k}, \]
which we in turn identify with the homomorphism
\[ f_i: \pi_1(D_{2k,i}) \to \pi_1(S^3 \setminus K) \xrightarrow{f} G. \]
In particular, $\gamma_Z$ is just the loop $\gamma_{1,1}$ around the leftmost puncture of $D_{2kn}$ in $\pi_1(S^3\setminus K_Z)$, which is a Wirtinger generator of the knot group. 

Of course, not every homomorphism $\pi_1(D_{2kn}) \to G$ yields an element of $H(K_Z,\gamma_Z,G,c)$.  We show that those that do necessarily come from solutions to $Z$.

\begin{lemma}
Let $Z$ be an instance of $\ZSAT_{J,A,I,F}$ and let $\#Z$ denote the number of solutions to $Z$.  Then the diagram $K_Z$ and meridian $\gamma_Z$ have the following properties:
\begin{enumerate}
\item $K_Z$ is a knot.
\item If $H \lneq G$ and $H \neq \langle c \rangle$, then $\#Q(K_Z,\gamma_Z,H,c) =0$.
\item $\#H(K_Z,\gamma_Z,G,c) = \#Z$.
\end{enumerate}
\end{lemma}

\begin{proof}
Our gadgets $b_\tau$, $\tau \in \Rub_J(A^2)$, are pure braids, so our choice of plats in Figure \ref{f:reduction} guarantees that $K_Z$ is a knot, and not a link with multiple components.  This proves the first property.

Regarding the second property, let $H$ be a proper subgroup of $G$, and let $f=(f_1,\dots,f_n)$ be a surjective homomorphism $\pi_1(S^3\setminus K_Z) \to H$ taking $\gamma_Z$ to $c$.  Since $K_Z$ is a knot and $f(\gamma_Z) = f(\gamma_{1,1}) = c$, we know that for all $i=1,\dots,n$,
\[ f_i \in \hat{R}_v \]
and moreover, the image of $f$ is $H \lneq G$, so
\[ f_i \notin R_v. \]
Better yet,
\[ (f_i,f_{i+1}) \in \hat{R}_{v\#v} \setminus R_{v\#v} \]
for all $i=1,\dots,n-1$.  By construction, every braid gadget $b_\tau \in PB_{4k}$ acts trivially on $(f_i,f_{i+1})$.  We conclude that the braid $b_Z$ encoding the circuit $Z$ acts trivially on $f$:
\[ b_Z \cdot f = f. \]
Combining this identify with our choice of plats, it is easy to check that $f$ must have cyclic image generated by $c$, which shows $\#Q(K_Z,\gamma_Z,H,c) = 0$ if $H \neq \langle c \rangle$.

We prove the third property by exhibiting a bijection between solutions to $Z$ and elements of $H(K_Z,\gamma_Z,G,c)$.  Suppose $f = (f_1,\dots,f_n)$ is a solution to $Z$.  Then, by definition,
\[ \begin{cases} (f_1,\dots,f_n) \in (I \cup \{z\})^n \\
Z(f) = (g_1,\dots,g_n) \in (F \cup \{z\})^n. \end{cases} \]
By our construction of the braid gadgets $b_{\tau}$, $b_Z$ acts on $A^n$ exactly as $Z$ does.  Of course, besides representing an abstract description of some input to $Z$, $f$ is also a homomorphism $f: \pi_1(D^2_{nk},*) \to G$.  The definition of the initialization subalphabet $I\cup\{z\}$ immediately implies that this homomorphism factors through the plat attached to the bottom of $b_Z$.  

Similarly, the definition of the finalization subalphabet $F\cup\{z\}$ implies $b_Z\cdot f: \pi_1(D^2_{nk}) \to G$ factors through the plat attached to the top of $b_Z$, although this requires us to use the fact that every element of $A$ sends the leftmost puncture of $D_{2k}$ to $c$ and the rightmost puncture to $c^{-1}$.  Indeed, \emph{a priori}, the top plats used in Figure \ref{f:reduction} constrain $f$ so that
\[ b_Z \cdot f(\gamma_{j,i}) = b_Z \cdot f(\gamma_{j+1,i})^{-1} \]
for all $i$ and all $1<j<2k$,
\begin{equation}
b_Z \cdot f(\gamma_{2k,i}) = b_Z \cdot f(\gamma_{1,i+1})^{-1}
\label{l:neighbors}
\end{equation}
for all $1 \leq i< n$, and
\begin{equation}
b_Z \cdot f(\gamma_{2k,n}) = b_Z \cdot f(\gamma_{1,1})^{-1}.
\label{l:ends}
\end{equation}
However, $f \in A^n$, so our choice of braid gadgets guarantees
\[ b_Z \cdot f(\gamma_{1,i}) = b_Z \cdot f(\gamma_{2n,i})^{-1} = c \]
for all $i$.  In particular, equations \ref{l:neighbors} and \ref{l:ends} are trivially satisfied.  See Figure \ref{f:surgery}.

\begin{figure*}[t]
\begin{center}
\begin{tikzpicture}[scale = 0.045, semithick, decoration={markings,
    mark=at position 0.43 with {\arrow{angle 90}}}]
\coordinate (*) at (38*2,-5);

\foreach \i in {1}
	{
	\coordinate (l) at ($ (8*\i,20) - (2.5,0) $);
	\coordinate (r) at ($ (8*\i,20) + (2.5,0) $);
	\coordinate (u) at ($ (8*\i,20) + (0,2.5) $);	
	\coordinate (ls) at ($ (l) - 2.5*(0,1) $);
	\coordinate (ln) at ($ (l) + 2.5*(0,1) $);	
	\coordinate (ue) at ($ (u) + 1*(1,0) $);		
	\coordinate (uw) at ($ (u) - 1*(1,0) $);
	\coordinate (rn) at ($ (r) + 2.5*(0,1) $);
	\coordinate (rs) at ($ (r) - 2.5*(0,1) $);
	\coordinate (v) at ($ (*)!.8!-5*\i:(0,-5) $);
	\draw[medgreen,postaction={decorate}] (*) .. controls (v) and (rs) .. (r) .. controls (rn) and (ue) .. (u) .. controls (uw) and (ln) .. (l) .. controls (ls) and (v) .. (*);
	}

\foreach \i in {6,7}
	{
	\coordinate (l) at ($ (8*\i,20) - (2.5,0) $);
	\coordinate (r) at ($ (8*\i,20) + (2.5,0) $);
	\coordinate (u) at ($ (8*\i,20) + (0,2.5) $);	
	\coordinate (ls) at ($ (l) - 2.5*(0,1) $);
	\coordinate (ln) at ($ (l) + 2.5*(0,1) $);	
	\coordinate (ue) at ($ (u) + 1*(1,0) $);		
	\coordinate (uw) at ($ (u) - 1*(1,0) $);
	\coordinate (rn) at ($ (r) + 2.5*(0,1) $);
	\coordinate (rs) at ($ (r) - 2.5*(0,1) $);
	\coordinate (v) at ($ (*)!.3!-5*\i:(0,-5) $);
	\draw[medgreen,postaction={decorate}] (*) .. controls (v) and (rs) .. (r) .. controls (rn) and (ue) .. (u) .. controls (uw) and (ln) .. (l) .. controls (ls) and (v) .. (*);
	}
	
\foreach \i in {12,13}
	{
	\coordinate (l) at ($ (8*\i,20) - (2.5,0) $);
	\coordinate (r) at ($ (8*\i,20) + (2.5,0) $);
	\coordinate (u) at ($ (8*\i,20) + (0,2.5) $);	
	\coordinate (ls) at ($ (l) - 2.5*(0,1) $);
	\coordinate (ln) at ($ (l) + 2.5*(0,1) $);	
	\coordinate (ue) at ($ (u) + 1*(1,0) $);		
	\coordinate (uw) at ($ (u) - 1*(1,0) $);
	\coordinate (rn) at ($ (r) + 2.5*(0,1) $);
	\coordinate (rs) at ($ (r) - 2.5*(0,1) $);
	\coordinate (v) at ($ (*)!.3!5*(19-\i):(38*4,-5) $);
	\draw[medgreen,postaction={decorate}] (*) .. controls (v) and (rs) .. (r) .. controls (rn) and (ue) .. (u) .. controls (uw) and (ln) .. (l) .. controls (ls) and (v) .. (*);
	}

\foreach \i in {18}
	{
	\coordinate (l) at ($ (8*\i,20) - (2.5,0) $);
	\coordinate (r) at ($ (8*\i,20) + (2.5,0) $);
	\coordinate (u) at ($ (8*\i,20) + (0,2.5) $);	
	\coordinate (ls) at ($ (l) - 2.5*(0,1) $);
	\coordinate (ln) at ($ (l) + 2.5*(0,1) $);	
	\coordinate (ue) at ($ (u) + 1*(1,0) $);		
	\coordinate (uw) at ($ (u) - 1*(1,0) $);
	\coordinate (rn) at ($ (r) + 2.5*(0,1) $);
	\coordinate (rs) at ($ (r) - 2.5*(0,1) $);
	\coordinate (v) at ($ (*)!.8!5*(19-\i):(38*4,-5) $);
	\draw[medgreen,postaction={decorate}] (*) .. controls (v) and (rs) .. (r) .. controls (rn) and (ue) .. (u) .. controls (uw) and (ln) .. (l) .. controls (ls) and (v) .. (*);
	}
	
\draw (*) arc [x radius = 40*2, y radius = 25, start angle = 270, end angle = -90];	

\foreach \i in {1} {\node[left,medgreen] at (8*\i-2,20) {$c$};}	
\foreach \i in {7,13} {\node[above right,medgreen] at (8*\i,20+1) {$c$};}
\foreach \i in {6} {\node[below left,medgreen] at (8*\i+1,20) {$c^{-1}$};}
\foreach \i in {12} {\node[below left,medgreen] at (8*\i-1,20) {$c^{-1}$};}
\foreach \i in {18} {\node[below left,medgreen] at (8*\i-4,20) {$c^{-1}$};}	

\foreach \i in {2,4,8,10,14,16}
	{
	\draw[white, line width = 4] (8*\i,20) -- + (0,2);
	\draw[white, line width = 4] (8*\i+8,20) -- + (0,2);	
	\draw[white, line width = 4] (8*\i,22) arc [radius=4, start angle = 180, end angle = 0];
	}
	
\foreach \i in {1,7,13}
	{
	\draw[white, line width = 4] (8*\i,20) -- + (0,2);
	\draw[white, line width = 4] (8*\i+40,20) -- + (0,2);	
	\draw[white, line width = 4] (8*\i,22) arc [radius=20, start angle = 180, end angle = 0];
	}

\foreach \i in {1,2,...,18}
	{
	\draw[fill] (8*\i,20) circle [radius=0.4];
	}

\foreach \i in {2,4,8,10,14,16}
	{
	\draw[darkblue] (8*\i,20) -- + (0,2);
	\draw[darkblue] (8*\i+8,20) -- + (0,2);	
	\draw[darkblue] (8*\i,22) arc [radius=4, start angle = 180, end angle = 0];
	}
	
\foreach \i in {1,7,13}
	{
	\draw[darkblue] (8*\i,20) -- + (0,2);
	\draw[darkblue] (8*\i+40,20) -- + (0,2);	
	\draw[darkblue] (8*\i,22) arc [radius=20, start angle = 180, end angle = 0];
	}

\foreach \i in {1,2,...,18}
	{
	\draw[fill] (8*\i,20) circle [radius=0.4];
	}	
	
\draw[fill] (*) circle [radius=0.4];
\node[below] at (*) {$*$};
\end{tikzpicture}
\ \ \ \ \
\begin{tikzpicture}[scale = 0.045, semithick, decoration={markings,
    mark=at position 0.43 with {\arrow{angle 90}}}]
\coordinate (*) at (38*2,-5);

\foreach \i in {1}
	{
	\coordinate (l) at ($ (8*\i,20) - (2.5,0) $);
	\coordinate (r) at ($ (8*\i,20) + (2.5,0) $);
	\coordinate (u) at ($ (8*\i,20) + (0,2.5) $);	
	\coordinate (ls) at ($ (l) - 2.5*(0,1) $);
	\coordinate (ln) at ($ (l) + 2.5*(0,1) $);	
	\coordinate (ue) at ($ (u) + 1*(1,0) $);		
	\coordinate (uw) at ($ (u) - 1*(1,0) $);
	\coordinate (rn) at ($ (r) + 2.5*(0,1) $);
	\coordinate (rs) at ($ (r) - 2.5*(0,1) $);
	\coordinate (v) at ($ (*)!.8!-5*\i:(0,-5) $);
	\draw[medgreen,postaction={decorate}] (*) .. controls (v) and (rs) .. (r) .. controls (rn) and (ue) .. (u) .. controls (uw) and (ln) .. (l) .. controls (ls) and (v) .. (*);
	}

\foreach \i in {6,7}
	{
	\coordinate (l) at ($ (8*\i,20) - (2.5,0) $);
	\coordinate (r) at ($ (8*\i,20) + (2.5,0) $);
	\coordinate (u) at ($ (8*\i,20) + (0,2.5) $);	
	\coordinate (ls) at ($ (l) - 2.5*(0,1) $);
	\coordinate (ln) at ($ (l) + 2.5*(0,1) $);	
	\coordinate (ue) at ($ (u) + 1*(1,0) $);		
	\coordinate (uw) at ($ (u) - 1*(1,0) $);
	\coordinate (rn) at ($ (r) + 2.5*(0,1) $);
	\coordinate (rs) at ($ (r) - 2.5*(0,1) $);
	\coordinate (v) at ($ (*)!.3!-5*\i:(0,-5) $);
	\draw[medgreen,postaction={decorate}] (*) .. controls (v) and (rs) .. (r) .. controls (rn) and (ue) .. (u) .. controls (uw) and (ln) .. (l) .. controls (ls) and (v) .. (*);
	}
	
\foreach \i in {12,13}
	{
	\coordinate (l) at ($ (8*\i,20) - (2.5,0) $);
	\coordinate (r) at ($ (8*\i,20) + (2.5,0) $);
	\coordinate (u) at ($ (8*\i,20) + (0,2.5) $);	
	\coordinate (ls) at ($ (l) - 2.5*(0,1) $);
	\coordinate (ln) at ($ (l) + 2.5*(0,1) $);	
	\coordinate (ue) at ($ (u) + 1*(1,0) $);		
	\coordinate (uw) at ($ (u) - 1*(1,0) $);
	\coordinate (rn) at ($ (r) + 2.5*(0,1) $);
	\coordinate (rs) at ($ (r) - 2.5*(0,1) $);
	\coordinate (v) at ($ (*)!.3!5*(19-\i):(38*4,-5) $);
	\draw[medgreen,postaction={decorate}] (*) .. controls (v) and (rs) .. (r) .. controls (rn) and (ue) .. (u) .. controls (uw) and (ln) .. (l) .. controls (ls) and (v) .. (*);
	}

\foreach \i in {18}
	{
	\coordinate (l) at ($ (8*\i,20) - (2.5,0) $);
	\coordinate (r) at ($ (8*\i,20) + (2.5,0) $);
	\coordinate (u) at ($ (8*\i,20) + (0,2.5) $);	
	\coordinate (ls) at ($ (l) - 2.5*(0,1) $);
	\coordinate (ln) at ($ (l) + 2.5*(0,1) $);	
	\coordinate (ue) at ($ (u) + 1*(1,0) $);		
	\coordinate (uw) at ($ (u) - 1*(1,0) $);
	\coordinate (rn) at ($ (r) + 2.5*(0,1) $);
	\coordinate (rs) at ($ (r) - 2.5*(0,1) $);
	\coordinate (v) at ($ (*)!.8!5*(19-\i):(38*4,-5) $);
	\draw[medgreen,postaction={decorate}] (*) .. controls (v) and (rs) .. (r) .. controls (rn) and (ue) .. (u) .. controls (uw) and (ln) .. (l) .. controls (ls) and (v) .. (*);
	}

\draw (*) arc [x radius = 40*2, y radius = 25, start angle = 270, end angle = -90];

\draw[white, line width = 4] (8,20) -- ++(0,2);
\draw[white, line width = 4] (8+8*17,20) -- ++(0,2);
\draw[white, line width = 4] (8,22) arc [x radius=34*2, y radius = 40,  start angle = 180, end angle = 0];

\foreach \i in {2,4,...,16}
	{
	\draw[white,line width = 4] (8*\i,20) -- + (0,2);
	\draw[white,line width = 4] (8*\i+8,20) -- + (0,2);	
	\draw[white,line width = 4] (8*\i,22) arc [radius=4, start angle = 180, end angle = 0];
	}

\foreach \i in {1,2,...,18}
	{
	\draw[fill] (8*\i,20) circle [radius=0.4];
	}

\draw[darkblue] (8,20) -- +(0,2);
\draw[darkblue] (8+8*17,20) -- +(0,2);
\draw[darkblue] (8,22) arc [x radius=34*2, y radius = 40,  start angle = 180, end angle = 0];

\foreach \i in {2,4,...,16}
	{
	\draw[darkblue] (8*\i,20) -- + (0,2);
	\draw[darkblue] (8*\i+8,20) -- + (0,2);	
	\draw[darkblue] (8*\i,22) arc [radius=4, start angle = 180, end angle = 0];
	}

\foreach \i in {1,2,...,18}
	{
	\draw[fill] (8*\i,20) circle [radius=0.4];
	}	
	
\draw[fill] (*) circle [radius=0.4];
\node[below] at (*) {$*$};
\foreach \i in {1} {\node[left,medgreen] at (8*\i-2,20) {$c$};}
\foreach \i in {7,13} {\node[above right,medgreen] at (8*\i,20+2) {$c$};}
\foreach \i in {6,12} {\node[above left,medgreen] at (8*\i+7,20+2) {$c^{-1}$};}
\foreach \i in {18} {\node[below left,medgreen] at (8*\i-4,20) {$c^{-1}$};}
\end{tikzpicture}
\caption{If $f_i(\gamma_{1,i}) = f_i(\gamma_{2n,i})^{-1} = c$ for all $i$, then $f = (f_1,\dots,f_n)$ either factors through both sets of plats, or it factors through neither set of plats.}
\label{f:surgery}
\end{center}
\end{figure*}
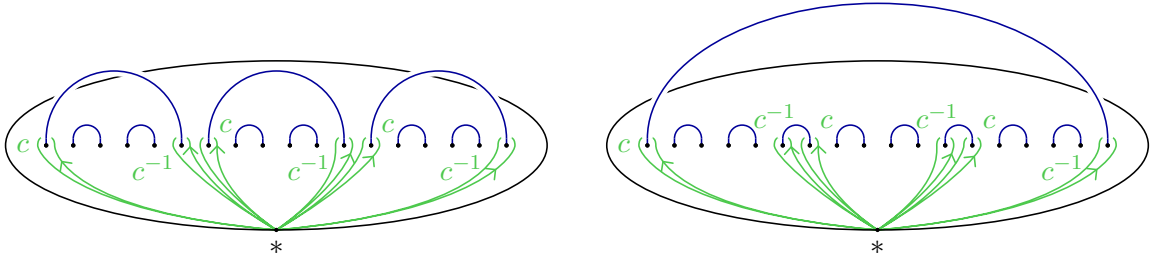

We conclude that because $f$ is a solution to $Z$, the homomorphism $f: \pi_1(D_{2nk},*) \to G$ factors through a homomorphism $\pi_1(S^3 \setminus K_Z,*) \to G$.  Moreover, $f(\gamma_K)=c$, hence every solution to $Z$ yields a unique element of $H(K_Z,\gamma_Z,G,c)$.

Conversely, suppose $f \in H(K_Z,\gamma_Z,G,c)$.  We must show that $f$ is not spurious, \ie, we must show that $f \in A^n$.  We begin with some useful observations.  Each $f_i$ factors through the initialization plat attached to $D_{2k,i}$, so Lemma \ref{l:props} shows $\inv_v(f_i) = 0$ for all $i=1,\dots,n$.  Thus, $f_i \in \hat{R}_v^0$.  Moreover, Lemma \ref{l:preserve} says the action of the braid gadgets preserves this condition.

Suppose $f_i$ is not in $\Aut(G,C) \cdot A$ and let $f_{i-1}' \in \hat{R}_v^0$.  Then it is straightforward to check that for every braid gadget $b_\tau$,
\[ b_\tau \cdot (f_{i-1}',f_i) = (f_{i-1}',f_i).\]
Similarly,
\[ b_\tau \cdot (f_i,f_{i+1}') = (f_i,f_{i+1}')\]
for all $f_{i+1}' \in \hat{R}_v^0$.  Write
\[ b_Z \cdot (f_1,\dots,f_n) = (g_1,\dots,g_n). \]
Then the above shows $g_i = f_i$.  But then there must be a $d \in C$ such that
\[ f_i = (d,d^{-1},d,d^{-1},\dots,d,d^{-1}), \]
hence $f_i \in \Inn(G) \cdot \{z\} \subset \Aut(G,C) \cdot A$, a contradiction.  Thus, $f_i \in \Aut(G,C) \cdot A$ for all $i$ and $f \in (\Aut(G,C) \cdot A)^n$.

We now show $f \in \Aut(G,C) \cdot A^n$.  If not, then there is an $1\leq i < n$ such that $f_i$ and $f_{i+1}$ are \emph{unaligned}, meaning $(f_i,f_{i+1})$ is in
\[ \Aut(G,C)\cdot A \times \Aut(G,C) \cdot A \setminus \Aut(G,C) \cdot (A \times A). \]
Note that $f_i, f_{i+1} \in \Aut(G,C)\cdot A$ are unaligned precisely when
\[ b_Z\cdot f_i(\gamma_{2k,i}) \neq b_Z \cdot f_{i+1}(\gamma_{1,i+1})^{-1}. \]
Thus, our choice of finalization plats guarantees that unaligned pairs can not be a part of any homomorphism in $H(K_Z,\gamma_Z,G,c)$.

Finally, because $f \in \Aut(G,C) \cdot A^n$ and $f(\gamma_Z) = c$, we conclude that $f \in A^n$.  
\end{proof}

The construction of $(K_Z,\gamma_Z)$ from $Z$ is linear in both time and space as a function of the size of $Z$.  Thus $Z \mapsto (K_Z,\gamma_Z)$ is a strictly parsimonious Levin reduction from $\ZSAT_{J,A,I,F}$ to $\#H(-,G,c)$.  Since $\ZSAT_{J,A,I,F}$ is almost parsimoniously $\shP$-hard, $\#H(-,G,c)$ is too. 

To see that $\#Q(-,G,c)$ is strictly parsimoniously $\shP$-hard via Levin reduction, note that when we postcompose the reduction from $\CSAT$ to $\ZSAT$ that serves as the proof of \Thm{l:zsat} with the reduction from $\ZSAT$ to $\#H(-,G,c)$ just constructed, the only non-surjective element of $H(K_Z,\gamma_Z,G,c)$ is the unique homomorphism $\pi_1(S^3\setminus K_Z,*) \to G$ with cyclic image $\langle c \rangle$ such that $\gamma_K \mapsto c$. \qed

\chapter[Discussion]{Discussion and further directions}
\label{ch:discussion}
\section{Sharper hardness}
\label{ss:sharper}
\subsection{Controlling stabilizations}
Even though the proofs of Theorems \ref{th:main1} and \ref{th:main2} are polynomially efficient reductions,
for any fixed, suitable target group $G$ and conjugacy class $C \subset G$, they are not otherwise particularly
efficient.   Various steps of the proofs require either the genus $g$ (which is used
to define the symbol alphabet $R_g^0$) or the number of punctures $k$ (used to define $R_v^0$) to be sufficiently large.  In fact, the crucial Theorem \ref{th:dt} and \ref{th:rv} do not even provide constructive lower bounds
on $g$.  Dunfield and Thurston \cite{DT:random} discuss possibilities
to improve the bound on $g$, and they conjecture that $g \ge 3$ suffices in
\Thm{th:dt} for many or possibly all choices of $G$.  We likewise believe
that there is some universal genus $g_0$ such that \Thm{th:dtrefine} holds
for all $g \ge g_0$.

In any case, the chains of reductions summarized in \Fig{f:reductions}
is not very efficient either.  What we really believe is that the random
3-manifold model of Dunfield and Thurston also yields computational hardness.
More precisely, Johnson showed that the Torelli group $\Tor(\Sigma_g)$
is finitely generated for $g \ge 3$ \cite{Johnson:finite}.  This yields
a model for generating a random homology 3-sphere:  We choose $\phi \in
\Tor(\Sigma_g)$ by evaluating a word of length $\ell$ in the Johnson
generators, and then we let
\[ M \defeq (H_g)_I \sqcup_\phi (H_g)_F. \]
Our \Thm{th:dtrefine} implies that \cite[Thm.~7.1]{DT:random} holds in this
model, \ie, that the distribution of $\#Q(M,G)$ converges to Poisson with
mean $|H_2(G)|/|\Out(G)|$ if we first send $\ell \to \infty$ and then send
$g \to \infty$.  We also conjecture that $\#Q(M,G)$ is hard on average in
the sense of average-case computational complexity \cite[Ch.~18]{AB:modern}
if $\ell$ grows faster than $g$.

Speaking non-rigorously, we conjecture that it is practical to randomly
generate triangulated homology 3-spheres $M$ in such a way that no one
will ever know the value of $\#Q(M,G)$, say for $G = A_5$.  Hence, no one
will ever know whether such an $M$ has a connected 5-sheeted cover.

\subsection{Varying conjugacy classes}
In the case of knots, we can consider how the invariants $Q(K,\gamma,G,c)$ and $Q(K,\gamma,G,c')$ are related when $c$ and $c'$ are elements of $G$ in distinct (outer) automorphism classes.  We expect a kind of decoupling is possible: given any pair of non-negative integers $m$ and $n$, it should be possible to construct knots $K$ so that $Q(K,\gamma,G,c)=m$ and $Q(K,\gamma,G,c')=n$.  We expect a version of Goursat's Lemma \ref{l:goursat} should hold where we replace the groups $G_1$ and $G_2$ with the braid group actions corresponding to the conjugation quandles induced by $G$ acting on $c$ and $c'$.

\subsection{Avoiding the classification of finite simple groups}
An important point we feel obliged to reiterate is that our proof of \Thm{th:rvrefine} (and, hence, Theorem \ref{th:main2}) depends on the classification of finite simple groups via the 6-transitivity trick.  Dunfield and Thurston's proof of Theorem \ref{th:dt} exploits the same 6-transitivity trick.  However, they briefly sketch a workaround to avoid the classification \cite{DT:random}.  Instead, they use a result of Gilman \cite{Gilman:quotients} and a theorem about permutation groups.  (They describe their workaround in the case of homologically trivial surjections, but the argument works more generally.)  We suppose there could be an analogous workaround for the full monodromy theorem \cite[Thm.~5.1]{RV:hurwitz}.

\section{Other spaces}
\label{ss:other}

Maher \cite{Maher:heegaard} showed that the probability that a randomly
chosen $M$ in the Dunfield-Thurston model is hyperbolic converges to 1
as $\ell \to \infty$, for any fixed $g \ge 2$.  Maher notes that the same
result holds if $M$ is a homology 3-sphere made using the Torelli group,
for any $g \ge 3$.  Thus our conjectures in \Sec{ss:sharper} would imply
that $\#Q(M,G)$ is computationally intractable when $M$ is a hyperbolic
homology 3-sphere.

We conjecture that a version of \Thm{th:main1} holds when $M$ fibers over
a circle.  In this case $M$ cannot be a homology 3-sphere, but it can be a
homology $S^2 \times S^1$.  If $M$ fibers over a circle, then the invariant
$H(M,G)$ is obviously analogous (indeed a special case of) counting solutions
to $Z(x) = x$ when $Z$ is a reversible circuit.  However, the reduction from
$Z$ to $M$ would require new techniques to avoid spurious solutions.

\section{Non-simple groups}
\label{ss:nonsimple}
We consider the invariant $\#H(M,G)$ for a general finite group $G$.

Recall that the \emph{perfect core} $G_\per$ of a group $G$ is its unique
largest perfect subgroup; if $G$ is finite, then it is also the limit of its
derived series.   If $M$ is a homology sphere, then its fundamental group
is perfect and $H(M,G) = H(M,G_\per)$.  We conjecture then that a version
of \Thm{th:main1} holds for any finite, perfect group $G$.  More precisely,
we conjecture that \Thm{th:main1} holds for $Q(M,G)$ when $G$ is finite
and perfect, and that the rest of $H(M,G)$ is explained by non-surjective
homomorphisms $f:G \to G$. Mochon's analysis \cite{Mochon:finite} in the
case when $G$ is non-solvable can be viewed as a partial result towards
this conjecture.

If $G$ is finite and $G_\per$ is trivial, then this is exactly the case that
$G$ is solvable.  In the case when $M$ is a link complement, Ogburn and
Preskill \cite{OP:topological} non-rigorously conjecture that $H(M,G)$
is not ``universal" for classical computation.  It is very believable
that the relevant actions of braid groups and mapping class groups are
too rigid for any analogue of the second half of \Thm{th:dt} to hold.
Rowell \cite{Rowell:paradigms} more precisely conjectured that $\#H(M,G)$
can be computed in polynomial time for any link complement $M$ and any
finite, solvable $G$.  We are much less confident that this more precise
conjecture is true.

\backmatter
\bibliographystyle{amsplain}
\providecommand{\bysame}{\leavevmode\hbox to3em{\hrulefill}\thinspace}
\providecommand{\MR}{\relax\ifhmode\unskip\space\fi MR }
\providecommand{\MRhref}[2]{%
  \href{http://www.ams.org/mathscinet-getitem?mr=#1}{#2}
}
\providecommand{\href}[2]{#2}

\end{document}